\documentclass[numbook,envcountsame,smallcondensed,draft]{amsart}

\usepackage{amssymb,amsmath,amsfonts,mathrsfs,cite,graphicx,gastex,rotating}
\DeclareMathOperator{\con}{con}
\DeclareMathOperator{\simple}{sim}
\DeclareMathOperator{\ini}{ini}
\DeclareMathOperator{\mul}{mul}
\DeclareMathOperator{\occ}{occ}
\DeclareMathOperator{\var}{var}

\newtheorem{theorem}{Theorem}[section]
\newtheorem{proposition}[theorem]{Proposition}
\newtheorem{lemma}[theorem]{Lemma}
\newtheorem{corollary}[theorem]{Corollary}

\theoremstyle{definition}

\newtheorem{remark}[theorem]{Remark}

\newtheorem{question}{Question}

\numberwithin{equation}{section}

\makeatletter

\renewcommand*\subjclass[2][2010]{\def\@subjclass{#2}\@ifundefined{subjclassname@#1}{\ClassWarning{\@classname}{Unknown edition (#1) of Mathematics Subject Classification; using '2010'.}}{\@xp\let\@xp\subjclassname\csname subjclassname@#1\endcsname}}

\renewcommand{\subjclassname}{\textup{2010} Mathematics Subject Classification}

\makeatother

\newcommand{\FB}{FB}
\newcommand{\HFB}{HFB}
\newcommand{\NFB}{NFB}

\newcommand{\Amon}{\mathbf A} 

\begin{document}

\title[Limit varieties of monoids satisfying a certain identity]{Limit varieties of monoids satisfying a certain identity}
\thanks{The first author is supported by the Ministry of Science and Higher Education of the Russian Federation (project FEUZ-2020-0016).
The second and the third authors are partially supported by the National Natural Science Foundation of China (No. 12171213, 11771191) and the Natural Science Foundation of Gansu Province (No. 20JR5RA275).}

\author{Sergey V. Gusev}

\address{Ural Federal University, Institute of Natural Sciences and Mathematics, Lenina 51, Ekaterinburg 620000, Russia}

\email{sergey.gusb@gmail.com}

\author{Yu Xian Li}

\address{Lanzhou University, School of Mathematics and Statistics, Lanzhou 730000, People's Republic of China}

\email{liyuxian21@lzu.edu.cn}

\author{Wen Ting Zhang}

\address{Lanzhou University, School of Mathematics and Statistics, Lanzhou 730000, People's Republic of China}

\email{zhangwt@lzu.edu.cn}

\begin{abstract}
A {\it limit variety} is a variety that is minimal with respect to being non-finitely based.
Since the turn of the millennium, much attention has been given to the classification of limit varieties of aperiodic monoids.
Seven explicit examples have so far been found, and the task of locating other examples has recently been reduced to two subproblems, one of which is concerned with monoids that satisfy the identity $xsxt \approx xsxtx$.
In the present article, we provide a complete solution to this subproblem by showing that there are precisely two limit varieties that satisfy this identity.
One of them turns out to be the first example having infinitely many subvarieties.

It is also deduced that the variety generated by any monoid of order five or less contains at most countably many subvarieties.
\end{abstract}

\keywords{Monoid, variety, hereditary finitely based variety, limit variety, finite basis problem.}

\subjclass{20M07}

\maketitle

\section{Introduction}

\label{sec: introduction}

A variety of algebras is \textit{finitely based} (\FB) if it has a finite basis of identities; otherwise, it is \textit{non-finitely based} (\NFB).
The \textit{finite basis problem}---to determine which varieties are finitely based---has been the focus of investigation for many years.
In the 1930s, Neumann~\cite{Neumann-37} questioned if every variety of groups is \FB, and it was not until 1970 when the existence of {\NFB} examples was independently confirmed by Adyan~\cite{Adyan-70}, Ol'shanski\u{\i}~\cite{Olshanskii-70}, and Vaughan-Lee~\cite{Vaughan-Lee-70}.
The finite basis problem for varieties of semigroups and of monoids gained much interest in the 1960s after Perkins~\cite{Perkins-69} exhibited the first examples of {\NFB} varieties generated by a finite semigroup.
For more information on the finite basis problem for these varieties, refer to Gupta and Krasilnikov~\cite{Gupta-Krasilnikov-03} for varieties of groups and to Volkov~\cite{Volkov-01} for varieties of semigroups and of monoids.

A {\FB} variety satisfies the stronger property of being \textit{hereditary finitely based} (\HFB) if all its subvarieties are \FB.
A {\it limit variety} is a {\NFB} variety whose proper subvarieties are all \FB.
By Zorn's lemma, every {\NFB} variety contains some limit subvariety; in other words, a variety is {\HFB} if and only if it excludes all limit varieties.
Therefore, classifying {\HFB} varieties in a certain sense reduces to classifying limit varieties.
But finding explicit examples of limit varieties turns out to be highly nontrivial.
For instance, no explicit example of limit variety of groups is known so far, even though uncountably many of them exist~\cite{Kozhevnikov-12}.
Locating such an example remains one of the foremost unsolved problems in the theory of varieties of groups; see Gupta and Krasilnikov~\cite{Gupta-Krasilnikov-03}.

One of the main goals of the present paper is to study limit varieties of monoids, considered as algebras of type $\langle 2,0\rangle$.
A complete classification of all limit varieties of monoids is highly infeasible since that would include a description of all limit varieties of groups.
Therefore it is natural to focus on the class $\Amon$ of \textit{aperiodic} monoids, that is, monoids whose subgroups are all trivial.
The first explicit examples of limit subvarieties of $\Amon$ were exhibited by Jackson~\cite{Jackson-05} in the early 2000s.
Since then, limit subvarieties of $\Amon$ have received much attention and a few more examples have been found; see, for example, Gusev~\cite{Gusev-20,Gusev-21}, Gusev and Sapir~\cite{Gusev-Sapir-22}, Lee~\cite{Lee-09,Lee-12}, Sapir~\cite{Sapir-non-J-trivial,Sapir-21}, Zhang~\cite{Zhang-13}, and Zhang and Luo~\cite{Zhang-Luo}.

Presently, seven explicit examples of limit subvarieties of $\Amon$ are known, and it follows from Gusev and Sapir \cite[Sorting Lemma~2]{Gusev-Sapir-22} that any other limit subvariety of $\Amon$ is contained in one of the following varieties or their duals:
$$
\mathbf P = \var\{ xsxt \approx xsxtx \}\  \text{ and }\ \mathbf Q  = \var\{xtx \approx xtx^2, \, xy^2tx \approx (xy)^2tx\}.
$$
Therefore, to completely classify all limit subvarieties of $\Amon$, it suffices to consider only subvarieties of~$\mathbf P $ and of~$\mathbf Q$.
Sapir~\cite{Sapir-non-J-trivial} made some progress by exhibiting three new examples of limit subvarieties of~$\mathbf Q$.

The present paper is concerned with the variety~$\mathbf P$, where the main objective is to provide a complete description of all its limit subvarieties (Theorem~\ref{T: main result}).
Specifically, we present two new limit varieties of monoids and show that they exhaust all limit subvarieties of $\mathbf P$.
One of these limit varieties contains countably infinitely many subvarieties (Corollary~\ref{P: J_2 has infinitely many subvarieties})---a result that is quite surprising given that every limit variety of monoids previously found is \textit{small} in the sense that it contains only finitely many subvarieties.

Our main result also implies an important fundamental result concerning varieties generated by a monoid of small order.
Finite members from several classical classes of algebras, such as groups~\cite{Oates-Powell-64} and associative rings \cite{Kruse-73,Lvov-73}, generate small {\HFB} varieties.
But this result does not hold for finite semigroups or monoids.
The smallest monoids currently known to generate a variety with continuum many subvarieties are of order six; the well-known Brandt monoid
$$
B_2^1=\langle a,b,1 \mid aba=a,\,bab=b,\,a^2=b^2=0\rangle=\{a,b,ab,ba,0,1\}
$$
is one such example \cite{Jackson-Lee-18,Jackson-Zhang-21}.
As for monoids of order five or less, up to isomorphism and anti-isomorphism, every one of them, with the exception of
$$
P_2^1=\langle a,b,1\mid a^2=ab=a,\,b^2a=b^2\rangle=\{a,b,ba,b^2,1\},
$$ generates a small {\HFB} variety~\cite{Lee-Zhang-14}.
The variety $\mathbf P_2^1$ generated by $P_2^1$ is not small, but not much else is known about it.
An answer to the following question is thus desirable.

\begin{question}[\mdseries{Jackson and Lee~\cite[Question~6.1]{Jackson-Lee-18}}]
\label{Prob: the monoid P_2^1}
Does the variety $\mathbf P_2^1$ contain uncountable many subvarieties?
\end{question}

The monoid $P_2^1$ satisfies the identity $xsxt \approx xsxtx$ and so $\mathbf P_2^1$ is a subvariety of~$\mathbf P$.
It follows from our classification of limit subvarieties of $\mathbf P$ that $\mathbf P_2^1$ is {\HFB} (Proposition~\ref{P: P_2^1}).
Consequently, the variety generated by any monoid of order five or less is {\HFB} and so contains at most countably many subvarieties.

\begin{remark}
\label{R: P2}
It is relevant to note that $P_2^1$ was originally investigated as a semigroup by Lee, who first questioned if it generates a {\HFB} variety of semigroups; see Edmunds et al. \cite[Question~4.2]{Edmunds-Lee-Lee-10} and Lee \cite[Problem~1.4]{Lee-13}.
The answer to this question, however, remains elusive.
\end{remark}

This article consists of six sections.
Background information and some basic results are first given in Section~\ref{sec: prelim}.
Then several important results on the variety $\mathbf P$ and its subvarieties are established in Section~\ref{sec: properties of P}.
Results that are crucial to the proof of the main results are established in Section~\ref{sec: critical pairs}.
Limit subvarieties of $\mathbf P$ are then classified in Section~\ref{sec: limit varieties}, while results concerning subvarieties of $\mathbf P_2^1$ are established in Section~\ref{sec: the monoid P_2^1}.

Many identities will be introduced and used throughout this article.
For the reader's convenience, these identities are collected in the appendix for quick referencing.

\section{Preliminaries}
\label{sec: prelim}

Acquaintance with rudiments of universal algebra is assumed of the reader.
Refer to Burris and Sankappanavar~\cite{Burris-Sankappanavar-81} for more information.

\subsection{Words, identities, and deduction}

Let $\mathfrak A^\ast$ denote the free monoid over a countably infinite alphabet $\mathfrak A$.
Elements of $\mathfrak A$ are called \textit{letters} and elements of $\mathfrak A^\ast$ are called \textit{words}.
The empty word, denoted by $\lambda$, is the identity element of $\mathfrak A^\ast$.

An identity is written as $\mathbf u \approx \mathbf v$, where $\mathbf u,\mathbf v \in \mathfrak A^\ast$; it is \textit{non-trivial} if $\mathbf u \neq \mathbf v$.
An identity $\mathbf u \approx \mathbf v$ is \textit{directly deducible} from an identity $\mathbf s \approx \mathbf t$ if there exist some words $\mathbf a,\mathbf b \in \mathfrak A^\ast$ and substitution $\phi\colon \mathfrak A \to \mathfrak A^\ast$ such that $\{ \mathbf u, \mathbf v \} = \big\{ \mathbf a\phi(\mathbf s)\mathbf b,\mathbf a\phi(\mathbf t)\mathbf b \big\}$.
A non-trivial identity $\mathbf u \approx \mathbf v$ is \textit{deducible} from a set $\Sigma$ of identities if there exists some finite sequence $\mathbf u = \mathbf w_0, \mathbf w_1, \ldots, \mathbf w_m = \mathbf v$ of distinct words such that each identity $\mathbf w_i \approx \mathbf w_{i+1}$ is directly deducible from some identity in $\Sigma$.

The following assertion is a specialization for monoids of a well-known universal-algebraic fact (see Burris and Sankappanavar \cite[Theorem~II.14.19]{Burris-Sankappanavar-81}).

\begin{proposition}
\label{P: deduction}
Let $\mathbf V$ be the variety defined by some set $\Sigma$ of identities.
Then $\mathbf V$ satisfies an identity $\mathbf u \approx \mathbf v$ if and only if $\mathbf u \approx \mathbf v$ is deducible from $\Sigma$.\qed
\end{proposition}

\subsection{$k$-decomposition of a word}

Here we introduce a series of notions such as $k$-divider, $k$-block, $k$-decomposition, and some others, which appeared in Gusev and Vernikov \cite[Chapter~3]{Gusev-Vernikov-18}.

The \textit{content} of a word $\mathbf w$, denoted by $\con(\mathbf w)$, is the set of all letters occurring in $\mathbf w$.
A letter is \textit{simple} [respectively, \textit{multiple}] \textit{in a word} $\mathbf w$ if it occurs in $\mathbf w$ once [respectively, at least twice].
The set of all simple [respectively, multiple] letters in a word $\mathbf w$ is denoted by $\simple(\mathbf w)$ [respectively, $\mul(\mathbf w)$].
If $\mathbf w$ is a word and $X$ is a set of letters then $\mathbf w_X$ denotes the word obtained from $\mathbf w$ by deleting all letters from $X$.
If $X=\{x\}$, then we write $\mathbf w_x$ rather than $\mathbf w_{\{x\}}$.
For a word $\mathbf w$ and letters $x_1,x_2,\dots,x_k\in \con(\mathbf w)$, let $\mathbf w(x_1,x_2,\dots,x_k)$ denote the word obtained from $\mathbf w$ by retaining the letters $x_1,x_2,\dots,x_k$.
Equivalently,
$$
\mathbf w(x_1,x_2,\dots,x_k)=\mathbf w_{\con(\mathbf w)\setminus \{x_1,x_2,\dots,x_k\}}.
$$

Let $\mathbf w$ be a word such that $\simple(\mathbf w)=\{t_1,t_2,\dots, t_m\}$.
We may assume without loss of generality that $\mathbf w(t_1,t_2,\dots, t_m) = t_1t_2 \cdots t_m$.
Then
\begin{equation}
\label{blocks and dividers}
\mathbf w = t_0\mathbf w_0 t_1 \mathbf w_1 \cdots t_m \mathbf w_m,
\end{equation}
where $\mathbf w_0,\mathbf w_1,\dots,\mathbf w_m\in \mathfrak A^\ast$ and $t_0=\lambda$.
The words $\mathbf w_0$, $\mathbf w_1$, \dots, $\mathbf w_m$ are 0-\textit{blocks} of $\mathbf w$, while $t_0,t_1,\dots,t_m$ are 0-\textit{dividers} of $\mathbf w$.
The representation of the word $\mathbf w$ as a product of alternating 0-dividers and 0-blocks, starting with the 0-divider $t_0$ and ending with the 0-block $\mathbf w_m$, is called the 0-\textit{decomposition} of $\mathbf w$.

For $k \ge 1$, we define the $k$-\textit{decomposition} of a word $\mathbf w$ recursively as follows.
Suppose that \eqref{blocks and dividers} is the \mbox{$(k-1)$}-decomposition of $\mathbf w$ with \mbox{$(k-1)$}-dividers $t_0,t_1,\dots,t_m$ and \mbox{$(k-1)$}-blocks $\mathbf w_0,\mathbf w_1,\dots,\mathbf w_m$.
For any $i=0,1,\dots,m$, let $s_{i1},s_{i2},\dots, s_{ir_i}$ be all the letters in $\simple(\mathbf w_i)$ that do not occur to the left of $\mathbf w_i$.
We may assume that $\mathbf w_i(s_{i1},s_{i2},\dots,s_{ir_i})=s_{i1}s_{i2}\cdots s_{ir_i}$. Then
\begin{equation}
\label{blocks and dividers within block}
\mathbf w_i = \mathbf v_{i0} s_{i1} \mathbf v_{i1} s_{i2} \mathbf v_{i2} \cdots s_{ir_i} \mathbf v_{ir_i}
\end{equation}
for some $\mathbf v_{i0},\mathbf v_{i1}, \dots,\mathbf v_{ir_i} \in \mathfrak A^\ast$.
Put $s_{i0}=t_i$.
Then the words $\mathbf v_{i0},\mathbf v_{i1}, \dots$, $\mathbf v_{ir_i}$ are $k$-\textit{blocks} of $\mathbf w$, while the letters $s_{i0},s_{i1},\dots, s_{ir_i}$ are $k$-\textit{dividers} of $\mathbf w$.

As noted in Gusev and Vernikov \cite[Remark~3.1]{Gusev-Vernikov-18}, only the first occurrence of a letter in a given word might be a $k$-divider of  the word for some $k$.
In view of this observation, we use below an expression like ``a letter $x$ is (or is not) a $k$-divider of a word $\mathbf w$'' to mean that the first occurrence of $x$ in $\mathbf w$ has the specified property.

Now decompose every \mbox{$(k-1)$}-block $\mathbf w_i$ from~\eqref{blocks and dividers} in the form~\eqref{blocks and dividers within block}.
Then the word~$\mathbf w$ in \eqref{blocks and dividers} is decomposed as a product of alternating $k$-dividers and $k$-blocks, starting with the $k$-divider $s_{00}=t_0$ and ending with the $k$-block $\mathbf v_{mr_m}$.
This is called the $k$-\textit{decomposition} of $\mathbf w$.

\begin{remark}
\label{R: max decomposition}
Since the length of the word $\mathbf w$ is finite, there is a number $k$ such that the $k$-decomposition of $\mathbf w$ coincides with its $n$-decomposition for each $n>k$.
\end{remark}

We say that the $k$-decomposition of $\mathbf w$ is \textit{maximal} if it coincides with the $n$-decomposition of $\mathbf w$ for all $n>k$.

Let $\occ_x(\mathbf w)$ denote the number of occurrences of a letter $x$ in a word $\mathbf w$.
For a given word $\mathbf w$, a letter $x\in\con(\mathbf w)$, a natural number $i\le\occ_x(\mathbf w)$ and an integer $k\ge 0$, we denote by $h_i^k(\mathbf w,x)$ the right-most $k$-divider of $\mathbf w$ that precedes to the $i$th occurrence of $x$ in $\mathbf w$.

For a given word $\mathbf w$ and a letter $x\in\con(\mathbf w)$, the \textit{depth}  of $x$ in $\mathbf w$ is the number $D(\mathbf w,x)$ defined as follows.
If $x\in\simple(\mathbf w)$, then we put $D(\mathbf w,x)=0$.
Suppose now that $x\in\mul(\mathbf w)$.
If there is a natural $k$ such that the first and the second occurrences of $x$ in $\mathbf w$ lie in different \mbox{$(k-1)$}-blocks of $\mathbf w$, then the depth of $x$ in $\mathbf w$ is equal to the minimal number $k$ with such a property.
If, for any natural $k$, the first and the second occurrences of $x$ in $\mathbf w$ lie in the same $k$-block of $\mathbf w$, then we put $D(\mathbf w,x)=\infty$.
In other words, $D(\mathbf w,x)=k$ if and only if $h_1^{k-1}(\mathbf w, x)\ne h_2^{k-1}(\mathbf w, x)$ and $k$ is the least number with such a property, while $D(\mathbf w,x)=\infty$ if and only if $h_1^{k-1}(\mathbf w, x)=h_2^{k-1}(\mathbf w, x)$ for any $k$.

\begin{lemma}[\mdseries{Gusev and Vernikov~\cite[Lemma~3.7]{Gusev-Vernikov-18}}]
\label{L: k-divider and depth}
A letter $t$ is a $k$-divider of a word $\mathbf w$ if and only if $D(\mathbf w,t)\le k$.\qed
\end{lemma}

The $i$th occurrence of a letter $x$ in a word $\mathbf w$ is denoted by $_{i\mathbf w}x$.
We write $({_{i\mathbf w}x}) < ({_{j\mathbf w}y})$ to indicate that within $\mathbf w$, the $i$th occurrence of $x$ precedes the $j$th occurrence of $y$.

\begin{lemma}[\mdseries{Gusev and Vernikov~\cite[Lemma~3.9]{Gusev-Vernikov-18}}]
\label{L: h_2^{k-1}}
Let $\mathbf w$ be a word, $x$ be a multiple letter in $\mathbf w$ with $D(\mathbf w,x)= k$ and $t$ be a \mbox{$(k-1)$}-divider of $\mathbf w$.
\begin{itemize}
\item[\textup{(i)}] If $t=h_2^{k-1}(\mathbf w,x)$, then $({_{1\mathbf w}x})<({_{1\mathbf w}t})$.
\item[\textup{(ii)}] If $({_{1\mathbf w}x})<({_{1\mathbf w}t})<({_{2\mathbf w}x})$, then $D(\mathbf w,t)= k-1$. Further, if $k>1$, then $({_{2\mathbf w}x})<({_{2\mathbf w}t})$.\qed
\end{itemize}
\end{lemma}

\begin{lemma}[\mdseries{Gusev and Vernikov~\cite[Lemma~3.13]{Gusev-Vernikov-18}}]
\label{L: if first then second}
Let $\mathbf w$ be a word, $r>1$ be a number and $y$ be a letter such that $D(\mathbf w, y)=r-2$.
If $({_{1\mathbf w}z})<({_{1\mathbf w}y})$ for some letter $z$ with $D(\mathbf w,z)\ge r$, then $({_{2\mathbf w}z})<({_{1\mathbf w}y})$.\qed
\end{lemma}

In the following, to facilitate understanding of the form of words, we will sometimes explicitly display the occurrence of a letter, for example,
$$
\mathbf w=\,\stackrel{(1)}{x_1}\,\stackrel{(1)}{x_2}\,\stackrel{(2)}{x_1}\,\stackrel{(1)}{x_3}\,\stackrel{(2)}{x_2}\,\stackrel{(3)}{x_1}.
$$

\begin{lemma}[\mdseries{Gusev and Vernikov~\cite[Lemma~3.14]{Gusev-Vernikov-18}}]
\label{L: form of the identity}
Let $\mathbf u\approx \mathbf v$ be an identity and $\ell$ be a natural number.
Suppose that
\begin{align}
\label{sim(u)=sim(v) & mul(u)=mul(v)}
\simple(\mathbf u)=\simple(\mathbf v)&\text{ and }\mul(\mathbf u)=\mul(\mathbf v),\\
\label{eq the same l-dividers}
h_i^{\ell-1}(\mathbf u,x) = h_i^{\ell-1}(\mathbf v,x)&\text{ for }i=1,2\text{ and all }x\in \con(\mathbf u),
\end{align}
and there is a letter $x_\ell$ such that $D(\mathbf u, x_\ell)=\ell$.
Then there are letters $x_0,x_1,\dots, x_{\ell-1}$ such that $D(\mathbf u,x_s)=D(\mathbf v,x_s)=s$ for any $0\le s<\ell$ and the identity $\mathbf u \approx \mathbf v$ has the form
$$
\begin{aligned}
&\mathbf u_{2\ell+1}\stackrel{(1)}{x_\ell}\mathbf u_{2\ell}\stackrel{(1)}{x_{\ell-1}}\mathbf u_{2\ell-1}\stackrel{(2)}{x_\ell}\mathbf u_{2\ell-2}\stackrel{(1)}{x_{\ell-2}}\mathbf u_{2\ell-3}\stackrel{(2)}{x_{\ell-1}}\mathbf u_{2\ell-4}\stackrel{(1)}{x_{\ell-3}}\\
&\cdot\,\mathbf u_{2\ell-5}\stackrel{(2)}{x_{\ell-2}}\cdots\mathbf u_4\stackrel{(1)}{x_1}\mathbf u_3\stackrel{(2)}{x_2}\mathbf u_2\stackrel{(1)}{x_0}\mathbf u_1\stackrel{(2)}{x_1}\mathbf u_0\\
\approx{}&\mathbf v_{2\ell+1}\stackrel{(1)}{x_\ell}\mathbf v_{2\ell}\stackrel{(1)}{x_{\ell-1}}\mathbf v_{2\ell-1}\stackrel{(2)}{x_\ell}\mathbf v_{2\ell-2}\stackrel{(1)}{x_{\ell-2}}\mathbf v_{2\ell-3}\stackrel{(2)}{x_{\ell-1}}\mathbf v_{2\ell-4}\stackrel{(1)}{x_{\ell-3}}\\
&\cdot\,\mathbf v_{2\ell-5}\stackrel{(2)}{x_{\ell-2}}\cdots\mathbf v_4\stackrel{(1)}{x_1}\mathbf v_3\stackrel{(2)}{x_2}\mathbf v_2\stackrel{(1)}{x_0}\mathbf v_1\stackrel{(2)}{x_1}\mathbf v_0
\end{aligned}
$$
for some $\mathbf u_0,\mathbf u_1,\dots,\mathbf u_{2r+1}, \mathbf v_0,\mathbf v_1,\dots,\mathbf v_{2r+1} \in \mathfrak A^\ast$.\qed
\end{lemma}

\begin{remark}
\label{R: u_{2s+2} and v_{2s+2} do not contain any p-dividers}
Analyzing the proof of Lemma~3.14 in Gusev and Vernikov \cite{Gusev-Vernikov-18}, one can notice that, in Lemma~\ref{L: form of the identity}, the letters $x_0,x_1,\dots,x_{\ell-1}$ can be chosen so that $\mathbf u_{2s+2}$ and $\mathbf v_{2s+2}$ do not contain any $s$-dividers of $\mathbf u$ and $\mathbf v$, respectively, for any $s=0,1,\dots,\ell-1$.\qed
\end{remark}

Evidently, if $\mathbf u = \mathbf v$, then \eqref{sim(u)=sim(v) & mul(u)=mul(v)} and~\eqref{eq the same l-dividers} are true for all $\ell$. So, one can apply Lemma~\ref{L: form of the identity} and Remark~\ref{R: u_{2s+2} and v_{2s+2} do not contain any p-dividers} for the trivial identity and obtain the following corollary.

\begin{corollary}
\label{C: form of the word}
Let $\mathbf u$ be a word and $k$ be a natural number.
Suppose that there is a letter $x_k$ such that $D(\mathbf u, x_k)=k$.
Then there exist letters $x_0,x_1,\dots, x_{k-1}$ such that the word $\mathbf u$ has the form
\begin{equation}
\label{form of u}
\begin{aligned}
&\mathbf u_{2k+1}\stackrel{(1)}{x_k}\mathbf u_{2k}\stackrel{(1)}{x_{k-1}}\mathbf u_{2k-1}\stackrel{(2)}{x_k}\mathbf u_{2k-2}\stackrel{(1)}{x_{k-2}}\mathbf u_{2k-3}\stackrel{(2)}{x_{k-1}}\mathbf u_{2k-4}\stackrel{(1)}{x_{k-3}}\\
&\cdot\,\mathbf u_{2k-5}\stackrel{(2)}{x_{k-2}}\cdots\mathbf u_4\stackrel{(1)}{x_1}\mathbf u_3\stackrel{(2)}{x_2}\mathbf u_2\stackrel{(1)}{x_0}\mathbf u_1\stackrel{(2)}{x_1}\mathbf u_0
\end{aligned}
\end{equation}
for some $\mathbf u_0,\mathbf u_1,\dots,\mathbf u_{2k+1} \in \mathfrak A^\ast$ such that $D(\mathbf u,x_s)=s$ and $\mathbf u_{2s+2}$ does not contain any $s$-dividers of $\mathbf u$ for any $s=0,1,\dots,k-1$.\qed
\end{corollary}

\begin{lemma}
\label{L: 2nd occurence of y_m general}
Let $k,m$ be natural numbers and $\mathbf u$ be a word of the form
$$
\begin{aligned}
&\mathbf u_{2k+2}\stackrel{(1)}{y_m}\mathbf u_{2k+1}\stackrel{(1)}{x_k}\mathbf u_{2k}\stackrel{(1)}{x_{k-1}}\mathbf u_{2k-1}\stackrel{(2)}{x_k}\mathbf u_{2k-2}\stackrel{(1)}{x_{k-2}}\mathbf u_{2k-3}\stackrel{(2)}{x_{k-1}}\\
&\cdot\,\mathbf u_{2k-4}\stackrel{(1)}{x_{k-3}}\mathbf u_{2k-5}\stackrel{(2)}{x_{k-2}}\cdots\mathbf u_4\stackrel{(1)}{x_1}\mathbf u_3\stackrel{(2)}{x_2}\mathbf u_2\stackrel{(1)}{x_0}\mathbf u_1\stackrel{(2)}{x_1}\mathbf u_0,
\end{aligned}
$$
where $D(\mathbf u, y_m)=m$ and $D(\mathbf u,x_s)=s$ for any $0\le s\le k$. If $\mathbf u_{2k+1}$ does not contain any \mbox{$(m-1)$}-divider of $\mathbf u$,  then ${_{2\mathbf u}y_m}$ occurs in one of the subwords $\mathbf u_{2m}$, $\mathbf u_{2m-1}$ or $\mathbf u_{2m-2}$ of $\mathbf u$.
\end{lemma}

\begin{proof}
Since $D(\mathbf u,y_m)=m>0$, we have $y_m \in \mul(\mathbf u)$.
Let $y_{m-1}=h_2^{m-1}(\mathbf u,y_m)$.
Since $({_{1\mathbf u}y_m})<({_{1\mathbf u}y_{m-1}})<({_{2\mathbf u}y_m})$ by Lemma~\ref{L: h_2^{k-1}}(i), the letter ${_{1\mathbf u}y_{m-1}}$ does not occur in $\mathbf u_{2k+2}$.
By the hypothesis, ${_{1\mathbf u}y_{m-1}}$ does not occur in $\mathbf u_{2k+1}$.
Since $({_{1\mathbf u}y_{m-1}})<({_{2\mathbf u}y_m})$, the letter ${_{2\mathbf u}y_m}$ occurs in neither $\mathbf u_{2k+2}$ nor $\mathbf u_{2k+1}$.
Then $({_{1\mathbf u}x_k})<({_{2\mathbf u}y_m})$.
It follows that $m \le k+1$.
If $m=k+1$, then ${_{2\mathbf u}y_m}$ occurs in $\mathbf u_{2m-2}=\mathbf u_{2k}$ because $h_1^k(\mathbf u,y_m) \ne h_2^k(\mathbf u,y_m)$ otherwise.
If $m=k$, then ${_{2\mathbf u}y_m}$ occurs in either $\mathbf u_{2k}$ or $\mathbf u_{2k-1}$ or $\mathbf u_{2k-2}$ because $h_1^{k-1}(\mathbf u,y_m) \ne h_2^{k-1}(\mathbf u,y_m)$ otherwise.
If $k > m$, then ${_{1\mathbf u}y_{m-1}}$ and so ${_{2\mathbf u}y_m}$ do not occur in both $\mathbf u_{2k}$ and $\mathbf u_{2k-1}$ because $h_1^{m-1}(\mathbf u,x_k) \ne h_2^{m-1}(\mathbf u,x_k)$ otherwise.
Further, if $k > m+1$, then ${_{1\mathbf u}y_{m-1}}$ and so ${_{2\mathbf u}y_m}$ do not occur in the subwords $\mathbf u_{2r+1}$, $\mathbf u_{2r}$ and $\mathbf u_{2r-1}$ of $\mathbf u$ for any $r=m+1,m+2,\dots,k-1$ because $h_1^{r-2}(\mathbf u,x_r) \ne h_2^{r-2}(\mathbf u,x_r)$ otherwise.
Finally, if $m>1$, then $({_{2\mathbf u}y_m})<({_{1\mathbf u}x_{m-2}})$ by Lemma~\ref{L: if first then second}.
This means that if $m>1$, then ${_{2\mathbf u}y_m}$ does not occur in $\mathbf u_{2m-3}$, $\mathbf u_{2m-4}$, \dots, $\mathbf u_0$.
Therefore, ${_{2\mathbf u}y_m}$ occurs in either $\mathbf u_{2m}$ or $\mathbf u_{2m-1}$ or $\mathbf u_{2m-2}$.
\end{proof}

\subsection{Some known facts}
\label{subsec: some known facts}

For an identity system $\Sigma$, we denote by $\var\,\Sigma$ the variety of monoids defined by $\Sigma$.
It is well known that the variety
$$
\mathbf{LRB}=\var\{xy\approx xyx\}
$$
of left regular bands is generated by the monoid obtained by adjoining an identity element to the left zero semigroup of order two.
The \textit{initial part} of a word $\mathbf w$, denoted by $\ini(\mathbf w)$, is the word obtained from $\mathbf w$ by retaining the first occurrence of each letter.
The following statement is well known and can be easily verified.

\begin{lemma}
\label{L: word problem LRB}
A non-trivial identity $\mathbf u \approx \mathbf v$ holds in the variety $\mathbf{LRB}$ if and only if $\ini(\mathbf u)=\ini(\mathbf v)$.\qed
\end{lemma}

We use the standard symbol $\mathbb N$ to denote the set of all natural numbers.
For any $s\in\mathbb N$ and $1\le q\le s$, put
$$
\mathbf b_{s,q}=x_{s-1}x_sx_{s-2}x_{s-1}\cdots x_{q-1}x_q.
$$
For brevity, we will write $\mathbf b_s$ rather than $\mathbf b_{s,1}$. We put also $\mathbf b_0=\lambda$ for convenience.
We introduce the following four countably infinite series of identities:
$$
\begin{aligned}
\label{alpha_k}
\alpha_k:&\enskip x_ky_kx_{k-1}x_ky_k \mathbf b_{k-1} \approx y_kx_kx_{k-1}x_ky_k \mathbf b_{k-1},\\
\label{beta_k}
\beta_k:&\enskip xx_kx\mathbf b_k \approx x_kx^2\mathbf b_k,\\
\label{gamma_k}
\gamma_k:&\enskip y_1y_0x_ky_1 \mathbf b_k \approx y_1y_0y_1x_k \mathbf b_k,\\
\label{delta_k^m}
\delta_k^m:&\enskip y_{m+1}y_mx_ky_{m+1} \mathbf b_{k,m}y_m\mathbf b_{m-1}\approx y_{m+1}y_my_{m+1}x_k\mathbf b_{k,m}y_m\mathbf b_{m-1},
\end{aligned}
$$
where $k\in \mathbb N$ and $1\le m\le k$.
The subvariety of a variety $\mathbf V$ defined by a set $\Sigma$ of identities is denoted by $\mathbf V \Sigma$.
Let
$$
\begin{aligned}
\mathbf F & =\mathbf P\{x^2y^2 \approx y^2x^2\},\\ \mathbf F_k & =\mathbf F\{\alpha_k\},\\ \mathbf H_k & =\mathbf F\{\beta_k\},\\ \mathbf I_k & =\mathbf F\{\gamma_k\}, \\ \text{ and } \mathbf J_k^m & =\mathbf F\{\delta_k^m\},
\end{aligned}
$$
where $k\in \mathbb N$ and $1\le m\le k$.
The trivial variety of monoids is denoted by $\mathbf T$.
Let $\mathbf{SL}$ denote the variety of all semilattice monoids.
Put
$$
\begin{aligned}
&\mathbf C=\var\{x^2\approx x^3,\,xy\approx yx\},\\
&\mathbf D=\var\{x^2 \approx x^3,\,x^2y \approx xyx \approx yx^2\},\\
&\mathbf E=\var\{x^2 \approx x^3,\,x^2y \approx xyx,\, x^2y^2 \approx y^2x^2\}.
\end{aligned}
$$
The subvariety lattice of a monoid variety $\mathbf V$ is denoted by $\mathfrak L(\mathbf V)$.

\begin{lemma}[\mdseries{Gusev and Vernikov~\cite[Proposition~6.1]{Gusev-Vernikov-18}}]
\label{L: L(F)}
The chain
$$
\begin{aligned}
\mathbf T\subset \mathbf{SL}\subset \mathbf C \subset\mathbf D\subset\mathbf E&\subset\mathbf F_1\subset\mathbf H_1\subset\mathbf I_1\subset\mathbf J_1^1\\
&\subset\mathbf F_2\subset\mathbf H_2\subset\mathbf I_2\subset\mathbf J_2^1\subset\mathbf J_2^2\\
&\rule{6pt}{0pt}\vdots\\
&\subset\mathbf F_k\subset\mathbf H_k\subset\mathbf I_k\subset\mathbf J_k^1\subset\mathbf J_k^2\subset\cdots\subset\mathbf J_k^k\\
&\rule{6pt}{0pt}\vdots\\
&\subset\mathbf F
\end{aligned}
$$
is the lattice $\mathfrak L(\mathbf F)$.\qed
\end{lemma}

\begin{lemma}[\mdseries{Gusev and Vernikov~\cite[Proposition~6.9(i)]{Gusev-Vernikov-18}}]
\label{L: word problem F_k}
A non-trivial identity $\mathbf u \approx \mathbf v$ holds in the variety $\mathbf F_k$ if and only if \eqref{sim(u)=sim(v) & mul(u)=mul(v)} and~\eqref{eq the same l-dividers}  hold with $\ell=k$.
\end{lemma}

An identity $\mathbf u \approx \mathbf v$ is $k$-\textit{well-balanced} if the words $\mathbf u$ and $\mathbf v$ have the same set of $k$-dividers, these $k$-dividers appear in $\mathbf u$ and $\mathbf v$ in the same order, and ${_{i\mathbf u}x}$ occurs in some $k$-block of $\mathbf u$ if and only if ${_{i\mathbf v}x}$ occurs in the corresponding  $k$-block of $\mathbf v$ for any letter $x$ and $i\in\{1,2\}$.

\begin{corollary}
\label{C: (k-1)-well-balanced}
Let $\mathbf u\approx \mathbf v$ be an identity of $\mathbf F_k$.
Then the identity $\mathbf u\approx \mathbf v$ is \mbox{$(k-1)$}-well-balanced.
\end{corollary}

\begin{proof}
Let $t_0\mathbf u_0t_1\mathbf u_1\cdots t_m\mathbf u_m$ be the \mbox{$(k-1)$}-decomposition of $\mathbf u$.
According to Lemma~\ref{L: word problem F_k}, \eqref{sim(u)=sim(v) & mul(u)=mul(v)} and~\eqref{eq the same l-dividers} hold with $\ell=k$.
Then the \mbox{$(k-1)$}-decomposition of $\mathbf v$ has the form $t_0\mathbf v_0t_1\mathbf v_1\cdots t_m\mathbf v_m$ by Gusev and Vernikov \cite[Lemma~3.8]{Gusev-Vernikov-18}.
Suppose that ${_{i\mathbf u}x}$ occurs in the \mbox{$(k-1)$}-block $\mathbf u_j$ of $\mathbf u$, where $i \in \{1,2\}$ and $j\in\{0,1,\dots,m\}$.
Then $h_i^{k-1}(\mathbf u,x)=t_j$.
Since~\eqref{eq the same l-dividers} holds with $\ell=k$, we have $h_i^{k-1}(\mathbf v,x)=t_j$ and so ${_{i\mathbf v}x}$ occurs in the \mbox{$(k-1)$}-block $\mathbf v_j$ of $\mathbf u$.
Therefore, the identity $\mathbf u\approx \mathbf v$ is \mbox{$(k-1)$}-well-balanced.
\end{proof}

\section{Some properties of the variety $\mathbf P$}
\label{sec: properties of P}

 Recall that the variety $\mathbf P$ is defined by the identity
\begin{equation}
\label{xsxt=xsxtx}
xsxt \approx xsxtx.
\end{equation}
For any word $\mathbf w$, let $\ini_2(\mathbf w)$ denote the word obtained from $\mathbf w$ by retaining the first and second occurrences of each letter.

The following result is evident.

\begin{lemma}
\label{L: 2-limited word}
For any $\mathbf w \in \mathfrak A^\ast$, the identity $\mathbf w \approx \ini_2(\mathbf w)$ holds in the variety $\mathbf P$.\qed
\end{lemma}

Let $\phi\colon \mathfrak A \to \mathfrak A^\ast$ be a substitution, $x_1,x_2,\dots,x_n\in \mathfrak A$ and $\mathbf w_1,\mathbf w_2,\dots,\mathbf w_n\in \mathfrak A^\ast$.
For brevity, we will say that $\phi$ is the substitution
$$
(x_1,x_2,\dots,x_n)\mapsto (\mathbf w_1,\mathbf w_2,\dots,\mathbf w_n)
$$
if $\phi(x_i)=\mathbf w_i$ for any $i=1,2,\dots,n$ and $\phi(a)=a$ otherwise.

\begin{lemma}
\label{L: does not contain LRB,F_1}
Let $\mathbf V$ be a subvariety of $\mathbf P$.
\begin{itemize}
\item[\textup{(i)}] If $\mathbf{LRB}\nsubseteq\mathbf V$, then $\mathbf V\subseteq\mathbf F$.
\item[\textup{(ii)}] If $\mathbf F_1\nsubseteq\mathbf V$, then $\mathbf V\subseteq\mathbf{LRB}\vee \mathbf C=\var\{x^2 \approx x^3,\,x^2y \approx xyx\}$.
\end{itemize}
\end{lemma}

\begin{proof}
(i) It follows from Lee et al. \cite[Theorem~5.17]{Lee-Rhodes-Steinberg-19} that $\mathbf V$ satisfies the identity $x^2(y^2x^2)^2 \approx (y^2x^2)^2$.
It is then easily shown that $\mathbf V$ satisfies $x^2y^2 \approx y^2x^2$.
Thus, $\mathbf V\subseteq\mathbf F$.

\smallskip

(ii) If $\mathbf C \nsubseteq \mathbf V$, then by Gusev and Vernikov \cite[Corollary~2.6]{Gusev-Vernikov-18}, the variety $\mathbf V$ is \textit{completely regular}, that is, it consists of unions of groups.
Then $\mathbf V$ satisfies the identity $x \approx x^n$ for some $n\ge 2$.
Clearly, this identity together with~\eqref{xsxt=xsxtx} imply the identity $xy \approx xyx$, whence $\mathbf V\subseteq\mathbf{LRB}$.
Therefore, we may assume that $\mathbf C \subseteq \mathbf V$.
There is an identity $\mathbf u \approx \mathbf v$ of $\mathbf V$ that is not satisfied by $\mathbf F_1$.
According to Gusev and Vernikov \cite[Proposition~2.2]{Gusev-Vernikov-18} and Lemma~\ref{L: word problem F_k}, \eqref{sim(u)=sim(v) & mul(u)=mul(v)} holds but \eqref{eq the same l-dividers} with $\ell=1$ does not.
Then $h_i^0(\mathbf u,x)\ne h_i^0(\mathbf v,x)$ for some $x\in\con(\mathbf u)$ and $i\in\{1,2\}$.
If $i=1$, then we can multiply both sides of the identity $\mathbf u\approx \mathbf v$ on the left by $xt$, where $t\notin\con(\mathbf u)$.
So, we may assume that $i=2$.
Put $y=h_2^0(\mathbf u,x)$ and $z=h_2^0(\mathbf v,x)$.
If $y$ and $z$ occur in $\mathbf u$ and $\mathbf v$ in different order, then $\mathbf V$ is commutative and so $\mathbf V\subseteq \mathbf C$.
Therefore, we may assume without loss of generality that $(_{1\mathbf u}z)<(_{1\mathbf u}y)$ and $(_{1\mathbf v}z)<(_{1\mathbf v}y)$.
Then the identity $\mathbf u(x,y) \approx \mathbf v(x,y)$ is equivalent modulo~\eqref{xsxt=xsxtx} to $xyx \approx x^2y$.
It follows from Lee \cite[Lemma~3.3]{Lee-12b} that $\mathbf{LRB}\vee \mathbf C=\var\{x^2 \approx x^3,\,x^2y \approx xyx\}$.
So, $\mathbf V \subseteq \mathbf{LRB}\vee \mathbf C$.
\end{proof}

\begin{lemma}
\label{L: V does not contain F_k}
Let $\mathbf V$ be a subvariety of $\mathbf P$ that does not contain $\mathbf F_{k+1}$.
Then $\mathbf V$ satisfies the identity
$$
\label{kappa_k}
\kappa_k:\enskip xx_kx\mathbf b_k \approx x^2x_k\mathbf b_k.
$$
\end{lemma}

\begin{proof}
If $\mathbf{LRB}\nsubseteq\mathbf V$, then $\mathbf V\subseteq\mathbf F$ by Lemma~\ref{L: does not contain LRB,F_1}(i).
Since $\mathbf F_{k+1}\nsubseteq \mathbf V$,  it follows from Lemma~\ref{L: L(F)} that $\mathbf V \subseteq \mathbf J_k^k$.
It is verified in Gusev and Vernikov \cite[Lemma~6.3]{Gusev-Vernikov-18} that $\mathbf J_k^k$ satisfies $\kappa_k$.
Then $\mathbf V$ satisfies $\kappa_k$ as well.
Therefore, we may assume that $\mathbf{LRB}\subseteq\mathbf V$.
Further, if $\mathbf F_1\nsubseteq\mathbf V$, then $\mathbf V$ satisfies $x^2y \approx xyx$ by Lemma~\ref{L: does not contain LRB,F_1}(ii).
Clearly, $x^2y \approx xyx$ implies $\kappa_k$.
Therefore, we may assume that $\mathbf F_1\subseteq\mathbf V$.

Let $r$ be the least number such that $\mathbf F_r\subseteq\mathbf V$.
Then $\mathbf F_{r+1}\nsubseteq\mathbf V$.
Evidently, $r \le k$.
There is an identity $\mathbf u \approx \mathbf v$ of $\mathbf V$ that is not satisfied by $\mathbf F_{r+1}$.
According to Lemma~\ref{L: word problem F_k}, \eqref{sim(u)=sim(v) & mul(u)=mul(v)} and~\eqref{eq the same l-dividers} hold with $\ell=1,2,\dots,r$, while \eqref{eq the same l-dividers} with $\ell=r+1$ does not.
Then $h_i^r(\mathbf u,x)\ne h_i^r(\mathbf v,x)$ for some $x\in\con(\mathbf u)$ and $i\in\{1,2\}$.
If $i=1$, then we multiply both sides of the identity $\mathbf u\approx \mathbf v$ on the left by $xt$, where $t\notin\con(\mathbf u)=\con(\mathbf v)$.
So, we may assume that $i=2$.
Put $a=h_2^r(\mathbf u,x)$ and $b=h_2^r(\mathbf v,x)$.
We may assume without loss of generality that $(_{1\mathbf u}b)<(_{1\mathbf u}a)$.
Then $(_{1\mathbf v}b)<(_{1\mathbf v}a)$ by Lemma~\ref{L: word problem LRB}.

In view of Lemma~\ref{L: k-divider and depth}, $D(\mathbf u, a)=s$ for some $s\le r$.
Put $x_s=a$.
According to Lemma~\ref{L: form of the identity}, there exist letters $x_0,x_1,\dots, x_{s-1}$ such that $D(\mathbf u,x_j)=D(\mathbf v,x_j)=j$ for any $0\le j<s$ and the identity $\mathbf u \approx \mathbf v$ is of the form
$$
\begin{aligned}
&\mathbf u_{2s+1}\stackrel{(1)}{x_s}\mathbf u_{2s}\stackrel{(1)}{x_{s-1}}\mathbf u_{2s-1}\stackrel{(2)}{x_s}\mathbf u_{2s-2}\stackrel{(1)}{x_{s-2}}\mathbf u_{2s-3}\stackrel{(2)}{x_{s-1}}\mathbf u_{2s-4}\stackrel{(1)}{x_{s-3}}\\
&\cdot\,\mathbf u_{2s-5}\stackrel{(2)}{x_{s-2}}\cdots\mathbf u_4\stackrel{(1)}{x_1}\mathbf u_3\stackrel{(2)}{x_2}\mathbf u_2\stackrel{(1)}{x_0}\mathbf u_1\stackrel{(2)}{x_1}\mathbf u_0\\
\approx{}&\mathbf v_{2s+1}\stackrel{(1)}{x_s}\mathbf v_{2s}\stackrel{(1)}{x_{s-1}}\mathbf v_{2s-1}\stackrel{(2)}{x_s}\mathbf v_{2s-2}\stackrel{(1)}{x_{s-2}}\mathbf v_{2s-3}\stackrel{(2)}{x_{s-1}}\mathbf v_{2s-4}\stackrel{(1)}{x_{s-3}}\\
&\cdot\,\mathbf v_{2s-5}\stackrel{(2)}{x_{s-2}}\cdots\mathbf v_4\stackrel{(1)}{x_1}\mathbf v_3\stackrel{(2)}{x_2}\mathbf v_2\stackrel{(1)}{x_0}\mathbf v_1\stackrel{(2)}{x_1}\mathbf v_0
\end{aligned}
$$
for some $\mathbf u_0,\mathbf u_1,\dots,\mathbf u_{2s+1}, \mathbf v_0,\mathbf v_1,\dots,\mathbf v_{2s+1} \in \mathfrak A^\ast$.
Clearly, ${_{1\mathbf u}x}$ occurs in the subword $\mathbf u_{2s+1}$ of $\mathbf u$, while ${_{1\mathbf v}x}$ and ${_{2\mathbf v}x}$ occur in the subword $\mathbf v_{2s+1}$ of $\mathbf v$.
Since $x_s=h_2^r(\mathbf u,x)$ and $x_{s-1}$ is an $r$-divider by Lemma~\ref{L: k-divider and depth}, the letter ${_{2\mathbf u}x}$ occurs in $\mathbf u_{2s}$.
Therefore, the identity
$$
\mathbf u(x,x_1,x_2,\dots,x_s) \approx \mathbf v(x,x_1,x_2,\dots,x_s)
$$
is equivalent modulo~\eqref{xsxt=xsxtx} to $\kappa_s$.
Clearly, $\kappa_s$ implies $\kappa_k$. Therefore, $\kappa_k$ is satisfied by $\mathbf V$.
\end{proof}

For any word $\mathbf w$, let $\ini^2(\mathbf w)$ denote the word obtained from $\ini(\mathbf w)$ by replacing each letter $x$ with $x^2$.
For example, if $\mathbf w = x^3yzyx$, then $\ini(\mathbf w) = xyz$ and so $\ini^2(\mathbf w)=x^2y^2z^2$.
A word $\mathbf w$ is said to be $2$-\textit{limited} if $\occ_x(\mathbf w) \le 2$ for any letter $x$.

\begin{lemma}
\label{L: w=ini^2(w)}
Let $\mathbf w$ be a word that does not contain simple letters.
Then the identities~\eqref{xsxt=xsxtx} and
\begin{equation}
\label{xyxy=xxyy}
(xy)^2 \approx x^2y^2.
\end{equation}
imply the identity $\mathbf w \approx \ini^2(\mathbf w)$.
\end{lemma}

\begin{proof}
First, we note that $\mathbf P\{\eqref{xyxy=xxyy}\}$ satisfies the identity
\begin{equation}
\label{x_1^2..x_n^2=(x_1..x_n)^2}
x_1^2x_2^2\cdots x_n^2 \approx (x_1x_2\cdots x_n)^2
\end{equation}
for any $n \ge 2$ because the identities
$$
(x_1x_2\cdots x_n)^2 \approx x_1^2(x_2x_3\cdots x_n)^2 \approx x_1^2x_2^2(x_3x_4\cdots x_n)^2 \approx \cdots \approx x_1^2x_2^2\cdots x_n^2
$$
follow from~\eqref{xyxy=xxyy}.

In view of Lemma~\ref{L: 2-limited word}, we may assume that the word $\mathbf w$ is 2-limited.
We say that a letter $x$ \textit{forms an island} in $\mathbf w$ if $\mathbf w = \mathbf w_1x^2\mathbf w_2$ for some words $\mathbf w_1,\mathbf w_2 \in \mathfrak A^\ast$ such that $x\notin\con(\mathbf w_1\mathbf w_2)$.
Let $k$ be the number of distinct letters that do not form an island in $\mathbf w$.
We will use induction on $k$.

\smallskip

\textit{Induction base}: $k=0$.
Then $\mathbf w=\ini^2(\mathbf w)$, and we are done.

\smallskip

\textit{Induction step}: $k>0$ and $\mathbf P\{\eqref{xyxy=xxyy}\}$ satisfies the identity $\mathbf v' \approx \ini^2(\mathbf v')$ for any 2-limited word $\mathbf v'$ with at most $k-1$ letters that do not form an island in $\mathbf v'$.
Then, since $\mathbf w$ is 2-limited, there is a letter $x$ such that $\mathbf w = \mathbf w_1xy_1^{k_1}y_2^{k_2}\cdots y_n^{k_n}x\mathbf w_2$, where $k_1,k_2,\dots,k_n\in\{1,2\}$ and if $k_i=1$, then $y_i\in\con(\mathbf w_1)$. Then the identities
\begin{align*}
\mathbf w & = \mathbf w_1xy_1^{k_1}y_2^{k_2}\cdots y_n^{k_n}x\mathbf w_2\\
& \stackrel{\eqref{xsxt=xsxtx}}\approx \mathbf w_1xy_1^2y_2^2\cdots y_n^2xy_1^2y_2^2\cdots y_n^2\mathbf w_2\\
& \stackrel{\eqref{x_1^2..x_n^2=(x_1..x_n)^2}}\approx \mathbf w_1x(y_1y_2\cdots y_n)^2x(y_1y_2\cdots y_n)^2\mathbf w_2\\
& \stackrel{\eqref{xyxy=xxyy}}\approx \mathbf w_1x^2(y_1y_2\cdots y_n)^4\mathbf w_2\\
& \stackrel{\eqref{x_1^2..x_n^2=(x_1..x_n)^2}}\approx \mathbf w_1x^2y_1^4y_2^4\cdots y_n^4\mathbf w_2\\
& \stackrel{\eqref{xsxt=xsxtx}}\approx \mathbf w_1x^2y_1^{k_1}y_2^{k_2}\cdots y_n^{k_n}\mathbf w_2=\mathbf w'
\end{align*}
are satisfied by $\mathbf P\{\eqref{xyxy=xxyy}\}$.
Clearly, the word $\mathbf w'$ has exactly $k-1$ letters that do not form an island in $\mathbf w'$.
By the induction assumption, $\mathbf w' \approx \ini^2(\mathbf w')$ holds in $\mathbf P\{\eqref{xyxy=xxyy}\}$.
It remains to note that $\ini^2(\mathbf w)=\ini^2(\mathbf w')$. Therefore, $\mathbf P\{\eqref{xyxy=xxyy}\}$ satisfies $\mathbf w \approx \ini^2(\mathbf w)$.
\end{proof}

\begin{corollary}
\label{C: block = ini^2(block)}
Let $\mathbf w$ be a word, $k$ be a number such that the $k$-decomposition of $\mathbf w$ is maximal and $\mathbf a$ be a $k$-block of $\mathbf u$.
If $\mathbf w = \mathbf w'\mathbf a \mathbf w''$, then the variety $\mathbf P\{\eqref{xyxy=xxyy}\}$ satisfies the identity $\mathbf w \approx \mathbf w'\,\ini^2(\mathbf a)\,\mathbf w''$.
\end{corollary}

\begin{proof}
Let $x$ be a letter such that ${_{1\mathbf w}x}$ occurs in the $k$-block $\mathbf a$ of $\mathbf w$. If ${_{2\mathbf w}x}$ does not occur in $\mathbf a$, then $h_1^k(\mathbf w,x) \ne h_2^k(\mathbf w,x)$ and so $x$ is a \mbox{$(k+1)$}-divider of $\mathbf w$.
This contradicts the fact that the $k$-decomposition of $\mathbf w$ is maximal.
Therefore, ${_{2\mathbf w}x}$ occurs in $\mathbf a$ as well. We see that $\con(\mathbf a) \subseteq \mul(\mathbf w'\mathbf a)$.
Then we can apply the identity~\eqref{xsxt=xsxtx} and add some occurrences of letters in the $k$-block $\mathbf a$ so that the resulting $k$-block $\mathbf b$ contains multiple letters only and $\ini(\mathbf a)=\ini(\mathbf b)$.
Then identity $\mathbf b \approx \mathbf \ini^2(\mathbf b)$ follows from~\eqref{xsxt=xsxtx} and~\eqref{xyxy=xxyy} by Lemma~\ref{L: w=ini^2(w)}.
Evidently, $\ini^2(\mathbf a)=\ini^2(\mathbf b)$.
We see that the identities
$$
\mathbf w = \mathbf w'\mathbf a\mathbf w'' \approx \mathbf w'\mathbf b\mathbf w'' \approx \mathbf w'\,\ini^2(\mathbf b)\,\mathbf w'' = \mathbf w'\,\ini^2(\mathbf a)\,\mathbf w''
$$
follow from the identities~\eqref{xsxt=xsxtx} and~\eqref{xyxy=xxyy}, and we are done.
\end{proof}

\section{Critical pairs}
\label{sec: critical pairs}

An identity $\mathbf u \approx \mathbf v$ is said to be \textit{balanced} if $\occ_x(\mathbf u)=\occ_x(\mathbf v)$ for any $x \in \mathfrak A$.
We say that a pair $\{_i x, {_j y}\}$ of occurrences of letters $x$ and $y$ in a balanced identity $\mathbf u \approx \mathbf v$ is \textit{critical} if $\mathbf u$ contains ${_{i\mathbf u} x}{_{j\mathbf u} y}$ as a subword and ${_{j\mathbf v} y}$ precedes ${_{i\mathbf v} x}$ in $\mathbf v$.
Let $\mathbf w$ denote the word obtained from $\mathbf u$ by replacing ${_{i\mathbf u} x}{_{j\mathbf u} y}$ with ${_{j\mathbf u} y}{_{i\mathbf u} x}$.
Given a set $\Delta$ of identities and a balanced identity $\mathbf u \approx \mathbf v$, we say that the critical pair $\{_i x, {_j y}\}$ is $\Delta$-\textit{removable in} $\mathbf u \approx \mathbf v$ if $\Delta$ implies $\mathbf u \approx \mathbf w$.

The following special case of Sapir \cite[Lemma~3.4]{Sapir-15}, which can be easily verified by induction, describes the standard method of deriving identities by removing critical pairs.
This method traces back to Jackson and Sapir~\cite{Jackson-Sapir-00} and Sapir~\cite{Sapir-00}, for instance.

\begin{lemma}
\label{L: fblemma}
Let $\mathbf V$ be a monoid variety and $\Delta$ be a set of identities.
Suppose that each critical pair in every balanced identity of $\mathbf V$ is $\Delta$-removable.
Then every balanced identity of $\mathbf V$ can be derived from $\Delta$.\qed
\end{lemma}

\subsection{Removing critical pairs of the form $\{{_1}x,{_2}y\}$}
\label{subsec: critical pairs (1x,2y)}

If $\phi\colon \mathfrak A \to \mathfrak A^\ast$ is a substitution and $\varepsilon$ denote an identity $\mathbf u \approx \mathbf v$, then, for convenience, we will write $\phi(\varepsilon)$ rather than $\phi(\mathbf u) \approx \phi(\mathbf v)$.

The following lemma generalizes Gusev and Vernikov \cite[Lemma~6.3]{Gusev-Vernikov-18}.

\begin{lemma}
\label{L: kappa_k and delta_k^k}
For any $k \in \mathbb N$, the identities $\kappa_k$ and $\delta_k^k$ are equivalent within $\mathbf P$.
\end{lemma}

\begin{proof}
Let $\phi$ be the substitution $(y_k,y_{k+1})\mapsto(1,x)$.
Then the identity $\phi(\delta_k^k)$ is equal to $\kappa_k$.
Therefore, $\delta_k^k$ implies $\kappa_k$.
Further, $\delta_k^k$ follows from~\eqref{xsxt=xsxtx} and $\kappa_k$ because
\begin{align*}
y_{k+1}y_kx_ky_{k+1}\mathbf b_{k,k}y_k\mathbf b_{k-1}=\ {}&y_{k+1}y_kx_ky_{k+1}x_{k-1}x_ky_k\mathbf b_{k-1}\\
\stackrel{\eqref{xsxt=xsxtx}}\approx{}&y_{k+1}y_kx_ky_{k+1}x_{k-1}x_ky_kx_k\mathbf c\\
\stackrel{\kappa_k}\approx\ {}&y_{k+1}^2y_kx_kx_{k-1}x_ky_kx_k\mathbf c\\
\stackrel{\eqref{xsxt=xsxtx}}\approx{}&y_{k+1}^2y_kx_kx_{k-1}x_ky_k\mathbf d\\
\stackrel{\kappa_k}\approx\ {}&y_{k+1}y_ky_{k+1}x_kx_{k-1}x_ky_k\mathbf d\\
\stackrel{\eqref{xsxt=xsxtx}}\approx{}&y_{k+1}y_ky_{k+1}x_kx_{k-1}x_ky_k\mathbf b_{k-1}\\
=\ {}&y_{k+1}y_ky_{k+1}x_k\mathbf b_{k,k}y_k\mathbf b_{k-1},
\end{align*}
where
$$
\mathbf c=
\begin{cases}
\lambda&\text{if }k=1,\\
x_{k-2}x_{k-1}x_k\mathbf b_{k-2}&\text{if }k>1
\end{cases}
\ \text{ and } \
\mathbf d=
\begin{cases}
\lambda&\text{if }k=1,\\
x_{k-2}x_kx_{k-1}x_k\mathbf b_{k-2}&\text{if }k>1.
\end{cases}
$$
The lemma is thus proved.
\end{proof}

We introduce the following countably infinite series of identities:
$$
\begin{aligned}
\label{varepsilon_k}
\varepsilon_{k-1}:&\enskip y_kx_{k-1}xy_kx\mathbf b_{k-1} \approx y_kx_{k-1}y_kx^2\mathbf b_{k-1},
\end{aligned}
$$
where $k \in \mathbb N$.

\begin{lemma}
\label{L: gamma_k, delta_k^k,epsilon_k}
The following proper inclusions hold:
\begin{itemize}
\item[\textup{(i)}] $\mathbf P\{\gamma_1\} \subset \mathbf P\{\gamma_2\} \subset \cdots \subset \mathbf P\{\gamma_k\} \subset \cdots  \subset \mathbf P\{\varepsilon_0\}$;
\item[\textup{(ii)}] $\mathbf P\{\gamma_k\} \subset \mathbf P\{\delta_k^1\} \subset \mathbf P\{\delta_k^2\} \subset \cdots \subset\mathbf P\{\delta_k^k\} \subset \mathbf P\{\delta_{k+1}^k\} \subset \cdots \subset \mathbf P\{\delta_{k+m}^k\} \subset \cdots  \subset \mathbf P\{\varepsilon_k\}$;
\item[\textup{(iii)}] $\mathbf P\{\varepsilon_0\} \subset \mathbf P\{\varepsilon_1\} \subset \cdots \subset \mathbf P\{\varepsilon_k\} \subset \cdots \subset \mathbf P\{\eqref{xyxy=xxyy}\}$.
\end{itemize}
\end{lemma}

\begin{proof}
(i) Evidently, $\mathbf P\{\gamma_1\} \subseteq \mathbf P\{\gamma_2\} \subseteq \cdots \subseteq \mathbf P\{\gamma_k\} \subseteq \cdots  \subseteq \mathbf P\{\varepsilon_0\}$.
These inclusions are proper because $\mathbf I_{k+1}\subseteq \mathbf P\{\gamma_{k+1}\}\subseteq \mathbf P\{\varepsilon_0\}$ but $\mathbf I_{k+1}\nsubseteq \mathbf P\{\gamma_k\}$ for any $k \in \mathbb N$.

(ii) The inclusion $\mathbf P\{\gamma_k\}\subseteq \mathbf P\{\delta_k^1\}$ is evident. Let $\phi_1$ and $\phi_2$ be the substitutions
$$
\begin{aligned}
&(x_{m-1},x_m,x_{m+1},\dots,x_k)\mapsto (x_mx_{m+1}x_{m-1}x_m,x_{m+1},x_{m+2},\dots,x_{k+1}) \\
\text{ and }\ &(x_{m-1},x_m,x_{m+1},\dots,x_{k-1},x_k,y_m)\mapsto (xx_{m-1},1,1,\dots,1,x,x_m),
\end{aligned}
$$
respectively.
Then the identities $\phi_1(\delta_k^m)$ and $\phi_2(\delta_k^m)$ are equivalent modulo~\eqref{xsxt=xsxtx} to $\delta_{k+1}^m$ and $\varepsilon_m$, respectively.
Therefore, $\mathbf P\{\delta_k^m\}\subseteq \mathbf P\{\delta_{k+1}^m\}$ and $\mathbf P\{\delta_k^m\}\subseteq \mathbf P\{\varepsilon_m\}$.
If $1\le m<k$ and $\phi_3$ is the substitution
$$
(x_{m-1},y_m,y_{m+1})\mapsto (y_{m+1}x_{m-1},y_{m+1},y_{m+2}),
$$
then the identity $\phi_3(\delta_k^m)$ is equivalent modulo~\eqref{xsxt=xsxtx} to $\delta_k^{m+1}$ and so $\mathbf P\{\delta_k^m\}\subseteq \mathbf P\{\delta_k^{m+1}\}$.
To show that the inclusions are proper, it suffices to note that for $1\le m\le k$, $\mathbf J_k^m\subseteq \mathbf P\{\delta_k^m\}\subseteq \mathbf P\{\varepsilon_k\}$ but $\mathbf J_k^1\nsubseteq \mathbf P\{\gamma_k\}$, $\mathbf J_{k+1}^m\nsubseteq \mathbf P\{\delta_k^m\}$ and $\mathbf J_k^{m+1}\nsubseteq \mathbf P\{\delta_k^m\}$ whenever $m<k$.

(iii) Evidently, $\mathbf P\{\varepsilon_0\} \subseteq \mathbf P\{\varepsilon_1\} \subseteq \cdots \subseteq \mathbf P\{\varepsilon_k\} \subseteq \cdots \subseteq \mathbf P\{\eqref{xyxy=xxyy}\}$.
It is routine to verify that $\mathbf P\{\varepsilon_k\}$ violates $\varepsilon_{k-1}$ for any $k \in \mathbb N$.
Therefore, the inclusions are proper.
\end{proof}

The \textit{head} of a non-empty word $\mathbf w$, denoted by $h(\mathbf w)$, is the the first letter of $\mathbf w$.

The following lemma generalizes Gusev and Vernikov \cite[Lemma~6.6(i),(ii)]{Gusev-Vernikov-18}.

\begin{lemma}
\label{L: u'abu''=u'bau''}
Let $\mathbf V$ be a subvariety of $\mathbf P$ and $\mathbf u$ be a word.
Further, let $\mathbf u=\mathbf u' \mathbf c\mathbf u''$ for some $\mathbf u',\mathbf u'' ,\mathbf c\in \mathfrak A^\ast$ such that $\mathbf c \in\{xy,yx\}$, $x \notin\con(\mathbf u')$ and $y \in \simple(\mathbf u')$.
Suppose that one of the following holds:
\begin{itemize}
\item[\textup{(i)}] $\mathbf V$ satisfies $\gamma_k$, $D(\mathbf u,x) = k$ and $D(\mathbf u,y) = 1$;
\item[\textup{(ii)}] $\mathbf V$ satisfies $\delta_k^m$, $D(\mathbf u,x) = k$ and $D(\mathbf u,y) = m+1$;
\item[\textup{(iii)}] $\mathbf V$ satisfies $\varepsilon_k$, $D(\mathbf u,x)=\infty$ and $D(\mathbf u,y)=k+1$;
\item[\textup{(iv)}] $\mathbf V$ satisfies~\eqref{xyxy=xxyy}, $D(\mathbf u,x)=D(\mathbf u,y)=\infty$.
\end{itemize}
Then $\mathbf V$ satisfies the identity $\mathbf u\approx\mathbf w$, where $\mathbf w$ is obtained from $\mathbf u$ by swapping of the first occurrence of $x$ and the second occurrence of $y$.
\end{lemma}

\begin{proof}
We consider only the case when $\mathbf c =xy$ since the case $\mathbf c = yx$ is similar.
We also omit the proof of Part~(i) because it is very similar to (and in fact simpler than) that of Part~(ii).

\smallskip

(ii) Put $x_k=x$ and $y_{m+1}=y$. It follows from Corollary~\ref{C: form of the word} that there exist letters $x_0,x_1,\dots, x_{k-1}$ such that $D(\mathbf u,x_s)=s$ for any $0\le s<k$ and
$$
\mathbf u=\mathbf v_{2k+2}y_{m+1}\mathbf v_{2k+1}x_ky_{m+1}\mathbf p,
$$
where
$$
\mathbf p = \mathbf v_{2k}x_{k-1}\mathbf v_{2k-1}x_k\mathbf v_{2k-2}x_{k-2}\mathbf v_{2k-3}x_{k-1}\cdots\mathbf v_4x_1\mathbf v_3x_2\mathbf v_2x_0\mathbf v_1x_1\mathbf v_0
$$
for some $\mathbf v_0,\mathbf v_1,\dots,\mathbf v_{2k+2} \in \mathfrak A^\ast$.
Clearly, for any $s=1,2,\dots,k-1$, there exist $\mathbf v_{2s}', \mathbf v_{2s}'' \in \mathfrak A^\ast$ such that $\mathbf v_{2s} = \mathbf v_{2s}'\mathbf v_{2s}''$, where $\mathbf v_{2s}'$ does not contain any \mbox{$(s-1)$}-divider of $\mathbf u$, while either $\mathbf v_{2s}''=\lambda$ or $h(\mathbf v_{2s}'')$ is an \mbox{$(s-1)$}-divider of $\mathbf u$.
Put $\mathbf v_{2k}''=\mathbf v_{2k}$ and $\mathbf v_0'=\mathbf v_0$.
Let
$$
\mathbf q_s = \mathbf v_{2s+2}''x_s\mathbf v_{2s+1}x_{s+1}\mathbf v_{2s}'
$$
for any $s=0,1,\dots,k-1$.
Let $\phi_1$ be the substitution
$$
(x_0,x_1,\dots,x_{k-1},y_m,y_{m+1})\mapsto (\mathbf q_0,\mathbf q_1,\dots,\mathbf q_{k-1},\mathbf v_{2k+1},y_{m+1}).
$$
Then the identity $\phi_1(\delta_k^m)$ coincides with the identity
\begin{equation}
\label{to delta_k^m}
y_{m+1}\mathbf v_{2k+1}x_ky_{m+1}\mathbf q \approx \mathbf v_{2k+2}y_{m+1}\mathbf v_{2k+1}y_{m+1}x_k\mathbf q
\end{equation}
where
$$
\mathbf q = \mathbf q_{k-1}x_k\mathbf q_{k-2}\mathbf q_{k-1}\cdots \mathbf q_{m-1}\mathbf q_m\mathbf v_{2k+1}\mathbf q_{m-2}\mathbf q_{m-1}\cdots \mathbf q_1\mathbf q_2\mathbf q_0\mathbf q_1.
$$

We note that $\mathbf p = \mathbf q_{k-1}\mathbf q_{k-2}\cdots \mathbf q_0$.
Since $D(\mathbf u, x_s)=s$ and $D(\mathbf u, x_{s+1})=s+1$, Lemma~\ref{L: k-divider and depth} implies that $x_s$ and $x_{s+1}$ are not \mbox{$(s-1)$}-dividers of $\mathbf u$.
If one of the subwords $\mathbf v_{2s+2}''$ or $\mathbf v_{2s+1}$ of $\mathbf u$ contains some \mbox{$(s-1)$}-divider $d$ of $\mathbf u$, then $({_{1\mathbf u}x_{s+1}})<({_{1\mathbf u}d})<({_{2\mathbf u}x_{s+1}})$ and so $D(\mathbf u,x_{s+1})\le s$ contradicting $D(\mathbf u, x_{s+1})=s+1$.
Therefore, $\mathbf v_{2s+2}''$ and $\mathbf v_{2s+1}$ do not contain any \mbox{$(s-1)$}-divider of $\mathbf u$.
Finally, the subword $\mathbf v_{2s}'$ of $\mathbf u$ does not contain any \mbox{$(s-1)$}-divider of $\mathbf u$ by its definition.
We see that the subword $\mathbf q_s$ of $\mathbf u$ does not contain \mbox{$(s-1)$}-dividers of $\mathbf u$ for any $s=1,2,\dots,k-1$.

Let $z$ be a letter such that ${_{1\mathbf u}z}$ occurs in the subword $\mathbf v_{2k+1}$ of $\mathbf u$.
Since $D(\mathbf u,y_{m+1})=m+1>1$, we have $z \in \mul(\mathbf u)$.
We are going to verify that $z \in \mul(\mathbf v_{2k+1}\mathbf q_{k-1}\cdots\mathbf q_m\mathbf q_{m-1})$.
If $m=1$, then this claim is evident because $\mathbf v_{2k+1}\mathbf q_{k-1}\cdots\mathbf q_m\mathbf q_{m-1}$ is a suffix of $\mathbf u$ in this case.
So, we may assume that $m>1$.
Let $z' = h(\mathbf q_{m-2})$.
Since $z'=h(\mathbf v_{2m-2}'')$ or $z'=x_{m-2}$, the letter $z'$ is an \mbox{$(m-2)$}-divider of $\mathbf u$.
If $({_{1\mathbf u}z'})<({_{2\mathbf u}z})$, then $h_1^{m-2}(\mathbf u,z)\ne h_2^{m-2}(\mathbf u,z)$ and so $D(\mathbf u, z) \le m-1$, whence $z$ is an \mbox{$(m-1)$}-divider of $\mathbf u$ by Lemma~\ref{L: k-divider and depth}.
But this contradicts the fact that $D(\mathbf u,y_{m+1})=m+1$ because $({_{1\mathbf u}y_{m+1}})<({_{1\mathbf u}z})<({_{2\mathbf u}y_{m+1}})$.
Therefore, $({_{2\mathbf u}z})<({_{1\mathbf u}z'})$.
We see that if a letter occurs in $\mathbf v_{2k+1}$, then the subword $\mathbf v_{2k+1}x_ky_{m+1}\mathbf q_{k-1}\cdots\mathbf q_m\mathbf q_{m-1}$ of $\mathbf u$ contains some non-first occurrence of this letter in $\mathbf u$.

Let $z$ be a letter such that ${_{1\mathbf u}z}$ occurs in the subword $\mathbf q_s$ of $\mathbf u$ for some $s \in \{1,2,\dots,k-1\}$.
We are going to verify that $z \in \mul(\mathbf q_s\mathbf q_{s-1})$.
If $s=1$, then this claim is evident because $\mathbf q_1$ does not contain $0$-dividers of $\mathbf u$ and $\mathbf q_1\mathbf q_0$ is a suffix of $\mathbf u$.
So, we may assume that $s>1$.
Clearly, either $z=x_s$ or $z \in \con(\mathbf v_{2s+2}''\mathbf v_{2s+1}\mathbf v_{2s}')$.
Evidently, if $z=x_s$, then $z \in \con(\mathbf q_{s-1})$.
Suppose now that $z \in \con(\mathbf v_{2s+2}''\mathbf v_{2s+1}\mathbf v_{2s}')$.
Since the subword $\mathbf q_s$ of $\mathbf u$ does not contain \mbox{$(s-1)$}-dividers of $\mathbf u$, this subword does not contain $0$-dividers of $\mathbf u$ as well by Lemma~\ref{L: k-divider and depth}.
Therefore, the letter $z$ is multiple in $\mathbf u$.
Let $z' = h(\mathbf q_{s-2})$.
Since $z'=h(\mathbf v_{2s-2}'')$ or $z'=x_{s-2}$, the letter $z'$ is an \mbox{$(s-2)$}-divider of $\mathbf u$.
If $({_{1\mathbf u}z'})<({_{2\mathbf u}z})$, then $h_1^{s-2}(\mathbf u,z)\ne h_2^{s-2}(\mathbf u,z)$ and so $D(\mathbf u, z) \le s-1$, whence $z$ is an \mbox{$(s-1)$}-divider of $\mathbf u$  by Lemma~\ref{L: k-divider and depth}.
But $\mathbf q_s$ does not contain \mbox{$(s-1)$}-dividers of $\mathbf u$. Therefore, $({_{2\mathbf u}z})<({_{1\mathbf u}z'})$.
This means that ${_{2\mathbf u}z}$ occurs in the subword $\mathbf q_s\mathbf q_{s-1}$ of $\mathbf u$.
We see that if a letter occurs in $\mathbf q_s$, then the subword $\mathbf q_s\mathbf q_{s-1}$ of $\mathbf u$ contains some non-first occurrence of this letter in $\mathbf u$.

Finally, we recall that ${_{2\mathbf u}x_k}$ occurs in the subword $\mathbf q_{k-1}$ of $\mathbf u$.

In view of the previous three paragraphs, the identity~\eqref{xsxt=xsxtx} implies
$$
\mathbf u = \mathbf v_{2k+2}y_{m+1}\mathbf v_{2k+1}x_ky_{m+1}\mathbf p \approx \mathbf v_{2k+2}y_{m+1}\mathbf v_{2k+1}x_ky_{m+1}\mathbf q.
$$
By a similar argument we can show that the identity~\eqref{xsxt=xsxtx} implies
$$
\mathbf v_{2k+2}y_{m+1}\mathbf v_{2k+1}y_{m+1}x_k\mathbf p \approx \mathbf v_{2k+2}y_{m+1}\mathbf v_{2k+1}y_{m+1}x_k\mathbf q.
$$
These identities together with~\eqref{to delta_k^m} imply the identity $\mathbf u \approx \mathbf u'yx\mathbf u''$. Thus, it is satisfied by $\mathbf V$.

\smallskip

(iii) There are two cases.

\smallskip

\textit{Case }1: $h(\mathbf u'')=x$.
If $k=0$, then $\varepsilon_k$ is nothing but the identity $y_1x_0xy_1x \approx y_1x_0y_1x^2$. Since $y \in \con(\mathbf u')$, this identity implies $\mathbf u \approx \mathbf u'yx\mathbf u''$.
So, we may assume that $k>0$. Put $y_{k+1}=y$. In view of Corollary~\ref{C: form of the word}, there are letters $x_0,x_1,\dots, x_k$ such that $D(\mathbf u,x_s)=s$ for any $0\le s\le k$ and
$$
\mathbf u=\mathbf v_{2k+3}y_{k+1}\mathbf v_{2k+2}x_k\mathbf v_{2k+1}xy_{k+1}x\mathbf p,
$$
where
$$
\mathbf p = \mathbf v_{2k}x_{k-1}\mathbf v_{2k-1}x_k\cdots\mathbf v_4x_1\mathbf v_3x_2\mathbf v_2x_0\mathbf v_1x_1\mathbf v_0
$$
for some $\mathbf v_0,\mathbf v_1,\dots,\mathbf v_{2k+3} \in \mathfrak A^\ast$.
Clearly, for any $s=1,2,\dots,k-1$, there exist $\mathbf v_{2s}', \mathbf v_{2s}'' \in \mathfrak A^\ast$ such that $\mathbf v_{2s} = \mathbf v_{2s}'\mathbf v_{2s}''$, where $\mathbf v_{2s}'$ does not contain any \mbox{$(s-1)$}-divider of $\mathbf u$, while either $\mathbf v_{2s}''=\lambda$ or $h(\mathbf v_{2s}'')$ is an \mbox{$(s-1)$}-divider of $\mathbf u$.
Put $\mathbf v_{2k}''=\mathbf v_{2k}$ and $\mathbf v_0'=\mathbf v_0$.
Let
$$
\mathbf q_s = \mathbf v_{2s+2}''x_s\mathbf v_{2s+1}x_{s+1}\mathbf v_{2s}'
$$
for any $s=0,1,\dots,k-1$.
Let $\phi_2$ be the substitution
$$
(x_0,x_1,\dots,x_{k-1},x_k)\mapsto (\mathbf q_0,\mathbf q_1,\dots,\mathbf q_{k-1},\mathbf v_{2k+2}x_k\mathbf v_{2k+1}).
$$
Then the identity $\phi_2(\varepsilon_k)$ coincides with the identity
\begin{equation}
\label{to epsilon_k}
y_{k+1}\mathbf v_{2k+2}x_k\mathbf v_{2k+1}xy_{k+1}x\mathbf q \approx y_{k+1}\mathbf v_{2k+2}x_k\mathbf v_{2k+1}y_{k+1}x^2\mathbf q
\end{equation}
where
$$
\mathbf q = \mathbf q_{k-1}\mathbf v_{2k+2}x_k\mathbf v_{2k+1}\mathbf q_{k-2}\mathbf q_{k-1}\cdots \mathbf q_1\mathbf q_2\mathbf q_0\mathbf q_1.
$$
We note that $\mathbf p = \mathbf q_{k-1}\mathbf q_{k-2}\cdots \mathbf q_0$.
By the same arguments as in Part~(ii) one can show that
\begin{itemize}
\item if a letter occurs in the word $\mathbf v_{2k+2}x_k\mathbf v_{2k+1}$, then the subword $\mathbf v_{2k+2}x_k\mathbf v_{2k+1}xy_{k+1}x\mathbf q_{k-1}$ of $\mathbf u$ contains some non-first occurrence of this letter in $\mathbf u$;{\sloppy

}
\item if a letter occurs in $\mathbf q_s$, then the subword $\mathbf q_s\mathbf q_{s-1}$ of $\mathbf u$ contains some non-first occurrence of this letter in $\mathbf u$.
\end{itemize}
Then the identity~\eqref{xsxt=xsxtx} implies
$$
\mathbf u = \mathbf v_{2k+3}y_{k+1}\mathbf v_{2k+2}x_k\mathbf v_{2k+1}xy_{k+1}x\mathbf p \approx \mathbf v_{2k+3}y_{k+1}\mathbf v_{2k+2}x_k\mathbf v_{2k+1}xy_{k+1}x\mathbf q.
$$
By a similar argument one can show that the identity~\eqref{xsxt=xsxtx} implies
$$
\mathbf v_{2k+3}y_{k+1}\mathbf v_{2k+2}x_k\mathbf v_{2k+1}y_{k+1}x^2\mathbf p \approx \mathbf v_{2k+3}y_{k+1}\mathbf v_{2k+2}x_k\mathbf v_{2k+1}y_{k+1}x^2\mathbf q.
$$
These identities together with~\eqref{to epsilon_k} imply the identity $\mathbf u \approx \mathbf u'yx\mathbf u''$.
Thus, it is satisfied by $\mathbf V$.

\smallskip

\textit{Case }2: $h(\mathbf u'') \ne x$.
In view of Remark~\ref{R: max decomposition}, there is a number $r$ such that the $r$-decomposition of $\mathbf u$ is maximal.
Since $D(\mathbf u,x)=\infty$, the subword ${_{1\mathbf u}}x\,{_{2\mathbf u}}y$ is contained in some $r$-block $\mathbf a$ of $\mathbf u$.
Then $\mathbf a = \mathbf a_1\,xy\,\mathbf a_2$ and $\mathbf u = \mathbf u_1\mathbf a \mathbf u_2$ for some $\mathbf a_1,\mathbf a_2,\mathbf u_1,\mathbf u_2 \in \mathfrak A^\ast$.

Let $\mathbf u^\ast=\mathbf u_1\mathbf a_1\,xyx\,\mathbf a_2 \mathbf u_2$.
The word $\mathbf u^\ast$ differs from $\mathbf u$ by the occurrences of $x$ only.
Then, since $D(\mathbf u,x)=D(\mathbf u^\ast,x)=\infty$, the identity $\mathbf u \approx  \mathbf u^\ast$ is $s$-well-balanced for any $s$.
Since~\eqref{xyxy=xxyy} is a consequence of $\varepsilon_k$ by Lemma~\ref{L: gamma_k, delta_k^k,epsilon_k}(iii), it follows from Corollary~\ref{C: block = ini^2(block)} that the identities
$$
\mathbf u \approx \mathbf u_1\ini^2(\mathbf a_1\,xy\,\mathbf a_2)\,\mathbf u_2\ \text{ and } \
\mathbf u^\ast \approx \mathbf u_1\ini^2(\mathbf a_1\,xyx\,\mathbf a_2)\,\mathbf u_2
$$
hold in $\mathbf V$.
Clearly, $\ini^2(\mathbf a_1xy\mathbf a_2) = \ini^2(\mathbf a_1xyx\mathbf a_2)$ because $x \in \con(\mathbf a_2) \setminus \con(\mathbf a_1)$.
This implies that $\mathbf V$ satisfies $\mathbf u \approx \mathbf u^\ast$.
By a similar argument we can show that $\mathbf u_1\mathbf a_1\,yx\,\mathbf a_2 \mathbf u_2 \approx \mathbf u_1\mathbf a_1\,yx^2\,\mathbf a_2 \mathbf u_2$ is satisfied by $\mathbf V$.
Then $D(\mathbf u^\ast,y)=k+1$ because the identity $\mathbf u \approx  \mathbf u^\ast$ is $s$-well-balanced for any $s$.
In view of Case~1, the identity $\mathbf u^\ast \approx \mathbf u_1\mathbf a_1\,yx^2\,\mathbf a_2 \mathbf u_2$ is satisfied by $\mathbf V$.
Therefore, $\mathbf u \approx \mathbf u'yx\mathbf u''$ holds in $\mathbf V$, and we are done.

\smallskip

(iv) In view of Remark~\ref{R: max decomposition}, there is a number $r$ such that the $r$-decomposition of $\mathbf u$ is maximal.
Since $D(\mathbf u,x)=D(\mathbf u,y)=\infty$, the subword ${_{1\mathbf u}}x\,{_{2\mathbf u}}y$ is contained in some $r$-block $\mathbf a$ of $\mathbf u$ and $\mathbf a = \mathbf a_1y\mathbf a_2\,xy\,\mathbf a_3x\mathbf a_4$ for some $\mathbf a_1,\mathbf a_2,\mathbf a_3,\mathbf a_4\in \mathfrak A^\ast$.
Evidently,
$$
\ini^2(\mathbf a_1y\mathbf a_2\,xy\,\mathbf a_3x\mathbf a_4) = \ini^2(\mathbf a_1y\mathbf a_2\,yx\,\mathbf a_3x\mathbf a_4).
$$
This fact and Corollary~\ref{C: block = ini^2(block)} imply that $\mathbf u \approx \mathbf u'yx\mathbf u''$ is satisfied by $\mathbf V$.
\end{proof}

\begin{corollary}
\label{C: u'abu''=u'bau'' kappa_k}
Let $\mathbf V$ be a subvariety of $\mathbf P\{\kappa_k\}$ and $\mathbf u$ be a word.
Further, let $\mathbf u=\mathbf u'\,xy\,\mathbf u''$ for some $\mathbf u',\mathbf u'' \in \mathfrak A^\ast$ such that $x,y \in \simple(\mathbf u')$.
Suppose that either $D(\mathbf u,x)> k$ or $D(\mathbf u,y)> k$.
Then $\mathbf V$ satisfies the identity $\mathbf u\approx\mathbf u'yx\mathbf u''$.
\end{corollary}

\begin{proof}
We may assume without loss of generality that $({_{1\mathbf u}x})<({_{1\mathbf u}y})$.
Then $\mathbf u'=\mathbf u_1x\mathbf u_2y\mathbf u_3$ for some $\mathbf u_1,\mathbf u_2,\mathbf u_3\in\mathfrak A^\ast$.
Evidently, $D(\mathbf u,x) \le D(\mathbf u,y)$.
It follows that $k < D(\mathbf u,y)$.
Let $\phi$ be a substitution defined as follows: $\phi(c)=c^2$ for any $c\in\con(\mathbf u_1)$ and $\phi(c)=c$ otherwise.
Put $\mathbf w=\phi(\mathbf u_2)y\phi(\mathbf u_3)xyx\phi(\mathbf u'')$.
We are going to verify that $D(\mathbf u,a) \le D(\mathbf w,a)$ for any $a \in \con(\mathbf w)$.
This fact is evident whenever $D(\mathbf w,a)=\infty$.
So, we may assume that $D(\mathbf w,a)=k<\infty$.
We will use induction on $k$.

\smallskip

\textit{Induction base}: $k=0$.
Then $a \in \simple(\mathbf w)$.
The definition of $\mathbf w$ implies that $a \notin \simple(\mathbf u_1)$ and so $a \in \simple(\mathbf u)$.
Therefore, $D(\mathbf u,a)=0$, and we are done.

\smallskip

\textit{Induction step}: $k>0$.
Then $a \in \mul(\mathbf w)$.
Let $b = h_2^{k-1}(\mathbf w, a)$.
In view of Lemma~\ref{L: h_2^{k-1}}, $D(\mathbf w,b)=k-1$ and $({_{1\mathbf w}a})<({_{1\mathbf w}b})<({_{2\mathbf w}a})$.
By the induction assumption $D(\mathbf u,b) \le k-1$.
The definition of $\mathbf w$ implies that $b \notin \simple(\mathbf u_1)$ and $({_{1\mathbf u}a})<({_{1\mathbf u}b})<({_{2\mathbf u}a})$.
Therefore, $D(\mathbf u,a) \le k$, and we are done.

\smallskip

We have proved that $D(\mathbf u,a) \le D(\mathbf w,a)$ for any $a \in \con(\mathbf w)$.
In particular, $k < D(\mathbf w,y)$.
According to Lemmas~\ref{L: kappa_k and delta_k^k} and~\ref{L: gamma_k, delta_k^k,epsilon_k}(ii),(iii), the variety $\mathbf P\{\kappa_k\}$ satisfies the identities~\eqref{xyxy=xxyy} and $\varepsilon_r$ for any $r \ge k$.
Hence if $k < D(\mathbf w, y)<\infty$, then
\begin{equation}
\label{w=phi(u_2)y phi(u_3) yxx phi(u'')}
\mathbf w=\phi(\mathbf u_2)y\phi(\mathbf u_3)xyx\phi(\mathbf u'') \approx \phi(\mathbf u_2)y\phi(\mathbf u_3)yx^2\phi(\mathbf u'')
\end{equation}
is satisfied by $\mathbf P\{\kappa_k\}$ by Lemma~\ref{L: u'abu''=u'bau''}(iii), and if $D(\mathbf w, y)=\infty$, then~\eqref{w=phi(u_2)y phi(u_3) yxx phi(u'')} holds in $\mathbf P\{\kappa_k\}$ by Lemma~\ref{L: u'abu''=u'bau''}(iv).
It remains to note that $\mathbf P\{\kappa_k\}$ satisfies
$$
\mathbf u \stackrel{\eqref{xsxt=xsxtx}}\approx \mathbf u_1x\mathbf w \stackrel{\eqref{w=phi(u_2)y phi(u_3) yxx phi(u'')}}\approx \mathbf u_1x\phi(\mathbf u_2)y\phi(\mathbf u_3)yx^2\phi(\mathbf u'') \stackrel{\eqref{xsxt=xsxtx}}\approx \mathbf u'yx\mathbf u''
$$
and so $\mathbf u\approx\mathbf u'yx\mathbf u''$.
\end{proof}

\begin{lemma}
\label{L: ab_k=ini(a)ini(a_{x_k})b_k}
Suppose that $k \in \mathbb N$ and $\mathbf a \in \mathfrak A^\ast$ are such that $x_0,x_1,\dots,x_{k-1} \notin\con(\mathbf a)$ and $\simple(\mathbf a)=\{x_k\}$.
Then the variety $\mathbf P\{\kappa_k\}$ satisfies the identities
$$
\begin{aligned}
\mathbf a\mathbf b_k \approx{}& \ini(\mathbf a)\ini(\mathbf a_{x_k})\mathbf b_k,\\ y_m\mathbf a \mathbf b_{k,m+1}x_{m-1}\mathbf c \mathbf b_{m-1} \approx{}& y_m\ini(\mathbf a)\ini(\mathbf a_{x_k})\mathbf b_{k,m+1}x_{m-1}\mathbf c \mathbf b_{m-1},
\end{aligned}
$$
where $\mathbf c\in\{x_my_m,y_mx_m\}$.
\end{lemma}

\begin{proof}
It is routine to verify that $D(\mathbf a\mathbf b_k, x_s)=s$ for any $0\le s\le k-1$.
Let $z$ be a letter from $\con(\mathbf a)$ with minimal depth in $\mathbf a\mathbf b_k$, which we denote by $d$.
Clearly, $d < \infty$ because $h_1^{k-1}(\mathbf a\mathbf b_k,x_k)\ne x_{k-1}=h_2^{k-1}(\mathbf a\mathbf b_k,x_k)$ and so $D(\mathbf a\mathbf b_k,x_k) \le k < \infty$.
Then, since $\con(\mathbf a)\subseteq \mul(\mathbf a\mathbf b_k)$, we have $d > 0$ and there is a \mbox{$(d-1)$}-divider $t$ between the first and the second occurrences of $z$ in $\mathbf a\mathbf b_k$.
In view of the choice of $z$, we have $t \notin\con(\mathbf a)$.
It follows that $z \in \simple(\mathbf a)$.
Hence $z=x_k$.
According to Lemma~\ref{L: k-divider and depth}, $x_{k-1}$ is a \mbox{$(k-1)$}-divider of the word $\mathbf a\mathbf b_k$.
Now Lemma~\ref{L: h_2^{k-1}}(ii) applies and we conclude that $D(\mathbf a\mathbf b_k, z) = k$.
In view of the choice of $z$, we have $D(\mathbf a\mathbf b_k, a) > k$ for any $a \in \con(\mathbf a)\setminus\{x_k\}$ and $D(\mathbf a\mathbf b_k, x_k) = k$.
Let $\con(\mathbf a)=\{a_1,a_2,\dots,a_r\}$.
We may assume without loss of generality that
$$
\mathbf a =\ \stackrel{(1)}{a_1} \mathbf a_1 \stackrel{(1)}{a_2} \mathbf a_2 \cdots \stackrel{(1)}{a_r} \mathbf a_r
$$
for some $\mathbf a_1,\mathbf a_2,\dots,\mathbf a_r \in \mathfrak A^\ast$.
The identity~\eqref{xsxt=xsxtx} allows us to remove any third or subsequent occurrence of letters from the subwords $\mathbf a_1,\mathbf a_2,\dots,\mathbf a_r$ of $\mathbf a$.
So, we may assume that these subwords contain second occurrences of letters only.
According to Lemmas~\ref{L: kappa_k and delta_k^k} and~\ref{L: gamma_k, delta_k^k,epsilon_k}(ii),(iii), the identities~\eqref{xyxy=xxyy}, $\delta_p^q$ and $\varepsilon_p$ are satisfied by $\mathbf P\{\kappa_k\}$ for all $p \ge q \ge k$.
Then, since $D(\mathbf a\mathbf b_k, x_k) = k$ and $D(\mathbf a\mathbf b_k, a) > k$ for any $a \in \con(\mathbf a)\setminus\{x_k\}$, the identities
$$
\begin{aligned}
\mathbf a \mathbf b_k = a_1 \mathbf a_1 a_2 \mathbf a_2 \cdots a_r \mathbf a_r\mathbf b_k
\approx a_1 a_2 \cdots a_r\mathbf a_1\mathbf a_2\cdots \mathbf a_r\mathbf b_k
 =\ini(\mathbf a)\mathbf a_1\mathbf a_2\cdots \mathbf a_r\mathbf b_k
\end{aligned}
$$
hold in $\mathbf P\{\kappa_k\}$ by Parts~(ii)--(iv) of Lemma~\ref{L: u'abu''=u'bau''}.
According to Corollary~\ref{C: u'abu''=u'bau'' kappa_k}, $\mathbf P\{\kappa_k\}$ satisfies $\ini(\mathbf a)\mathbf a_1\mathbf a_2\cdots \mathbf a_{m+1}\mathbf b_k \approx \ini(\mathbf a)\ini(\mathbf a_{x_k})\mathbf b_k$ and so $\mathbf a\mathbf b_k \approx \ini(\mathbf a)\ini(\mathbf a_{x_k})\mathbf b_k$.
Substituting~$x_{m-1}\mathbf c$ for $x_{m-1}$ in the last identity, we obtain the identity
\begin{equation}
\label{x_{m-1}c to x_{m-1}}
\mathbf a\mathbf b_{k,m+1}x_{m-1}\mathbf cx_m\mathbf d \approx \ini(\mathbf a)\ini(\mathbf a_{x_k})\mathbf b_{k,m+1}x_{m-1}\mathbf cx_m\mathbf d
\end{equation}
where
$$
\mathbf d=
\begin{cases}
\mathbf b_0&\text{if }m=1,\\
\mathbf b_{m-1,m-1}\mathbf c \mathbf b_{m-2}&\text{if }m>1.
\end{cases}
$$
Then $\mathbf P\{\kappa_k\}$ satisfies
$$
\begin{aligned}
y_m\mathbf a\mathbf b_{k,m+1}x_{m-1}\mathbf c \mathbf b_{m-1} \stackrel{\eqref{xsxt=xsxtx}}\approx{}& y_m\mathbf a\mathbf b_{k,m+1}x_{m-1}\mathbf cx_m\mathbf d\\
\stackrel{\eqref{x_{m-1}c to x_{m-1}}}\approx{}& y_m\ini(\mathbf a)\ini(\mathbf a_{x_k})\mathbf b_{k,m+1}x_{m-1}\mathbf cx_m\mathbf d\\
\stackrel{\eqref{xsxt=xsxtx}}\approx{}& y_m\ini(\mathbf a)\ini(\mathbf a_{x_k})\mathbf b_{k,m+1}x_{m-1}\mathbf c \mathbf b_{m-1},
\end{aligned}
$$
and we are done.
\end{proof}

Let
$$
\Phi = \{\eqref{xsxt=xsxtx},\,\eqref{xyxy=xxyy},\,\gamma_k,\,\delta_k^m,\,\varepsilon_{k-1} \mid k \in \mathbb N, \ 1 \le m \le k \}.
$$
Given a monoid variety $\mathbf V$ and a system of identities $\Sigma$ we use $\Sigma(\mathbf V)$ to denote the set of those identities from $\Sigma$, which hold in $\mathbf V$.

\begin{proposition}
\label{P: 1st and 2nd occurrences}
Let $\mathbf V$ be a monoid variety from the interval $[\mathbf{LRB}\vee\mathbf F_1,\mathbf P]$, $\mathbf u \approx \mathbf v$ be an identity of $\mathbf V$ and $\{_i x, {_j y}\}$ be a critical pair in $\mathbf u \approx \mathbf v$. If $\{i,j\}=\{1,2\}$, then the critical pair $\{_i x, {_j y}\}$ is $\Delta$-removable for some $\Delta\subseteq\Phi(\mathbf V)$.
\end{proposition}

\begin{proof}
It suffices to assume that $(i,j)=(1,2)$ since the case when $(i,j)=(2,1)$ is similar.
Let $\mathbf w$ denote the  word obtained from $\mathbf u$ by replacing ${_{1\mathbf u} x}\,{_{2\mathbf u} y}$ with ${_{2\mathbf u} y}\,{_{1\mathbf u} x}$.
Since $\{{_1 x}, {_2 y}\}$ is a critical pair in $\mathbf u \approx \mathbf v$, the identity $\mathbf u(x,y) \approx \mathbf v(x,y)$ is equivalent modulo~\eqref{xsxt=xsxtx} to~\eqref{xyxy=xxyy}.
Therefore,~\eqref{xyxy=xxyy} is satisfied by $\mathbf V$.

Suppose that $D(\mathbf u,y)=\infty$.
Then $x$ is not an $r$-divider of $\mathbf u$ for any $r\ge 0$.
Therefore, $D(\mathbf u,x)=\infty$ by Lemma~\ref{L: k-divider and depth}.
Then $\mathbf u \approx \mathbf w$ follows from $\{\eqref{xsxt=xsxtx},\,\eqref{xyxy=xxyy}\} \subseteq \Phi(\mathbf V)$ by Lemma~\ref{L: u'abu''=u'bau''}(iv).
So, we may assume that $D(\mathbf u,y)<\infty$, say $D(\mathbf u,y)=s$ for some $s \in \mathbb N$.
Further, Lemma~\ref{L: 2-limited word} allows us to assume that the identity $\mathbf u \approx \mathbf v$ is 2-limited.

There are two cases.

\smallskip

\textit{Case }1: $\mathbf F \subseteq \mathbf V$.
In view of Remark~\ref{R: max decomposition}, there is a number $k$ such that the $k$-decompositions of $\mathbf u$ and $\mathbf v$ are maximal.
According to Corollary~\ref{C: (k-1)-well-balanced} and the inclusion $\mathbf F_{k+1} \subset \mathbf F$, the identity $\mathbf u \approx \mathbf v$ is $k$-well-balanced.
It follows that $x$ is not an $r$-divider of $\mathbf u$ for any $r\ge 0$.
Therefore, $D(\mathbf u,x)=\infty$ by Lemma~\ref{L: k-divider and depth}.

Now we apply Lemma~\ref{L: form of the identity} and conclude that there are letters $x_0,x_1,\dots, x_{s-1}$ such that $D(\mathbf u,x_r)=D(\mathbf v,x_r)=r$ for any $0\le r<s$ and the identity $\mathbf u \approx \mathbf v$ has the form
$$
\begin{aligned}
&\mathbf u_{2s+1}\stackrel{(1)}{y}\mathbf u_{2s}\stackrel{(1)}{x_{s-1}}\mathbf u_{2s-1}\stackrel{(2)}{y}\mathbf u_{2s-2}\stackrel{(1)}{x_{s-2}}\mathbf u_{2s-3}\stackrel{(2)}{x_{s-1}}\mathbf u_{2s-4}\stackrel{(1)}{x_{s-3}}\\
&\cdot\,\mathbf u_{2s-5}\stackrel{(2)}{x_{s-2}}\cdots\mathbf u_4\stackrel{(1)}{x_1}\mathbf u_3\stackrel{(2)}{x_2}\mathbf u_2\stackrel{(1)}{x_0}\mathbf u_1\stackrel{(2)}{x_1}\mathbf u_0\\
\approx{}&\mathbf v_{2s+1}\stackrel{(1)}{y}\mathbf v_{2s}\stackrel{(1)}{x_{s-1}}\mathbf v_{2s-1}\stackrel{(2)}{y}\mathbf v_{2s-2}\stackrel{(1)}{x_{s-2}}\mathbf v_{2s-3}\stackrel{(2)}{x_{s-1}}\mathbf v_{2s-4}\stackrel{(1)}{x_{s-3}}\\
&\cdot\,\mathbf v_{2s-5}\stackrel{(2)}{x_{s-2}}\cdots\mathbf v_4\stackrel{(1)}{x_1}\mathbf v_3\stackrel{(2)}{x_2}\mathbf v_2\stackrel{(1)}{x_0}\mathbf v_1\stackrel{(2)}{x_1}\mathbf v_0
\end{aligned}
$$
for some words $\mathbf u_0,\mathbf u_1,\dots,\mathbf u_{2s+1}$ and $\mathbf v_0,\mathbf v_1,\dots,\mathbf v_{2s+1}$.
Clearly, ${_{1\mathbf u}x}$ is the last letter of $\mathbf u_{2s-1}$.
Since $D(\mathbf u,x)=\infty$, we have $({_{2\mathbf u}x})<({_{1\mathbf u}x_{s-2}})$. It follows that ${_{2\mathbf u}x}$ occurs in $\mathbf u_{2s-2}$.
Further, $D(\mathbf v, x)=\infty$ by Corollary~\ref{C: (k-1)-well-balanced}.
Since $({_{2\mathbf v}y})<({_{1\mathbf v}x})$ and the identity $\mathbf u \approx \mathbf v$ is $k$-well-balanced, both ${_{1\mathbf v}x}$ and ${_{2\mathbf v}x}$ occur in $\mathbf v_{2s-2}$.
Then we substitute $y_s$ for $y$ in the identity
$$
\mathbf u(x_0,x_1,\dots,x_{s-1},x,y) \approx \mathbf v(x_0,x_1,\dots,x_{s-1},x,y)
$$
and obtain the identity $\varepsilon_{s-1}$.
Therefore, $\varepsilon_{s-1}$ is satisfied by $\mathbf V$.
Now we apply Lemma~\ref{L: u'abu''=u'bau''}(iii) and conclude that the identity $\mathbf u \approx \mathbf w$ follows from $\{\eqref{xsxt=xsxtx},\,\varepsilon_{s-1}\} \subseteq \Phi(\mathbf V)$.

\smallskip

\textit{Case }2: $\mathbf F \nsubseteq \mathbf V$.
In view of Lemma~\ref{L: L(F)}, there is $k \in \mathbb N$ such that $\mathbf F_{k+1} \nsubseteq \mathbf V$.
Let $k$ be the least number with such a property.
Then $\mathbf F_k \subseteq\mathbf V$.
According to Lemma~\ref{L: V does not contain F_k}, $\mathbf V$ satisfies $\kappa_k$.
Then $\delta_k^k$ holds in $\mathbf V$ by Lemma~\ref{L: kappa_k and delta_k^k}.

Suppose that $k<s$.
Then, by Parts~(ii) and~(iii) of Lemma~\ref{L: gamma_k, delta_k^k,epsilon_k}, $\mathbf V$ satisfies $\varepsilon_{s-1}$ and $\delta_r^s$ for any $r\ge s$.
Clearly, $D(\mathbf u,x)\ge k$ because $D(\mathbf u,y)\le k<s$ otherwise.
If $D(\mathbf u,x)=\infty$, then Lemma~\ref{L: u'abu''=u'bau''}(iii) applies and we conclude that the identity $\mathbf u \approx \mathbf w$ follows from $\{\eqref{xsxt=xsxtx},\,\varepsilon_{s-1}\} \subseteq \Phi(\mathbf V)$.
If $D(\mathbf u,x)=r$ for some $k \le r < \infty$, then, by Lemma~\ref{L: u'abu''=u'bau''}(ii), the identity $\mathbf u \approx \mathbf w$ follows from $\{\eqref{xsxt=xsxtx},\,\delta_r^s\} \subseteq \Phi(\mathbf V)$.
So, we may assume without loss of generality that $s \le k$.

If $D(\mathbf u,x)=\infty$, then since \eqref{sim(u)=sim(v) & mul(u)=mul(v)} and~\eqref{eq the same l-dividers} hold for all $\ell=1,2,\dots,k$ by Lemma~\ref{L: word problem F_k}, one can repeat literally arguments from the second paragraph of Case~1 and obtain that the identity $\mathbf u \approx \mathbf w$ follows from $\{\eqref{xsxt=xsxtx},\,\varepsilon_{s-1}\} \subseteq \Phi(\mathbf V)$.
So, we may assume without loss of generality that $D(\mathbf u,x)=p$ for some $p \in \mathbb N$.

Then by Corollary~\ref{C: form of the word}, there exist letters $x_0,x_1,\dots, x_{p-1}$ such that
\begin{align*}
\mathbf u={}&\mathbf u_{2p+1}\stackrel{(1)}{x}\mathbf u_{2p}\stackrel{(1)}{x_{p-1}}\mathbf u_{2s-1}\stackrel{(2)}{x}\mathbf u_{2p-2}\stackrel{(1)}{x_{p-2}}\mathbf u_{2p-3}\stackrel{(2)}{x_{p-1}}\mathbf u_{2p-4}\stackrel{(1)}{x_{p-3}}\\
&\cdot\,\mathbf u_{2p-5}\stackrel{(2)}{x_{p-2}}\cdots\mathbf u_4\stackrel{(1)}{x_1}\mathbf u_3\stackrel{(2)}{x_2}\mathbf u_2\stackrel{(1)}{x_0}\mathbf u_1\stackrel{(2)}{x_1}\mathbf u_0
\end{align*}
for some $\mathbf u_0,\mathbf u_1,\dots,\mathbf u_{2p+1} \in \mathfrak A^\ast$, with $D(\mathbf u,x_r)=r$ and $\mathbf u_{2r+2}$ not containing any $r$-dividers of $\mathbf u$ for any $r=0,1,\dots,k-1$.

Clearly, ${_{1\mathbf u}y}$ occurs in $\mathbf u_{2p+1}$, while ${_{2\mathbf u}y}$ is the first letter of $\mathbf u_{2p}$.
In view of Lemma~\ref{L: h_2^{k-1}}, there is a letter $y_{s-1}$ such that $D(\mathbf u, y_{s-1})=s-1$ and $({_{1\mathbf u}y})<({_{1\mathbf u}y_{s-1}})<({_{2\mathbf u}y})$.
Let $X=\{x_0,x_1,\dots,x_{p-1},x,y_{s-1},y\}$.

\smallskip

\textit{Subcase }2.1: $s=1$.
Then $y_{s-1}=y_0 \in \simple(\mathbf u)=\simple(\mathbf v)$ and
$$
\mathbf u(X) = yy_0xyx_{p-1}x\,\mathbf b_{p-1}.
$$

Let $p \le k$.
Since $({_{2\mathbf v}y})<({_{1\mathbf v}x})$ and the identity $\mathbf u \approx \mathbf v$ is \mbox{$(k-1)$}-well-balanced by Corollary~\ref{C: (k-1)-well-balanced}, we have $\mathbf v(X) =  yy_0yxx_{p-1}x\,\mathbf b_{p-1}$.
Let $\phi$ be the substitution $(x,y)\mapsto(x_p,y_s)$.
Then the identity $\phi(\mathbf u(X)) \approx \phi(\mathbf v(X))$ coincides with $\gamma_p$.

Now let $p > k$.
Since  the identity $\mathbf u \approx \mathbf v$ is \mbox{$(k-1)$}-well-balanced by Corollary~\ref{C: (k-1)-well-balanced}, we have $\mathbf v(X) = yy_0\,\mathbf a' \mathbf b_k$ for some word $\mathbf a'$ with $\simple(\mathbf a')=\{x_k,y\}$ and $\mul(\mathbf a')=\{x_{k+1},x_{k+2},\dots,x_{p-1},x\}$.
Now Lemma~\ref{L: word problem LRB} and the fact that $({_{2\mathbf v}y})<({_{1\mathbf v}x})$ imply that $\ini(\mathbf a') = yx x_{p-1} x_{p-2}\cdots x_k$.
It follows that $\mathbf a'=y\mathbf a$ for some $a$.
Since $\kappa_k$ holds in $\mathbf V$, Lemma~\ref{L: ab_k=ini(a)ini(a_{x_k})b_k} implies that $\mathbf V$ satisfies the identities
$$
\begin{aligned}
yy_0y\mathbf a \mathbf b_k \approx{}& yy_0y\ini(\mathbf a)\ini(\mathbf a_{x_k}) \mathbf b_k = yy_0yx\,x_{p-1}x_{p-2}\cdots x_k\,x_{p-1}x_{p-2}\cdots x_{k+1}\,\mathbf b_k\\
\approx{}& yy_0yxx_{p-1}x\mathbf b_{p-1,k+1}\mathbf b_k=yy_0yxx_{p-1}x\mathbf b_{p-1}
\end{aligned}
$$
and so the identity $\gamma_p$.

We see that $\gamma_p$ is satisfied by $\mathbf V$ in either case.
Now we apply Lemma~\ref{L: u'abu''=u'bau''}(i) and conclude that the identity $\mathbf u \approx \mathbf w$ follows from $\{\eqref{xsxt=xsxtx},\,\gamma_p\} \subseteq \Phi(\mathbf V)$.

\smallskip

\textit{Subcase }2.2: $s>1$.
Clearly, $p \ge s-1$ because $h_1^p(\mathbf u,y) \ne h_2^p(\mathbf u,y)$ and so $D(\mathbf u,y)<s$ otherwise.
Suppose that $p=s-1$.
Then $p \le k$.
In view of Corollary~\ref{C: (k-1)-well-balanced}, the identity $\mathbf u \approx \mathbf v$ is \mbox{$(k-1)$}-well-balanced.
Hence
$$
\begin{aligned}
\mathbf v={}&\mathbf v_{2p+1}\stackrel{(1)}{x}\mathbf v_{2p}\stackrel{(1)}{x_{p-1}}\mathbf v_{2p-1}\stackrel{(2)}{x}\mathbf v_{2p-2}\stackrel{(1)}{x_{p-2}}\mathbf v_{2p-3}\stackrel{(2)}{x_{p-1}}\mathbf v_{2s-4}\stackrel{(1)}{x_{p-3}}\\
&\cdot\,\mathbf v_{2p-5}\stackrel{(2)}{x_{p-2}}\cdots\mathbf v_4\stackrel{(1)}{x_1}\mathbf v_3\stackrel{(2)}{x_2}\mathbf v_2\stackrel{(1)}{x_0}\mathbf v_1\stackrel{(2)}{x_1}\mathbf v_0
\end{aligned}
$$
for some words $\mathbf v_0,\mathbf v_1,\dots,\mathbf v_{2p+1}$.
Let $\phi$ be the substitution $(x,y)\mapsto(x_p,x)$.
Then the identity
$$
\phi(\mathbf u(x_0,x_1,\dots,x_{p-1},x,y)) \approx \phi(\mathbf v(x_0,x_1,\dots,x_{p-1},x,y))
$$
is equal to $\kappa_{s-1}$.
Then $\mathbf V$ satisfies $\delta_{s-1}^{s-1}$ by Lemma~\ref{L: kappa_k and delta_k^k}.
Now we apply Lemma~\ref{L: u'abu''=u'bau''}(ii) and conclude that the identity $\mathbf u \approx \mathbf w$ follows from $\{\eqref{xsxt=xsxtx},\,\delta_{s-1}^{s-1}\} \subseteq \Phi(\mathbf V)$.
So, we may assume that $p > s-1$.
In particular, $x \ne y_{s-1}$.

Since $s>1$, the letter $y_{s-1}$ is in $\mul(\mathbf u)=\mul(\mathbf v)$.
Since $D(\mathbf u,y)=s$, there are no \mbox{$(s-2)$}-dividers of $\mathbf u$ between ${_{1\mathbf u}y}$ and ${_{2\mathbf u}y}$.
This fact and Lemma~\ref{L: 2nd occurence of y_m general} imply that ${_{2\mathbf u}y_{s-1}}$ occurs in either $\mathbf u_{2s-2}$ or $\mathbf u_{2s-3}$ or $\mathbf u_{2s-4}$.
Let $y_{s-2}=h_2^{s-1}(\mathbf u,y_{s-1})$.
Using Lemmas~\ref{L: h_2^{k-1}} and~\ref{L: if first then second}, one can show that ${_{1\mathbf u}y_{s-2}}$ follows ${_{2\mathbf u}x_s}$.
Since $\mathbf u_{2s-2}$ does not contain any \mbox{$(s-2)$}-dividers of $\mathbf u$, ${_{1\mathbf u}y_{s-2}}$ does not occur in $\mathbf u_{2s-2}$.
Therefore, ${_{2\mathbf u}y_{s-1}}$ occurs in either $\mathbf u_{2s-3}$ or $\mathbf u_{2s-4}$.
Since $s \le k$ and $\mathbf u \approx \mathbf v$ is \mbox{$(k-1)$}-well-balanced by Corollary~\ref{C: (k-1)-well-balanced}, the letter ${_{2\mathbf v}y_{s-1}}$ lies between ${_{2\mathbf v}x_{s-2}}$ and ${_{1\mathbf v}x_{s-3}}$.
This means that ${_{2\mathbf v}y_{s-1}}$ occurs in either $\mathbf v_{2s-3}$ or $\mathbf v_{2s-4}$.

Clearly,
$$
\mathbf u(X) = yy_{s-1}xyx_{p-1}x\,\mathbf b_{p-1,s}\,x_{s-2}\,\mathbf c\mathbf b_{s-2},
$$
where $\mathbf c \in\{x_{s-1}y_{s-1}, y_{s-1}x_{s-1}\}$.

Let $p \le k$. Since the identity $\mathbf u \approx \mathbf v$ is \mbox{$(k-1)$}-well-balanced and $({_{2\mathbf v}y})<({_{1\mathbf v}x})$, we have
$$
\mathbf v(X) = yy_{s-1}yxx_{p-1}x\,\mathbf b_{p-1,s}\,x_{s-2}\,\mathbf d\mathbf b_{s-2},
$$
where $\mathbf d \in\{x_{s-1}y_{s-1}, y_{s-1}x_{s-1}\}$.
Then the identity $\mathbf u(X) \approx \mathbf v(X)$ coincides with
\begin{equation}
\label{almost delta_p^{s-1}}
yy_{s-1}xy\,x_{p-1}x\mathbf b_{p-1,s}\,x_{s-2}\,\mathbf c\mathbf b_{s-2} \approx yy_{s-1}yx\,x_{p-1}x\mathbf b_{p-1,s}\,x_{s-2}\,\mathbf d\mathbf b_{s-2}.
\end{equation}

Now let $p > k$.
Since the identity $\mathbf u \approx \mathbf v$ is \mbox{$(k-1)$}-well-balanced and $({_{2\mathbf v}y})<({_{1\mathbf v}x})$, we have
$$
\mathbf v(X) = yy_{s-1}\,\mathbf a' \mathbf b_{k,s}\,x_{s-2}\,\mathbf d\mathbf b_{s-2},
$$
where $\mathbf d \in\{x_{s-1}y_{s-1}, y_{s-1}x_{s-1}\}$ and $\mathbf a'$ is a word with $\simple(\mathbf a')=\{x_k,y\}$ and $\mul(\mathbf a')=\{x_{k+1},x_{k+2},\dots,x_{p-1},x\}$.
Now Lemma~\ref{L: word problem LRB} and the fact that $({_{2\mathbf v}y})<({_{1\mathbf v}x})$ imply that $\ini(\mathbf a') = yx x_{p-1} x_{p-2}\cdots x_k$.
Then $\mathbf a'=y\mathbf a$ for some $\mathbf a$.
Since $\mathbf V$ satisfies $\kappa_k$, Lemma~\ref{L: ab_k=ini(a)ini(a_{x_k})b_k} implies that $\mathbf V$ satisfies also
$$
\begin{aligned}
yy_{s-1}y\,\mathbf a \mathbf b_{k,s}\,x_{s-2}\,\mathbf d\mathbf b_{s-2} \approx{}& yy_{s-1}y\,\ini(\mathbf a) \ini(\mathbf a_{x_k})\, \mathbf b_{k,s}\,x_{s-2}\,\mathbf d\mathbf b_{s-2} \\
={}&yy_{s-1}y\,xx_{p-1}\cdots x_k\, xx_{p-1}\cdots x_{k+1}\,\mathbf b_{k,s}\,x_{s-2}\,\mathbf d\mathbf b_{s-2}\\
\approx{}& yy_{s-1}y\,xx_{p-1}x\,\mathbf b_{p-1,s}\,x_{s-2}\,\mathbf d\mathbf b_{s-2}
\end{aligned}
$$
and so~\eqref{almost delta_p^{s-1}}.

We see that~\eqref{almost delta_p^{s-1}} is satisfied by $\mathbf V$ in either case.
Then $\mathbf V$ satisfies the identity $\delta_p^{s-1}$ because
$$
\begin{aligned}
&y_s\stackrel{(1)}{y_{s-1}}x_p y_s \cdots x_{s-2}\stackrel{(2)}{x_{s-1}}\stackrel{(2)}{y_{s-1}}x_{s-3}x_{s-2}\cdots x_1 x_2 x_0 x_1\\
\stackrel{\eqref{xsxt=xsxtx}}\approx{}&y_s\stackrel{(1)}{y_{s-1}}x_p y_s \cdots (x_{s-2}\stackrel{(2)}{x_{s-1}}\stackrel{(2)}{y_{s-1}})\,\mathbf c\,x_{s-3}(x_{s-2}\stackrel{(4)}{x_{s-1}}\stackrel{(4)}{y_{s-1}})\cdots x_1 x_2 x_0 x_1\\
\stackrel{\eqref{almost delta_p^{s-1}}}\approx{}&y_s\stackrel{(1)}{y_{s-1}} y_sx_p \cdots (x_{s-2}\stackrel{(2)}{x_{s-1}}\stackrel{(2)}{y_{s-1}})\,\mathbf d\,x_{s-3}(x_{s-2}\stackrel{(4)}{x_{s-1}}\stackrel{(4)}{y_{s-1}})\cdots x_1 x_2 x_0 x_1\\
\stackrel{\eqref{xsxt=xsxtx}}\approx{}&y_s\stackrel{(1)}{y_{s-1}}y_sx_p \cdots x_{s-2}\stackrel{(2)}{x_{s-1}}\stackrel{(2)}{y_{s-1}}x_{s-3}x_{s-2}\cdots x_1 x_2 x_0 x_1.
\end{aligned}
$$
Now we apply Lemma~\ref{L: u'abu''=u'bau''}(ii) and conclude that the identity $\mathbf u \approx \mathbf w$ follows from $\{\eqref{xsxt=xsxtx},\,\delta_p^{s-1}\} \subseteq \Phi(\mathbf V)$.
\end{proof}

\subsection{Removing critical pairs of the form $\{{_2}x,{_2}y\}$}
\label{subsec: critical pairs (2x,2y)}

We introduce the following countably infinite series of identities:
$$
\begin{aligned}
\label{zeta_k}
\zeta_{k-1}:&\enskip x y_0 y x_{k-1} xy \mathbf b_{k-1} \approx x y_0 y x_{k-1} yx \mathbf b_{k-1},\\
\label{lambda_k^m}
\lambda_k^m:&\enskip x y_m y x_k xy \mathbf b_{k,m}y_m\mathbf b_{m-1}\approx x y_m y x_k yx \mathbf b_{k,m}y_m\mathbf b_{m-1},\\
\label{eta_k}
\eta_{k-1}:&\enskip x y_1 y y_0 xy x_{k-1}y_1\mathbf b_{k-1}\approx x y_1 y y_0 yxx_{k-1}y_1\mathbf b_{k-1},\\
\label{mu_k^m}
\mu_k^m:&\enskip x y_{m+1} y y_m xy x_ky_{m+1}\mathbf b_{k,m}y_m\mathbf b_{m-1}\approx x y_{m+1} y y_m yxx_k y_{m+1}\mathbf b_{k,m}y_m\mathbf b_{m-1},\\
\label{nu_k}
\nu_{k-1}:&\enskip x x_{k-1}yzxy z_\infty^2z \mathbf b_{k-1}\approx x x_{k-1}yzyx z_\infty^2z \mathbf b_{k-1},
\end{aligned}
$$
where $k \in \mathbb N$ and $1\le m \le k$.

We omit the proof of the following statement because it is very similar to the proof of Lemma~\ref{L: gamma_k, delta_k^k,epsilon_k}.

\begin{lemma}
\label{L: zeta_k,lambda_k^m,eta_k,mu_k^m}
The following proper inclusions hold:
\begin{itemize}
\item[\textup{(i)}] $\mathbf P\{\zeta_0\} \subset \mathbf P\{\zeta_1\}\subset\cdots\subset \mathbf P\{\zeta_k\}\subset\cdots\subset\mathbf P\{\nu_0\}\subset\mathbf P\{\nu_1\}\subset\cdots\subset \mathbf P\{\nu_k\}\subset\cdots$;
\item[\textup{(ii)}] $\mathbf P\{\varepsilon_k\} \subset \mathbf P\{\zeta_k\}\subset \mathbf P\{\lambda_k^1\}\subset \mathbf P\{\lambda_k^2\} \subset \cdots \subset \mathbf P\{\lambda_k^k\}\subset \mathbf P\{\lambda_{k+1}^k\} \subset \cdots \subset \mathbf P\{\lambda_{k+m}^k\}\subset\cdots$;
\item[\textup{(iii)}] $\mathbf P\{\eta_0\} \subset \mathbf P\{\eta_1\}\subset\cdots\subset \mathbf P\{\eta_k\}\subset\cdots$;
\item[\textup{(iv)}] $\mathbf P\{\mu_k^1\} \subset \mathbf P\{\mu_k^2\}\subset\cdots\subset \mathbf P\{\mu_k^k\}\subset\mathbf P\{\mu_{k+1}^k\}\subset\cdots\subset\mathbf P\{\mu_{k+m}^k\}\subset\cdots$.\qed
\end{itemize}
\end{lemma}

Let $\mathbf w = \mathbf w_1\mathbf w_2\mathbf w_3$.
We say that the \textit{depth} of the subword $\mathbf w_2$ in $\mathbf w$ is equal to $k$ if $D(\mathbf w,a) = k$ for some $a \in \con(\mathbf w_2)\setminus\con(\mathbf w_1)$ and $k$ is the least number with such a property.
Equivalently, the depth of the subword $\mathbf w_2$ in $\mathbf w$ is equal to $k$ if $\mathbf w_2$ contains a $k$-divider of $\mathbf w$ but does not contain any \mbox{$(k-1)$}-divider of $\mathbf w$ (see Lemma~\ref{L: k-divider and depth}).
If, for any $k\ge0$, $\mathbf w_2$ does not contain any $k$-divider of $\mathbf w$, then we say that the depth of the subword $\mathbf w_2$ in $\mathbf w$ is \textit{infinite}.
Equivalently, the depth of the subword $\mathbf w_2$ in $\mathbf w$ is infinite if $D(\mathbf w,a)=\infty$ for any $a \in \con(\mathbf w_2)\setminus\con(\mathbf w_1)$ (see Lemma~\ref{L: k-divider and depth}).

\begin{lemma}
\label{L: u'abu''=u'bau'' zeta_k,lambda_k^m}
Let $\mathbf V$ be a subvariety of $\mathbf P$ and $\mathbf u=\mathbf u'\stackrel{(2)}{x}\stackrel{(2)}{y}\mathbf u''$ for some $\mathbf u',\mathbf u'' \in \mathfrak A^\ast$.
Further, let $\mathbf u'=\mathbf u_1\stackrel{(1)}{x}\mathbf u_2\stackrel{(1)}{y}\mathbf u_3$, where the depth of the subwords $\mathbf u_2$ and $\mathbf u_3$ in $\mathbf u$ is equal to $m$ and $k$, respectively.
Suppose also that one of the following holds:
\begin{itemize}
\item[\textup{(i)}] $k\ge 0$, $m=0$ and $\mathbf V$ satisfies $\zeta_k$;
\item[\textup{(ii)}] $1\le m \le k$ and $\mathbf V$ satisfies $\lambda_k^m$.
\end{itemize}
Then $\mathbf V$ satisfies the identity $\mathbf u\approx\mathbf u'yx\mathbf u''$.
\end{lemma}

\begin{proof}
We omit the proof of Part~(i) because it is very similar to (and in fact simpler than) that of Part~(ii).

\smallskip

(ii) The fact that the depth of $\mathbf u_3$ is equal to $k$ in $\mathbf u$ and Lemma~\ref{L: k-divider and depth} imply that $\mathbf u_3$ contains a $k$-divider $x_k$ of $\mathbf u$ such that $D(\mathbf u,x_k)=k$.
It follows from Corollary~\ref{C: form of the word} that there exist letters $x_0,x_1,\dots, x_{k-1}$ such that $D(\mathbf u,x_s)=s$ for any $0\le s<k$ and
$$
\mathbf u''=\mathbf v_{2k}x_{k-1}\mathbf v_{2k-1}x_k\mathbf v_{2k-2}x_{k-2}\mathbf v_{2k-3}x_{k-1}\cdots\mathbf v_4x_1\mathbf v_3x_2\mathbf v_2x_0\mathbf v_1x_1\mathbf v_0,
$$
for some $\mathbf v_0,\mathbf v_1,\dots,\mathbf v_{2k} \in \mathfrak A^\ast$.
Clearly, for any $s=1,2,\dots,k-1$, there exist $\mathbf v_{2s}', \mathbf v_{2s}''\in \mathfrak A^\ast$ such that $\mathbf v_{2s} = \mathbf v_{2s}'\mathbf v_{2s}''$, where $\mathbf v_{2s}'$ does not contain any \mbox{$(s-1)$}-divider of $\mathbf u$, while either $\mathbf v_{2s}''=\lambda$ or $h(\mathbf v_{2s}'')$ is an \mbox{$(s-1)$}-divider of $\mathbf u$.
Put $\mathbf v_{2k}''=\mathbf v_{2k}$ and $\mathbf v_0'=\mathbf v_0$.
Let
$$
\mathbf q_s = \mathbf v_{2s+2}''x_s\mathbf v_{2s+1}x_{s+1}\mathbf v_{2s}'
$$
for any $s=0,1,\dots,k-1$.

Let $\phi$ be the substitution
$$
(x_0,\dots,x_{k-1},x_k,y_m)\mapsto (\mathbf q_0,\dots,\mathbf q_{k-1},\mathbf u_3,\mathbf u_2).
$$
Then the identity $\phi(\lambda_k^m)$ coincides with
\begin{equation}
\label{to lambda_k^m}
x \mathbf u_2 y \mathbf u_3xy\mathbf q\approx x \mathbf u_2 y \mathbf u_3yx\mathbf q
\end{equation}
where
$$
\mathbf q = \mathbf q_{k-1}\mathbf u_3\mathbf q_{k-2}\mathbf q_{k-1}\cdots \mathbf q_{m-1}\mathbf q_m\mathbf u_2\mathbf q_{m-2}\mathbf q_{m-1}\cdots \mathbf q_1\mathbf q_2\mathbf q_0\mathbf q_1.
$$

We note that $\mathbf u'' = \mathbf q_{k-1}\mathbf q_{k-2}\cdots \mathbf q_0$.
By the same arguments as in the proof of Lemma~\ref{L: u'abu''=u'bau''}(ii) we can show that
\begin{itemize}
\item if a letter occurs in $\mathbf u_2$, then the subword $\mathbf u_2y\mathbf u_3xy\mathbf q_{k-1}\cdots\mathbf q_m\mathbf q_{m-1}$ of $\mathbf u$ contains some non-first occurrence of this letter in $\mathbf u$;
\item if a letter occurs in $\mathbf u_3$, then the subword $\mathbf u_3xy\mathbf q_{k-1}$ of $\mathbf u$ contains some non-first occurrence of this letter in $\mathbf u$;
\item if a letter occurs in $\mathbf q_s$, then the subword $\mathbf q_s\mathbf q_{s-1}$ of $\mathbf u$ contains some non-first occurrence of this letter in $\mathbf u$ for any $s=1,2,\dots,k-1$.
\end{itemize}
Then the identity~\eqref{xsxt=xsxtx} implies $
\mathbf u =\mathbf u'xy\mathbf u'' \approx \mathbf u'xy\mathbf q$.
By a similar argument we can show that $\mathbf u'yx\mathbf u'' \approx \mathbf u'yx\mathbf q$ follows from~\eqref{xsxt=xsxtx}.
These identities together with~\eqref{to lambda_k^m} imply the identity $\mathbf u \approx \mathbf u'yx\mathbf u''$.
Thus, it is satisfied by $\mathbf V$.
\end{proof}

The proof of the following lemma is very similar to the proof of Lemma~\ref{L: u'abu''=u'bau'' zeta_k,lambda_k^m}.
We provide it for the sake of completeness.

\begin{lemma}
\label{L: u'abu''=u'bau'' 2nd occurrences eta_k, mu_k^m}
Let $\mathbf V$ be a subvariety of $\mathbf P$ and $\mathbf u=\mathbf u'\stackrel{(2)}{x}\stackrel{(2)}{y}\mathbf u''$ for some $\mathbf u',\mathbf u'' \in \mathfrak A^\ast$.
Further, let
$$
\mathbf u'=\mathbf u_1\stackrel{(1)}{x}\mathbf u_2\stackrel{(1)}{y_{m+1}}\mathbf u_3\stackrel{(1)}{y}\mathbf u_4 \ \text{ and }\ \mathbf u''=\mathbf u_5\stackrel{(1)}{x_k}\mathbf u_6\stackrel{(1)}{y_{m+1}}\mathbf u_7,
$$
where $D(\mathbf u,y_{m+1})=m+1$, $D(\mathbf u,x_k)=k$ and $y_{m+1}$ is a letter such that $({_{2\mathbf u}a}) \le ({_{2\mathbf u}y_{m+1}})$ for any $a \in \{ c \mid {_{1\mathbf u}c} \text{ lies in } \mathbf u_2y_{m+1}\mathbf u_3\}$.
Suppose that the depth of the subwords $\mathbf u_2y_{m+1}\mathbf u_3$, $\mathbf u_4$ and $\mathbf u_5x_k\mathbf u_6$ is equal to $m+1$, $m$ and $k$, respectively.
Suppose also that one of the following holds:
\begin{itemize}
\item[\textup{(i)}] $k \ge 0$, $m=0$ and $\mathbf V$ satisfies $\eta_k$;
\item[\textup{(ii)}] $1\le m \le k$ and $\mathbf V$ satisfies $\mu_k^m$.
\end{itemize}
Then $\mathbf V$ satisfies the identity $\mathbf u\approx\mathbf u'yx\mathbf u''$.
\end{lemma}

\begin{proof}
We omit the proof of Part~(i) because it is very similar to (and in fact simpler than) that of Part~(ii).

\smallskip

(ii) It follows from Corollary~\ref{C: form of the word} and the fact that the depth of $\mathbf u_5x_k\mathbf u_6$ is equal to $k$ that there exist letters $x_0,x_1,\dots, x_{k-1}$ such that $D(\mathbf u,x_s)=s$ for any $0\le s<k$ and
$$
\mathbf u_7=\mathbf v_{2k}x_{k-1}\mathbf v_{2k-1}x_k\mathbf v_{2k-2}x_{k-2}\mathbf v_{2k-3}x_{k-1}\cdots\mathbf v_4x_1\mathbf v_3x_2\mathbf v_2x_0\mathbf v_1x_1\mathbf v_0,
$$
for some $\mathbf v_0,\mathbf v_1,\dots,\mathbf v_{2k} \in \mathfrak A^\ast$.
Clearly, for any $s=1,2,\dots,k-1$, there exist $\mathbf v_{2s}', \mathbf v_{2s}''\in \mathfrak A^\ast$ such that $\mathbf v_{2s} = \mathbf v_{2s}'\mathbf v_{2s}''$, where $\mathbf v_{2s}'$ does not contain any \mbox{$(s-1)$}-divider of $\mathbf u$, while either $\mathbf v_{2s}''=\lambda$ or $h(\mathbf v_{2s}'')$ is an \mbox{$(s-1)$}-divider of $\mathbf u$.
Put $\mathbf v_{2k}''=\mathbf v_{2k}$ and $\mathbf v_0'=\mathbf v_0$.
Let
$$
\mathbf q_s = \mathbf v_{2s+2}''x_s\mathbf v_{2s+1}x_{s+1}\mathbf v_{2s}'
$$
for any $s=0,1,\dots,k-1$.

Let $\phi$ be the substitution
$$
(x_0,\dots,x_{k-1},x_k,y_m,y_{m+1})\mapsto (\mathbf q_0,\dots,\mathbf q_{k-1},\mathbf u_5x_k\mathbf u_6,\mathbf u_4,\mathbf u_2y_{m+1}\mathbf u_3).
$$
Then the identity $\phi(\mu_k^m)$ coincides with
\begin{equation}
\label{to mu_k^m}
x \mathbf u_2y_{m+1}\mathbf u_3 y \mathbf u_4xy \mathbf u_5x_k\mathbf u_6\mathbf q\approx x \mathbf u_2y_{m+1}\mathbf u_3 y \mathbf u_4 yx\mathbf u_5x_k\mathbf u_6\mathbf q
\end{equation}
where
$$
\mathbf q =  \mathbf u_2y_{m+1}\mathbf u_3\mathbf q_{k-1}\mathbf u_5x_k\mathbf u_6\mathbf q_{k-2}\mathbf q_{k-1}\cdots \mathbf q_{m-1}\mathbf q_m\mathbf u_4\mathbf q_{m-2}\mathbf q_{m-1}\cdots \mathbf q_1\mathbf q_2\mathbf q_0\mathbf q_1.
$$

We note that $\mathbf u_7 = \mathbf q_{k-1}\mathbf q_{k-2}\cdots \mathbf q_0$.
The choice of $y_{m+1}$ implies that if $z$ is a letter such that ${_{1\mathbf u}z}$ occurs in the subword $\mathbf u_2y_{m+1}\mathbf u_3$ of $\mathbf u$, then $({_{2\mathbf u}z}) \le ({_{2\mathbf u}y_{m+1}})$.
We see that if a letter lies in $\mathbf u_2y_{m+1}\mathbf u_3$, then the subword $\mathbf u_2y_{m+1}\mathbf u_3 y \mathbf u_4 xy \mathbf u_5x_k\mathbf u_6y_{m+1}$ of $\mathbf u$ contains some non-first occurrence of this letter in $\mathbf u$.
By the same arguments as in the proof of Lemma~\ref{L: u'abu''=u'bau''}(ii) we can show that
\begin{itemize}
\item if a letter occurs in $\mathbf u_4$, then the subword $\mathbf u_4xy\mathbf u_5x_k\mathbf u_6y_{m+1}\mathbf q_{k-1}\cdots\mathbf q_m\mathbf q_{m-1}$ of $\mathbf u$ contains some non-first occurrence of this letter in $\mathbf u$;
\item if a letter occurs in $\mathbf u_5x_k\mathbf u_6$, then the subword $\mathbf u_5x_k\mathbf u_6y_{m+1}\mathbf q_{k-1}$ of $\mathbf u$ contains some non-first occurrence of this letter in $\mathbf u$;
\item if a letter occurs in $\mathbf q_s$, then the subword $\mathbf q_s\mathbf q_{s-1}$ of $\mathbf u$ contains some non-first occurrence of this letter in $\mathbf u$ for any $s=1,2,\dots,k-1$.
\end{itemize}
Then the identity~\eqref{xsxt=xsxtx} implies
$$
\mathbf u =\mathbf u'xy\mathbf u_5x_k\mathbf u_6y_{m+1}\mathbf u_7 \approx \mathbf u'xy\mathbf u_5x_k\mathbf u_6y_{m+1}\mathbf q.
$$
By a similar argument we can show that
$$
\mathbf u'yx\mathbf u_5x_k\mathbf u_6y_{m+1}\mathbf u_7 \approx \mathbf u'yx\mathbf u_5x_k\mathbf u_6y_{m+1}\mathbf q
$$
follows from~\eqref{xsxt=xsxtx}.
These identities together with~\eqref{to mu_k^m} imply the identity $\mathbf u \approx \mathbf u'yx\mathbf u''$.
Thus, it is satisfied by $\mathbf V$.
\end{proof}

\begin{lemma}
\label{L: u'abu''=u'bau'' nu_k}
Let $\mathbf V$ be a subvariety of $\mathbf P$ and $\mathbf u=\mathbf u'\stackrel{(2)}{x}\stackrel{(2)}{y}\mathbf u''$ for some $\mathbf u',\mathbf u'' \in \mathfrak A^\ast$.
Further, let $\mathbf u'=\mathbf u_1\stackrel{(1)}{x}\mathbf u_2\stackrel{(1)}{y}\mathbf u_3$, where the depth of the subword $\mathbf u_2$ in $\mathbf u$ is equal to $k$, while the depth of the subword $\mathbf u_3$ in $\mathbf u$ is infinite.
Suppose that $\mathbf V$ satisfies $\nu_k$.
Then $\mathbf V$ satisfies the identity $\mathbf u\approx\mathbf u'yx\mathbf u''$.
\end{lemma}

\begin{proof}
The fact that the depth of $\mathbf u_2$ is equal to $k$ in $\mathbf u$ and Lemma~\ref{L: k-divider and depth} imply that $\mathbf u_2$ contains a $k$-divider $x_k$ of $\mathbf u$ such that $D(\mathbf u,x_k)=k$.
It follows from Corollary~\ref{C: form of the word} and the fact that the depth of $\mathbf u_3$ is infinite that there exist letters $x_0,x_1,\dots, x_{k-1}$ such that $D(\mathbf u,x_s)=s$ for any $0\le s<k$ and
$$
\mathbf u''=\mathbf v_{2k}x_{k-1}\mathbf v_{2k-1}x_k\mathbf v_{2k-2}x_{k-2}\mathbf v_{2k-3}x_{k-1}\cdots\mathbf v_4x_1\mathbf v_3x_2\mathbf v_2x_0\mathbf v_1x_1\mathbf v_0,
$$
for some $\mathbf v_0,\mathbf v_1,\dots,\mathbf v_{2k} \in \mathfrak A^\ast$.
Clearly, for any $s=1,2,\dots,k-1$, there exist $\mathbf v_{2s}', \mathbf v_{2s}''\in \mathfrak A^\ast$ such that $\mathbf v_{2s} = \mathbf v_{2s}'\mathbf v_{2s}''$, where $\mathbf v_{2s}'$ does not contain any \mbox{$(s-1)$}-divider of $\mathbf u$, while either $\mathbf v_{2s}''=\lambda$ or $h(\mathbf v_{2s}'')$ is an \mbox{$(s-1)$}-divider of $\mathbf u$.
Besides that, there are words $\mathbf v_{2k}'$ and $\mathbf v_{2k}''$ such that $\mathbf v_{2k} = \mathbf v_{2k}'\mathbf v_{2k}''$, where $\mathbf v_{2k}'$ does not contain \mbox{$r$}-dividers of $\mathbf u$ for any $r \in \mathbb N\cup\{0\}$, while either $\mathbf v_{2s}''=\lambda$ or $h(\mathbf v_{2s}'')$ is an $r'$-divider of $\mathbf u$ for some $r'\in\mathbb N\cup\{0\}$.
Put $\mathbf v_0'=\mathbf v_0$.
Let
$$
\mathbf q_s = \mathbf v_{2s+2}''x_s\mathbf v_{2s+1}x_{s+1}\mathbf v_{2s}'
$$
for any $s=0,1,\dots,k-1$.

Let $\phi$ be the substitution
$$
(x_0,x_1,\dots,x_{k-1},x_k,z_\infty,z)\mapsto (\mathbf q_0,\mathbf q_1,\dots,\mathbf q_{k-1},\mathbf u_2,\mathbf v_{2k}',\mathbf u_3).
$$
Then the identity $\phi(\nu_k)$ coincides with
\begin{equation}
\label{to nu_k}
x \mathbf u_2 y \mathbf u_3xy\mathbf q\approx x \mathbf u_2 y \mathbf u_3yx\mathbf q
\end{equation}
where
$$
\mathbf q = (\mathbf v_{2k}')^2\mathbf u_3 \mathbf q_{k-1}\mathbf u_2\mathbf q_{k-2}\mathbf q_{k-1}\cdots \mathbf q_1\mathbf q_2\mathbf q_0\mathbf q_1.
$$
We note that $\mathbf u'' = \mathbf v_{2k}'\mathbf q_{k-1}\mathbf q_{k-2}\cdots \mathbf q_0$.
Since the subword $\mathbf u_3xy\mathbf v_{2k}'$ of $\mathbf u$ does not contain $r$-dividers of $\mathbf u$ for any $r \in \mathbb N\cup\{0\}$ and the first letter of $\mathbf q_{k-1}$ is an $r'$-divider of $\mathbf u$ for some $r' \in \mathbb N\cup\{0\}$, if a letter occurs in the subword $\mathbf u_3xy\mathbf v_{2k}'$ of $\mathbf u$, then this subword contains some non-first occurrnce of this letter.
By the same arguments as in the proof of Lemma~\ref{L: u'abu''=u'bau''}(ii) we can show that
\begin{itemize}
\item if a letter occurs in $\mathbf u_2$, then the subword $\mathbf u_2y\mathbf u_3xy\mathbf v_{2k}'\mathbf q_{k-1}$ of $\mathbf u$ contains some non-first occurrence of this letter in $\mathbf u$;
\item if a letter occurs in $\mathbf q_s$, then the subword $\mathbf q_s\mathbf q_{s-1}$ of $\mathbf u$ contains some non-first occurrence of this letter in $\mathbf u$ for any $s=1,2,\dots,k-1$.
\end{itemize}
Then the identity~\eqref{xsxt=xsxtx} implies $\mathbf u =\mathbf u'xy\mathbf u'' \approx \mathbf u'xy\mathbf q$.
By a similar argument we can show that $\mathbf u'yx\mathbf u'' \approx \mathbf u'yx\mathbf q$ follows from~\eqref{xsxt=xsxtx}.
These identities together with~\eqref{to nu_k} imply the identity $\mathbf u \approx \mathbf u'yx\mathbf u''$.
Thus, it is satisfied by $\mathbf V$.
\end{proof}

Let $\mathbf w = \mathbf w_1\mathbf w_2\mathbf w_3$.
The depth of $\mathbf w_2$ in $\mathbf w$ is \textit{strictly infinite} if
\begin{itemize}
\item the depth of $\mathbf w_2$ in $\mathbf w$ is infinite, in particular, $\con(\mathbf w_2)\subseteq\mul(\mathbf w)$;
\item for any $a \in \simple(\mathbf w_2)\setminus \con(\mathbf w_1)$, there are no first occurrences of letters between the last letter of $\mathbf w_2$ and the second occurrence of $a$ in $\mathbf w$ .
\end{itemize}
For an arbitrary $n \in \mathbb N$, we denote by $S_n$ the full symmetric group on the set $\{1,2,\dots,n\}$.

\begin{lemma}
\label{L: xtyxy=xtyyx implies}
Let $\mathbf u = \mathbf u_1\stackrel{(1)}{x}\mathbf u_2\stackrel{(1)}{y}\mathbf u_3\stackrel{(2)}{x}\stackrel{(2)}{y}\mathbf u_4$.
If the depth of $\mathbf u_3$ in $\mathbf u$ is strictly infinite, then the identity $\mathbf u \approx \mathbf u_1x\mathbf u_2y\mathbf u_3yx\mathbf u_4$ follows from the identities~\eqref{xsxt=xsxtx} and
\begin{equation}
\label{xtyxy=xtyyx}
xtyxy \approx xty^2x.
\end{equation}
\end{lemma}

\begin{proof}
First, we note that if $\mathbf w_1,\mathbf w_2,\mathbf w_3$ are words with $\simple(\mathbf w_3)\subseteq\con(\mathbf w_1\mathbf w_2)$, then the identities~\eqref{xsxt=xsxtx} and~\eqref{xtyxy=xtyyx} imply the identity
$$
\mathbf w_1x\mathbf w_2y\mathbf w_3xy \approx\mathbf w_1x\mathbf w_2y\mathbf w_3yx
$$
because
$$
\mathbf w_1x\mathbf w_2y\mathbf w_3xy \stackrel{\eqref{xsxt=xsxtx}}\approx\mathbf w_1x\mathbf w_2y\mathbf w_3xy\mathbf w_3 \stackrel{\eqref{xtyxy=xtyyx}}\approx \mathbf w_1x\mathbf w_2(y\mathbf w_3)^2x \stackrel{\eqref{xsxt=xsxtx}}\approx \mathbf w_1x\mathbf w_2y\mathbf w_3yx.
$$

By the hypothesis, $\mathbf u_3 = \mathbf a_0z_1\mathbf a_1z_2\mathbf a_2\cdots z_n\mathbf a_n$ and $\mathbf u_4 = \mathbf b_1z_{1\pi}\mathbf b_2z_{2\pi}\cdots \mathbf b_nz_{n\pi}\mathbf b'$, where $\pi \in S_n$ and $\mathbf a_0,\mathbf a_1,\mathbf b_1,\mathbf a_2,\mathbf b_2,\dots,\mathbf a_n,\mathbf b_n$ do not contain first occurrences of letters in $\mathbf u$.
In view of the saying in previous paragraph, the identities~\eqref{xsxt=xsxtx} and~\eqref{xtyxy=xtyyx} imply the identities
$$
\begin{aligned}
\mathbf u \approx{}& \mathbf u_1x\mathbf u_2y\mathbf u_3z_nxy (\mathbf b_1z_{1\pi}\mathbf b_2z_{2\pi}\cdots \mathbf b_nz_{n\pi})_{\{z_n\}}\mathbf b'\\
\approx{}&\mathbf u_1x\mathbf u_2y\mathbf u_3z_nz_{n-1}xy (\mathbf b_1z_{1\pi}\mathbf b_2z_{2\pi}\cdots \mathbf b_nz_{n\pi})_{\{z_n,z_{n-1}\}}\mathbf b'\\
\rule{4pt}{0pt}\vdots \;\; &\\
\approx{}&\mathbf u_1x\mathbf u_2y\mathbf u_3z_nz_{n-1}\cdots z_1xy \mathbf b_1\mathbf b_2\cdots \mathbf b_n\mathbf b'\\
\approx{}&\mathbf u_1x\mathbf u_2y\mathbf u_3z_nz_{n-1}\cdots z_1yx \mathbf b_1\mathbf b_2\cdots \mathbf b_n\mathbf b'.
\end{aligned}
$$
By a similar argument we can show that~\eqref{xsxt=xsxtx} and~\eqref{xtyxy=xtyyx} imply
$$
\mathbf u_1x\mathbf u_2y\mathbf u_3yx\mathbf u_4 \approx \mathbf u_1x\mathbf u_2y\mathbf u_3z_nz_{n-1}\cdots z_1yx \mathbf b_1\mathbf b_2\cdots \mathbf b_n\mathbf b'.
$$
We see that $\mathbf u \approx \mathbf u_1x\mathbf u_2y\mathbf u_3yx\mathbf u_4$ follows from~\eqref{xsxt=xsxtx} and~\eqref{xtyxy=xtyyx}.
\end{proof}

\begin{lemma}
\label{L: V satisfies zeta_k,lambda_k^m}
Let $\mathbf V$ be a monoid variety from the interval $[\mathbf F_1,\mathbf P]$ and $\mathbf u \approx \mathbf v$ be an identity of $\mathbf V$ such that
$$
\mathbf u = \mathbf w_1\stackrel{(1)}{x}\mathbf w_2\stackrel{(1)}{y}\mathbf w_3\stackrel{(2)}{x}\stackrel{(2)}{y}\mathbf w_4
$$
for some words $\mathbf w_1,\mathbf w_2,\mathbf w_3,\mathbf w_4$ and $({_{2\mathbf v}y})<({_{2\mathbf v}x})$.
Suppose that the depth of the subwords $\mathbf w_2$ and $\mathbf w_3$ of $\mathbf u$ is equal to $m$ and $k$, respectively.
\begin{itemize}
\item[\textup{(i)}] If $k \ge 0$ and $m=0$, then $\mathbf V$ satisfies $\zeta_k$.
\item[\textup{(ii)}] If $1\le m \le k$, then $\mathbf V$ satisfies $\lambda_k^m$.
\end{itemize}
\end{lemma}

\begin{proof}
In view of Lemma~\ref{L: 2-limited word}, we may assume without loss of generality that the identity $\mathbf u \approx \mathbf v$ is 2-limited.

Let $y_m$ be an $m$-divider of $\mathbf u$ that occurs in $\mathbf w_2$ and let $x_k$ be a $k$-divider of $\mathbf u$ that occurs in $\mathbf w_3$.
In view of Lemma~\ref{L: k-divider and depth}, $D(\mathbf u,x_k)=k$ and $D(\mathbf u,y_m)=m$.

If $k=0$, then $m=0$ as well.
Then Corollary~\ref{C: (k-1)-well-balanced} and the inclusion $\mathbf F_1 \subset \mathbf V$ imply that the identity $\mathbf u(x,y,x_0,y_0) \approx \mathbf v(x,y,x_0,y_0)$ coincides with the identity $\zeta_0$, and we are done.
So, we may assume that $k>0$.

Suppose that the identity $\mathbf u \approx \mathbf v$ is not $k$-well-balanced.
In view of Corollary~\ref{C: (k-1)-well-balanced} and the inclusion $\mathbf F_1 \subset \mathbf V$, the identity $\mathbf u \approx \mathbf v$ is $0$-well-balanced.
Then there is the smallest $p$ such that $0< p \le k$ and the identity $\mathbf u \approx \mathbf v$ is not $p$-well-balanced.
In view of Corollary~\ref{C: (k-1)-well-balanced}, $\mathbf F_{p+1} \nsubseteq \mathbf V$.
According to Lemma~\ref{L: V does not contain F_k}, $\mathbf V$ satisfies $\kappa_p$.
Then Lemmas~\ref{L: kappa_k and delta_k^k},~\ref{L: gamma_k, delta_k^k,epsilon_k}(ii) and~\ref{L: zeta_k,lambda_k^m,eta_k,mu_k^m}(ii) imply that $\mathbf V$ satisfies $\zeta_k$ and $\lambda_k^r$ for any $r=1,2,\dots,m$, and we are done.
So, we may assume that the identity $\mathbf u \approx \mathbf v$ is $k$-well-balanced.

Now we apply Corollary~\ref{C: form of the word} to deduce that there exist letters $x_0,x_1,\dots, x_{k-1}$ such that
\begin{align*}
\mathbf u={}&\mathbf u_{2k+1}\stackrel{(1)}{x_k}\mathbf u_{2k}\stackrel{(1)}{x_{k-1}}\mathbf u_{2k-1}\stackrel{(2)}{x_k}\mathbf u_{2k-2}\stackrel{(1)}{x_{k-2}}\mathbf u_{2s-3}\stackrel{(2)}{x_{k-1}}\mathbf u_{2k-4}\stackrel{(1)}{x_{k-3}}\\
&\cdot\,\mathbf u_{2k-5}\stackrel{(2)}{x_{k-2}}\cdots\mathbf u_4\stackrel{(1)}{x_1}\mathbf u_3\stackrel{(2)}{x_2}\mathbf u_2\stackrel{(1)}{x_0}\mathbf u_1\stackrel{(2)}{x_1}\mathbf u_0
\end{align*}
for some $\mathbf u_0,\mathbf u_1,\dots,\mathbf u_{2k+1} \in \mathfrak A^\ast$, $D(\mathbf u,x_r)=r$ and $\mathbf u_{2r+2}$ does not contain any $r$-dividers of $\mathbf u$ for any $r=0,1,\dots,k-1$.
The choice of $x_k$ implies that both ${_{2\mathbf u}x}$ and ${_{2\mathbf u}y}$ lie in $\mathbf u_{2k}$.
Put
$$
X = \{x_0,x_1,\dots,x_k,x,y_m,y\}.
$$

Suppose that $m=0$. Then $y_0 \in \simple(\mathbf u)=\simple(\mathbf v)$ and $\mathbf u(X) = xy_0yx_kxy\,\mathbf b_k$.
Since the identity $\mathbf u \approx \mathbf v$ is $k$-well-balanced, $\mathbf v(X) = xy_0yx_kyx\,\mathbf b_k$ and, therefore, the identity $\mathbf u(X) \approx \mathbf v(X)$ coincides with $\zeta_k$, we are done.

Suppose now that $1\le m \le k$.
Then $y_m \in \mul(\mathbf u)=\mul(\mathbf v)$.
According to the choice of $y_m$ and $x_k$, the subwords $\mathbf w_2$ and $\mathbf w_3$ do not contain any \mbox{$(m-1)$}-divider of $\mathbf u$.
It follows from Lemma~\ref{L: 2nd occurence of y_m general} that ${_{2\mathbf u}y_m}$ occurs in either $\mathbf u_{2m}$ or $\mathbf u_{2m-1}$ or $\mathbf u_{2m-2}$.
Let $y_{m-1}=h_2^{m-1}(\mathbf u,y_m)$.
Since $\mathbf u_{2m}$ does not contain any \mbox{$(m-1)$}-dividers of $\mathbf u$, the letter ${_{1\mathbf u}y_{m-1}}$ does not occur in $\mathbf u_{2m}$.
Therefore, ${_{2\mathbf u}y_m}$ occurs in either $\mathbf u_{2m-1}$ or $\mathbf u_{2m-2}$.
Then
$$
\mathbf u(X) = xy_myx_kxy\,\mathbf b_{k,m+1}\,x_{m-1}\,\mathbf c\mathbf b_{m-1},
$$
where $\mathbf c \in\{x_my_m, y_mx_m\}$.
Since $\mathbf u \approx \mathbf v$ is $k$-well-balanced, ${_{2\mathbf v}y_m}$ lies between ${_{1\mathbf v}x_{m-1}}$ and ${_{1\mathbf v}x_{m-2}}$ and
$$
\mathbf v(X) = xy_myx_kyx\,\mathbf b_{k,m+1}\,x_{m-1}\,\mathbf d\mathbf b_{m-1},
$$
where $\mathbf d \in\{x_my_m, y_mx_m\}$.
We see that the identity $\mathbf u(X) \approx \mathbf v(X)$ coincides with
\begin{equation}
\label{almost lambda_k^m}
xy_myx_kxy\,\mathbf b_{k,m+1}\,x_{m-1}\,\mathbf c\mathbf b_{m-1} \approx xy_myx_kyx\,\mathbf b_{k,m+1}\,x_{m-1}\,\mathbf d\mathbf b_{m-1}.
\end{equation}
Then $\mathbf V$ satisfies $\lambda_k^m$ because
$$
\begin{aligned}
&x\stackrel{(1)}{y_m}yx_k xy \cdots x_{m-1}\stackrel{(2)}{x_m}\stackrel{(2)}{y_m}x_{m-2}x_{m-1}\cdots x_1 x_2 x_0 x_1\\
\stackrel{\eqref{xsxt=xsxtx}}\approx{}&x\stackrel{(1)}{y_m}yx_k xy \cdots (x_{m-1}\stackrel{(2)}{x_m}\stackrel{(2)}{y_m})\,\mathbf c\,x_{m-2}(x_{m-1}\stackrel{(4)}{x_m}\stackrel{(4)}{y_m})\cdots x_1 x_2 x_0 x_1\\
\stackrel{\eqref{almost lambda_k^m}}\approx{}&x\stackrel{(1)}{y_m}yx_k yx \cdots (x_{m-1}\stackrel{(2)}{x_m}\stackrel{(2)}{y_m})\,\mathbf d\,x_{m-2}(x_{m-1}\stackrel{(4)}{x_m}\stackrel{(4)}{y_m})\cdots x_1 x_2 x_0 x_1\\
\stackrel{\eqref{xsxt=xsxtx}}\approx{}&x\stackrel{(1)}{y_m}yx_k yx \cdots x_{m-1}\stackrel{(2)}{x_m}\stackrel{(2)}{y_m}x_{m-2}x_{m-1}\cdots x_1 x_2 x_0 x_1,
\end{aligned}
$$
and we are done.
\end{proof}

\begin{lemma}
\label{L: V satisfies eta_k,mu_k^m}
Let $\mathbf V$ be a monoid variety from the interval $[\mathbf F_1,\mathbf P]$ and $\mathbf u \approx \mathbf v$ be an identity of $\mathbf V$ such that
$$
\mathbf u = \mathbf w_1\stackrel{(1)}{x}\mathbf w_2\stackrel{(1)}{y_{m+1}}\mathbf w_3\stackrel{(1)}{y}\mathbf w_4\,\stackrel{(2)}{x}\stackrel{(2)}{y}\mathbf w_5\stackrel{(2)}{y_{m+1}}\mathbf w_6
$$
for some words $\mathbf w_1,\mathbf w_2,\mathbf w_3,\mathbf w_4,\mathbf w_5,\mathbf w_6$ and $({_{2\mathbf v}x})<({_{2\mathbf v}y})$.
Suppose that $({_{2\mathbf u}a})<({_{2\mathbf u}y_{m+1}})$ for any $a \in \con(\mathbf w_2\mathbf w_3)\setminus \con(\mathbf w_1)$.
Suppose also that the depth of the subwords $\mathbf w_2y_{m+1}\mathbf w_3$, $\mathbf w_4$ and $\mathbf w_5$ of $\mathbf u$ is equal to $m+1$, $m$ and $k$, respectively.
\begin{itemize}
\item[\textup{(i)}] If $k \ge 0$ and $m=0$, then $\mathbf V$ satisfies $\eta_k$.
\item[\textup{(ii)}] If $1\le m \le k$, then $\mathbf V$ satisfies $\mu_k^m$.
\end{itemize}
\end{lemma}

\begin{proof}
If $\mathbf{LRB}\nsubseteq \mathbf V$, then $\mathbf V\subseteq \mathbf F$ by Lemma~\ref{L: does not contain LRB,F_1}.
Evidently, $\mathbf F$ satisfies $\eta_k$ and $\mu_k^m$ for any $k \ge 0$ and $m\le k$.
So, we may further assume that $\mathbf{LRB}\subseteq \mathbf V$.
Further, Lemma~\ref{L: 2-limited word} allows us to assume that the identity $\mathbf u \approx \mathbf v$ is 2-limited.

Let $x_k$ be a $k$-divider of $\mathbf u$ that occurs in $\mathbf w_5$.
In view of Lemma~\ref{L: k-divider and depth}, $D(\mathbf u,x_k)=k$ and  $D(\mathbf u,y_{m+1})=m+1$.
Now by Corollary~\ref{C: form of the word}, there exist letters $x_0,x_1,\dots, x_{k-1}$ such that $D(\mathbf u,x_r)=r$ and
\begin{align*}
\mathbf u={}&\mathbf u_{2k+1}\stackrel{(1)}{x_k}\mathbf u_{2k}\stackrel{(1)}{x_{k-1}}\mathbf u_{2k-1}\stackrel{(2)}{x_k}\mathbf u_{2k-2}\stackrel{(1)}{x_{k-2}}\mathbf u_{2k-3}\stackrel{(2)}{x_{k-1}}\mathbf u_{2k-4}\stackrel{(1)}{x_{k-3}}\\
&\cdot\,\mathbf u_{2k-5}\stackrel{(2)}{x_{k-2}}\cdots\mathbf u_4\stackrel{(1)}{x_1}\mathbf u_3\stackrel{(2)}{x_2}\mathbf u_2\stackrel{(1)}{x_0}\mathbf u_1\stackrel{(2)}{x_1}\mathbf u_0
\end{align*}
for some $\mathbf u_0,\mathbf u_1,\dots,\mathbf u_{2k+1} \in \mathfrak A^\ast$ such that $\mathbf u_{2r+2}$ does not contain any $r$-dividers of $\mathbf u$ for any $r=0,1,\dots,k-1$.
Since the depth of $\mathbf w_5$ in $\mathbf u$ is equal to $k$, the letter ${_{2\mathbf u}y_{m+1}}$ lies in $\mathbf u_{2k}$.
Let $y_m$ be an $m$-divider of $\mathbf u$ that occurs in $\mathbf w_4$.
Clearly,
$$
({_{1\mathbf u}x})<({_{1\mathbf u}y_{m+1}})<({_{1\mathbf u}y})<({_{1\mathbf u}y_m})<({_{2\mathbf u}x})<({_{2\mathbf u}y})<({_{1\mathbf u}x_k}).
$$
Put
$$
X = \{x_0,x_1,\dots,x_k,x,y_m,y_{m+1},y\}.
$$
There are two cases.

\smallskip

\textit{Case }1: the identity $\mathbf u \approx \mathbf v$ is $k$-well-balanced.

Suppose that $k \ge 0$ and $m=0$.
Then $y_m=y_0 \in \simple(\mathbf u)=\simple(\mathbf v)$ and $\mathbf u(X) = xy_1yy_0xyx_ky_1\,\mathbf b_k$.
Since the identity $\mathbf u \approx \mathbf v$ is $k$-well-balanced, we have $\mathbf v(X) = xy_1yy_0yxx_ky_1\,\mathbf b_k$.
Therefore, $\mathbf u(X) \approx \mathbf v(X)$ coincides with $\eta_k$, and we are done.

Suppose now that $1 \le m \le k$.
Then $y_m \in \mul(\mathbf u)=\mul(\mathbf v)$.
Clearly, the subwords $\mathbf w_4$ and $\mathbf w_5$ do not contain any \mbox{$(m-1)$}-divider of $\mathbf u$.
It follows from Lemma~\ref{L: 2nd occurence of y_m general} that ${_{2\mathbf u}y_m}$ occurs in either $\mathbf u_{2m}$ or $\mathbf u_{2m-1}$ or $\mathbf u_{2m-2}$.
Let $y_{m-1}=h_2^{m-1}(\mathbf u,y_m)$.
Since $\mathbf u_{2m}$ does not contain any \mbox{$(m-1)$}-dividers of $\mathbf u$, the letter ${_{1\mathbf u}y_{m-1}}$ does not occur in $\mathbf u_{2m}$.
Therefore, ${_{2\mathbf u}y_m}$ occurs in either $\mathbf u_{2m-1}$ or $\mathbf u_{2m-2}$.
Then
$$
\mathbf u(X) = xy_{m+1}yy_mxyx_ky_{m+1}\,\mathbf b_{k,m+1}\,x_{m-1}\,\mathbf c\mathbf b_{m-1},
$$
where $\mathbf c \in\{x_my_m, y_mx_m\}$.
Since $\mathbf u \approx \mathbf v$ is $k$-well-balanced, ${_{2\mathbf v}y_m}$ lies between ${_{1\mathbf v}x_{m-1}}$ and ${_{1\mathbf v}x_{m-2}}$ and
$$
\mathbf v(X) = xy_{m+1}yy_myxx_ky_{m+1}\,\mathbf b_{k,m+1}\,x_{m-1}\,\mathbf d\mathbf b_{m-1},
$$
where $\mathbf d \in\{x_my_m, y_mx_m\}$.
We see that the identity $\mathbf u(X) \approx \mathbf v(X)$ coincides with
\begin{equation}
\label{almost mu_k^m}
\begin{aligned}
&xy_{m+1}yy_mxyx_ky_{m+1}\,\mathbf b_{k,m+1}\,x_{m-1}\,\mathbf c\mathbf b_{m-1}\\ \approx{}& xy_{m+1}yy_myxx_ky_{m+1}\,\mathbf b_{k,m+1}\,x_{m-1}\,\mathbf d\mathbf b_{m-1}.
\end{aligned}
\end{equation}
Then $\mathbf V$ satisfies $\mu_k^m$ because
$$
\begin{aligned}
&x\stackrel{(1)}{y_{m+1}}yy_m xyx_ky_{m+1} \cdots x_{m-1}\stackrel{(2)}{x_m}\stackrel{(2)}{y_m}x_{m-2}x_{m-1}\cdots x_1 x_2 x_0 x_1\\
\stackrel{\eqref{xsxt=xsxtx}}\approx{}&x\stackrel{(1)}{y_{m+1}}yy_m xyx_ky_{m+1} \cdots (x_{m-1}\stackrel{(2)}{x_m}\stackrel{(2)}{y_m})\,\mathbf c\,x_{m-2}(x_{m-1}\stackrel{(4)}{x_m}\stackrel{(4)}{y_m})\cdots x_1 x_2 x_0 x_1\\
\stackrel{\eqref{almost mu_k^m}}\approx{}&x\stackrel{(1)}{y_{m+1}}yy_m yxx_ky_{m+1} \cdots (x_{m-1}\stackrel{(2)}{x_m}\stackrel{(2)}{y_m})\,\mathbf d\,x_{m-2}(x_{m-1}\stackrel{(4)}{x_m}\stackrel{(4)}{y_m})\cdots x_1 x_2 x_0 x_1\\
\stackrel{\eqref{xsxt=xsxtx}}\approx{}&x\stackrel{(1)}{y_{m+1}}yx_m yxx_ky_{m+1} \cdots x_{m-1}\stackrel{(2)}{x_m}\stackrel{(2)}{y_m}x_{m-2}x_{m-1}\cdots x_1 x_2 x_0 x_1,
\end{aligned}
$$
and we are done.

\textit{Case }2: the identity $\mathbf u \approx \mathbf v$ is not $k$-well-balanced.
In view of Corollary~\ref{C: (k-1)-well-balanced} and the inclusion $\mathbf F_1 \subset \mathbf V$, the identity $\mathbf u \approx \mathbf v$ is $0$-well-balanced.
Then there is the smallest $n$ such that $0< n \le k$ and the identity $\mathbf u \approx \mathbf v$ is not $n$-well-balanced.
In view of Corollary~\ref{C: (k-1)-well-balanced}, $\mathbf F_{n+1} \nsubseteq \mathbf V$.
Then $\mathbf F_n \subseteq\mathbf V$.
According to Lemma~\ref{L: V does not contain F_k}, $\mathbf V$ satisfies $\kappa_n$.

Suppose that $k \ge 0$ and $m=0$.
Then $y_m=y_0 \in \simple(\mathbf u)=\simple(\mathbf v)$ and $\mathbf u(X) = xy_1yy_0xyx_ky_1\mathbf b_k$.
Lemma~\ref{L: word problem LRB} and the inclusion $\mathbf{LRB}\subset\mathbf V$ together with the fact that the identity $\mathbf u \approx \mathbf v$ is \mbox{$(n-1)$}-well-balanced imply $\mathbf v(X) = xy_1yy_0\,\mathbf a \mathbf b_n$ for some word $\mathbf a$ with $\simple(\mathbf a)=\{x_n,x,y_1,y\}$ and $\mul(\mathbf a)=\{x_{n+1},x_{n+2},\dots,x_k\}$.
According to Lemma~\ref{L: word problem LRB}, $\ini(\mathbf a_{\{x,y\}}) = x_k x_{k-1} \cdots x_n$.
Further, $\mathbf a(x,y)=yx$ because $({_{2\mathbf v}y})<({_{2\mathbf v}x})$.
Then $\mathbf V$ satisfies the identities
$$
\begin{aligned}
xy_1yy_0\,\mathbf a \mathbf b_n\approx{}&xy_1yy_0\,\mathbf axyy_1\, \mathbf b_n&&\textup{by }\eqref{xsxt=xsxtx}\\
\approx{}&xy_1yy_0\,\ini(\mathbf axyy_1)\ini((\mathbf axyy_1)_{x_n})\, \mathbf b_n&&\textup{by Lemma~\ref{L: ab_k=ini(a)ini(a_{x_k})b_k}}\\
={}&xy_1yy_0\,\ini(\mathbf a)\ini(\mathbf a_{x_n})\, \mathbf b_n.
\end{aligned}
$$
Since $({_{1\mathbf a}y})<({_{1\mathbf a}x})$, it follows from Proposition~\ref{P: 1st and 2nd occurrences} that $\Phi(\mathbf V)$ together with
\begin{equation}
\label{xytxy=xytyx}
xytxy \approx xytyx
\end{equation}
imply the identity
$$
\ini(\mathbf a)\ini(\mathbf a_{x_n})\, \mathbf b_n \approx \mathbf a' \mathbf b_n,
$$
where $\simple(\mathbf a')=\{x_n\}$ and $\ini(\mathbf a')=yxx_ky_1 x_{k-1} \cdots x_n$.
Now we apply Lemma~\ref{L: ab_k=ini(a)ini(a_{x_k})b_k} again and obtain that $\mathbf V$ satisfies the identities
$$
\begin{aligned}
\mathbf a'\mathbf b_n \approx{}&\ini(\mathbf a')\ini(\mathbf a_{x_n}')\mathbf b_n  =yxx_ky_1 x_{k-1} \cdots x_nyxx_ky_1 x_{k-1} \cdots x_{n+1}\mathbf b_n\\
={}&\ini(y^2x^2x_ky_1^2\mathbf b_{k,n+1})\ini((y^2x^2x_ky_1^2\mathbf b_{k,n+1})_{x_n})\mathbf b_n \approx y^2x^2x_ky_1^2\mathbf b_{k,n+1}\mathbf b_n\\
= {}&y^2x^2x_ky_1^2\mathbf b_k.
\end{aligned}
$$
Then the identities
$$
xy_1yy_0\,\mathbf a'\mathbf b_n \approx xy_1yy_0\, y^2x^2x_ky_1^2\mathbf b_k \stackrel{\eqref{xsxt=xsxtx}}\approx xy_1yy_0\, yxx_ky_1\mathbf b_k
$$
hold in $\mathbf V$.
Therefore, $\mathbf V$ satisfies
$$
xy_1yy_0\,\mathbf a \mathbf b_n\approx xy_1yy_0\, yxx_ky_1\mathbf b_k
$$
and so $\eta_k$.

Suppose now that $1 \le m \le k$.
Then $y_m \in \mul(\mathbf u)=\mul(\mathbf v)$.
It is routine to verify that the depth of the letter $y$ in both sides of the identity $\mu_k^m$ is equal to $m+1$.
Thus, if $n \le m$, then $\mathbf V$ satisfies $\mu_k^m$ by Corollary~\ref{C: u'abu''=u'bau'' kappa_k}.
So, we may assume that $m < n$.
Clearly, the subwords $\mathbf w_4$ and $\mathbf w_5$ do not contain any \mbox{$(m-1)$}-divider of $\mathbf u$.
It follows from Lemma~\ref{L: 2nd occurence of y_m general} that ${_{2\mathbf u}y_m}$ occurs in either $\mathbf u_{2m}$ or $\mathbf u_{2m-1}$ or $\mathbf u_{2m-2}$.
Let $y_{m-1}=h_2^{m-1}(\mathbf u,y_m)$.
Since $\mathbf u_{2m}$ does not contain any \mbox{$(m-1)$}-divider of $\mathbf u$, the letter ${_{1\mathbf u}y_{m-1}}$ does not occur in $\mathbf u_{2m}$.
Therefore, ${_{2\mathbf u}y_m}$ occurs in either $\mathbf u_{2m-1}$ or $\mathbf u_{2m-2}$.
Then
$$
\mathbf u(X) = xy_{m+1}yy_mxyx_ky_{m+1}\,\mathbf b_{k,m+1}\,x_{m-1}\,\mathbf c\mathbf b_{m-1},
$$
where $\mathbf c \in\{x_my_m, y_mx_m\}$.
Lemma~\ref{L: word problem LRB} and the inclusion $\mathbf{LRB}\subset\mathbf V$ together with the facts that $m < n$ and $\mathbf u \approx \mathbf v$ is \mbox{$(n-1)$}-well-balanced imply that ${_{2\mathbf v}y_m}$ lies between ${_{1\mathbf v}x_{m-1}}$ and ${_{1\mathbf v}x_{m-2}}$ and
$$
\mathbf v(X) = xy_{m+1}yy_m\mathbf a\,\mathbf b_{n,m+1}\,x_{m-1}\,\mathbf d\mathbf b_{m-1},
$$
$\simple(\mathbf a)=\{x_n,x,y,y_{m+1}\}$, $\mul(\mathbf a)=\{x_{n+1},x_{n+2},\dots,x_k\}$ and $\mathbf d \in \{x_my_m,y_mx_m\}$.
As in the previous paragraph, one can show that $\mathbf V$ satisfies the identity
$$
xy_{m+1}yy_m\mathbf a\,\mathbf b_{n,m+1}\,x_{m-1}\,\mathbf d\mathbf b_{m-1} \approx xy_{m+1}yy_myxx_ky_{m+1}\,\mathbf b_{k,m+1}\,x_{m-1}\,\mathbf d\mathbf b_{m-1}
$$
and so the identity~\eqref{almost mu_k^m}.
As we have verified in Case~1, the last identity together with~\eqref{xsxt=xsxtx} imply $\mu_k^m$.
We see that the identity $\mu_k^m$ holds in $\mathbf V$.
\end{proof}

\begin{lemma}
\label{L: V satisfies nu_p}
Let $\mathbf V$ be a variety from the interval $[\mathbf{LRB},\mathbf P]$.
Suppose that $\mathbf V$ satisfies the identity
\begin{equation}
\label{almost nu_p}
xx_pyzxyz_\infty^2z\mathbf b_p \approx xx_p\mathbf c\mathbf b_p,
\end{equation}
where
\begin{equation}
\label{conditions for almost nu_p}
\mathbf c=\mathbf c_1y\mathbf c_2y\mathbf c_3x\mathbf c_4z_\infty^2\mathbf c_5 \ \text{ and }\ \mathbf c_1\mathbf c_2\mathbf c_3\mathbf c_4\mathbf c_5=z^2.
\end{equation}
Then the identity $\nu_p$ holds in $\mathbf V$.
\end{lemma}

\begin{proof}
According to Lemma~\ref{L: word problem LRB} and the inclusion $\mathbf{LRB}\subset\mathbf V$, $\ini(\mathbf c_x)=yzz_\infty$.
Therefore, $\mathbf c_1=\lambda$ and $\mathbf c_2\mathbf c_3\mathbf c_4\ne \lambda$.
If $\mathbf c_2=\lambda$, then it follows from Proposition~\ref{P: 1st and 2nd occurrences} that $\Phi(\mathbf V)$ implies the identity $xx_p\mathbf c\mathbf b_p\approx xx_pyz\mathbf d\mathbf b_p$, where $\mathbf d$ is a word obtained from $y\mathbf c_3x\mathbf c_4z_\infty^2\mathbf c_5$ by removing the first occurrence of $z$.
So, we may assume that $z\in\con(\mathbf c_2)$.
If $z\in\con(\mathbf c_5)$, then~\eqref{almost nu_p} is nothing but $\nu_p$.
So, it remains to consider the case when $\mathbf c_5=\lambda$.
In this case, we substitute $z_\infty z$ for $z_\infty$ in the identity $\mathbf u(z,z_\infty)\approx \mathbf v(z,z_\infty)$.
As a result, we obtain an identity that is equivalent modulo~\eqref{xsxt=xsxtx} to~\eqref{xyxy=xxyy}.
Since $\ini^2(\mathbf c)=y^2z^2x^2z_\infty^2=\ini^2(yzyxz_\infty^2 z)$, Corollary~\ref{C: block = ini^2(block)} implies that $\mathbf V$ satisfies the identity $xx_p\mathbf c\mathbf b_p \approx  xx_pyzyxz_\infty^2z\mathbf b_p$ and, therefore, the identity $\nu_p$.
\end{proof}

Let $\mathbf O_1$ denote the variety defined by the identities~\eqref{xsxt=xsxtx},~\eqref{xtyxy=xtyyx} and
\begin{equation}
\label{xzyt xy z_infty z_infty z = xzyt yx z_infty z_infty z}
xzytxyz_\infty^2z \approx xzytyxz_\infty^2z.
\end{equation}
Put
$$
\Psi_1 = \{\eqref{xsxt=xsxtx},\,\eqref{xtyxy=xtyyx},\,\eqref{xzyt xy z_infty z_infty z = xzyt yx z_infty z_infty z},\,\zeta_{k-1},\,\lambda_k^m,\,\eta_{k-1},\,\mu_k^m,\,\nu_{k-1} \mid k \in \mathbb N, \ 1 \le m \le k \}.
$$

\begin{proposition}
\label{P: 2nd occurrences O_1}
Let $\mathbf V$ be a monoid variety from the interval $[\mathbf{LRB}\vee\mathbf F_1,\mathbf O_1]$, $\mathbf u \approx \mathbf v$ be an identity of $\mathbf V$ and $\{_2 x, {_2 y}\}$ be a critical pair in $\mathbf u \approx \mathbf v$. Then the critical pair $\{_2 x, {_2 y}\}$ is $\Delta$-removable for some $\Delta\subseteq\Psi_1(\mathbf V)$.
\end{proposition}

\begin{proof}
We consider only the case when $({_{1\mathbf u}x})<({_{1\mathbf u}y})$ since the case when $({_{1\mathbf u}y})<({_{1\mathbf u}x})$ is similar.
Then $\mathbf u = \mathbf w_1\stackrel{(1)}{x}\mathbf w_2\stackrel{(1)}{y}\mathbf w_3\stackrel{(2)}{x}\stackrel{(2)}{y}\mathbf w_4$ for some words $\mathbf w_1,\mathbf w_2,\mathbf w_3,\mathbf w_4$.
According to the inclusion $\mathbf{LRB}\subset\mathbf V$ and Lemma~\ref{L: word problem LRB}, $({_{1\mathbf v}x})<({_{1\mathbf v}y})$.
Let $\mathbf w$ be the word obtained from $\mathbf u$ by replacing ${_{2\mathbf u} x}\,{_{2\mathbf u} y}$ with ${_{2\mathbf u} y}\,{_{1\mathbf u} x}$.
Lemma~\ref{L: 2-limited word} allows us to assume that the identity $\mathbf u \approx \mathbf v$ is 2-limited.

In view of Remark~\ref{R: max decomposition}, there is a number $d$ such that the $d$-decomposition of $\mathbf u$ is maximal.
Let $\mathbf a$ denote the $d$-block of $\mathbf u$, which contains the critical pair $\{{_2 x}, {_2 y}\}$.
Then $\mathbf a = \mathbf a'\stackrel{(2)}{x}\stackrel{(2)}{y}\mathbf a''$ and $\mathbf u = \mathbf u'\mathbf a\mathbf u''$ for some $\mathbf a',\mathbf a',\mathbf u',\mathbf u'' \in \mathfrak A^\ast$.
Clearly, $\con(\mathbf a)\subseteq\mul(\mathbf u'\mathbf a)$.

If $\con(\mathbf w_2)\subseteq\mul(\mathbf u'\mathbf a)$, then the identity $\mathbf u \approx \mathbf w$ follows from $\{\eqref{xsxt=xsxtx},\,\eqref{xzyt xy z_infty z_infty z = xzyt yx z_infty z_infty z}\}\subseteq \Psi_1(\mathbf V)$ because
$$
\mathbf u \stackrel{\eqref{xsxt=xsxtx}}\approx \mathbf w_1x\mathbf w_2y\mathbf w_3xy(\mathbf a'')^2\mathbf w_2\mathbf u'' \stackrel{\eqref{xzyt xy z_infty z_infty z = xzyt yx z_infty z_infty z}}\approx \mathbf w_1x\mathbf w_2y\mathbf w_3yx(\mathbf a'')^2\mathbf w_2\mathbf u'' \stackrel{\eqref{xsxt=xsxtx}}\approx \mathbf w.
$$
So, we may assume that $\con(\mathbf w_2)\nsubseteq\mul(\mathbf u'\mathbf a)$.
Then there exist a letter $x_p$ and a number $p\in \mathbb N\cup\{0\}$ such that the following claims hold:
\begin{itemize}
\item ${_{1\mathbf u}x_p}$ occurs in the subword $\mathbf w_2$ of $\mathbf u$;
\item $D(\mathbf u,x_p)=p$;
\item the depth of $\mathbf w_2$ in $\mathbf u$ is equal to $p$;
\item if $p>0$, then ${_{2\mathbf u}x_p}$ occurs in the subword $\mathbf u''$ of $\mathbf u$.
\end{itemize}
There are two subcases.

\smallskip

\textit{Case }1: $D(\mathbf u,y)=\infty$.
Then the first occurrence of $y$ in $\mathbf u$ lies in the $d$-block $\mathbf a$.
This means that $\mathbf a' = \mathbf a_1\,y\,\mathbf w_3$ for some $\mathbf a_1\in \mathfrak A^\ast$ such that $y \notin\con(\mathbf a_1)$.

Clearly, the depth of $\mathbf w_3$ in $\mathbf u$ is infinite.
If the depth of $\mathbf w_3$ in $\mathbf u$ is strictly infinite, then the identity $\mathbf u \approx \mathbf w$ follows from $\{\eqref{xsxt=xsxtx},\,\eqref{xtyxy=xtyyx}\}\subseteq \Psi_1(\mathbf V)$ by Lemma~\ref{L: xtyxy=xtyyx implies}.
So, it remains to consider the case when the depth of $\mathbf w_3$ in $\mathbf u$ is not strictly infinite.
Then there are letters $z$ and $z_\infty$ such that $({_{1\mathbf u}z})<({_{1\mathbf u}z_\infty})<({_{2\mathbf u}z})$ and ${_{1\mathbf u}z}$ lies in the subword $\mathbf w_3$ of $\mathbf u$, while ${_{1\mathbf u}z_\infty}$ and ${_{2\mathbf u}z}$ lie in the subword $\mathbf a''$ of $\mathbf u$.

Now by Corollary~\ref{C: form of the word}, there exist letters $x_0,x_1,\dots, x_{p-1}$ such that
\begin{align*}
\mathbf u={}&\mathbf u_{2p+1}\stackrel{(1)}{x_p}\mathbf u_{2p}\stackrel{(1)}{x_{p-1}}\mathbf u_{2p-1}\stackrel{(2)}{x_p}\mathbf u_{2p-2}\stackrel{(1)}{x_{p-2}}\mathbf u_{2p-3}\stackrel{(2)}{x_{p-1}}\mathbf u_{2p-4}\stackrel{(1)}{x_{p-3}}\\
&\cdot\,\mathbf u_{2p-5}\stackrel{(2)}{x_{p-2}}\cdots\mathbf u_4\stackrel{(1)}{x_1}\mathbf u_3\stackrel{(2)}{x_2}\mathbf u_2\stackrel{(1)}{x_0}\mathbf u_1\stackrel{(2)}{x_1}\mathbf u_0
\end{align*}
for some $\mathbf u_0,\mathbf u_1,\dots,\mathbf u_{2p+1} \in \mathfrak A^\ast$ with $D(\mathbf u,x_r)=r$.
Clearly, $({_{1\mathbf u}x_p})<({_{1\mathbf u}y})<({_{2\mathbf u}z})<({_{1\mathbf u}x_{p-1}})$.

Corollary~\ref{C: (k-1)-well-balanced} and the inclusion $\mathbf F_1 \subset \mathbf V$ imply that the identity $\mathbf u \approx \mathbf v$ is $0$-well-balanced.
Then if the identity $\mathbf u \approx \mathbf v$ is not $p$-well-balanced, then there is the smallest $k$ such that $0< k \le p$ and the identity $\mathbf u \approx \mathbf v$ is not $k$-well-balanced.
In view of Corollary~\ref{C: (k-1)-well-balanced}, $\mathbf F_{k+1} \nsubseteq \mathbf V$.
According to Lemma~\ref{L: V does not contain F_k}, $\mathbf V$ satisfies $\kappa_k$.
Then Lemmas~\ref{L: kappa_k and delta_k^k},~\ref{L: gamma_k, delta_k^k,epsilon_k}(ii) and~\ref{L: zeta_k,lambda_k^m,eta_k,mu_k^m}(i) imply that $\mathbf V$ satisfies $\nu_p$.
Now we apply Lemma~\ref{L: u'abu''=u'bau'' nu_k} and conclude that the identity $\mathbf u \approx \mathbf w$ follows from $\{\eqref{xsxt=xsxtx},\,\nu_p\} \subseteq \Psi_1(\mathbf V)$.

So, we may assume that the identity $\mathbf u \approx \mathbf v$ is $p$-well-balanced.
Now we substitute $yxz_\infty^2$ for $z_\infty$ in the identity
$$
\mathbf u(x_0,x_1,\dots,x_p,x,y,z,z_\infty) \approx \mathbf v(x_0,x_1,\dots,x_p,x,y,z,z_\infty)
$$
Since $({_{2\mathbf v}y})<({_{2\mathbf v}x})$, we obtain an identity that is equivalent modulo~\eqref{xsxt=xsxtx} to the identity~\eqref{almost nu_p} such that \eqref{conditions for almost nu_p} is true.
Now we apply Lemma~\ref{L: V satisfies nu_p} and conclude that $\nu_p$ is satisfied by $\mathbf V$.
Then the identity $\mathbf u \approx \mathbf w$ follows from $\{\eqref{xsxt=xsxtx},\,\nu_p\}\subseteq \Psi_1(\mathbf V)$ by Lemma~\ref{L: u'abu''=u'bau'' nu_k}.

\smallskip

\textit{Case }2: $D(\mathbf u,y)<\infty$, say $D(\mathbf u,y)=s+1$ for some $s \in \mathbb N\cup\{0\}$.
This means that the depth of the subword $\mathbf w_3$ in $\mathbf u$ is equal to $s$.
The choice of $x_p$ implies that  $p \le s+1$.

Suppose that $p \le s$.
If $p=0$, then $\mathbf V$ satisfies $\zeta_s$ by Lemma~\ref{L: V satisfies zeta_k,lambda_k^m}(i).
Then we apply Lemma~\ref{L: u'abu''=u'bau'' zeta_k,lambda_k^m}(i) and conclude that the identity $\mathbf u \approx \mathbf w$ follows from $\{\eqref{xsxt=xsxtx},\,\zeta_s\} \subseteq \Psi_1(\mathbf V)$.
If $p>0$, then $\mathbf V$ satisfies $\lambda_s^p$ by Lemma~\ref{L: V satisfies zeta_k,lambda_k^m}(ii).
Then the identity $\mathbf u \approx \mathbf w$ follows from $\{\eqref{xsxt=xsxtx},\,\lambda_s^p\} \subseteq \Psi_1(\mathbf V)$ by Lemma~\ref{L: u'abu''=u'bau'' zeta_k,lambda_k^m}(ii).

Suppose now that $p = s+1$.
Evidently, we may assume without loss of generality that $({_{2\mathbf u}a})\le({_{2\mathbf u}x_p})$ for any $a\in\con(\mathbf w_2)\setminus\con(\mathbf w_1)$.
Let $\mathbf u''=\mathbf u_1''\,{_{2\mathbf u}x_p}\,\mathbf u_2''$.
Since the first letter of $\mathbf u''$ is a $d$-divider of $\mathbf u$, the depth of the subword $\mathbf a''\mathbf u_1''$ in $\mathbf u$ is equal to some number, say $k$.
Clearly, $p-1=s \le k$.
If $p=1$, then $\mathbf V$ satisfies $\eta_k$ by Lemma~\ref{L: V satisfies eta_k,mu_k^m}(i).
Then we apply Lemma~\ref{L: u'abu''=u'bau'' 2nd occurrences eta_k, mu_k^m}(i) and conclude that the identity $\mathbf u \approx \mathbf w$ follows from $\{\eqref{xsxt=xsxtx},\,\eta_k\} \subseteq \Psi_1(\mathbf V)$.
If $p>1$, then $\mathbf V$ satisfies $\mu_k^s$ by Lemma~\ref{L: V satisfies eta_k,mu_k^m}(ii).
Then the identity $\mathbf u \approx \mathbf w$ follows from $\{\eqref{xsxt=xsxtx},\,\mu_k^s\} \subseteq \Psi_1(\mathbf V)$ by Lemma~\ref{L: u'abu''=u'bau'' 2nd occurrences eta_k, mu_k^m}(ii).
\end{proof}

Let $\mathbf O_2= \mathbf P\{\eqref{xzyt xy z_infty z_infty z = xzyt yx z_infty z_infty z},\,\nu_1\}$.
Put
$$
\Psi_2 = \{\eqref{xsxt=xsxtx},\,\eqref{xtyxy=xtyyx},\,\eqref{xzyt xy z_infty z_infty z = xzyt yx z_infty z_infty z},\,\zeta_{k-1},\,\lambda_k^m,\,\eta_{k-1},\,\mu_k^m,\,\nu_0,\,\nu_1 \mid k \in \mathbb N, \ 1 \le m \le k \}.
$$

\begin{proposition}
\label{P: 2nd occurrences O_2}
Let $\mathbf V$ be a monoid variety from the interval $[\mathbf{LRB}\vee\mathbf F_1,\mathbf O_2]$, $\mathbf u \approx \mathbf v$ be an identity of $\mathbf V$ and $\{_2 x, {_2 y}\}$ be a critical pair in $\mathbf u \approx \mathbf v$. Then the critical pair $\{_2 x, {_2 y}\}$ is $\Delta$-removable for some $\Delta\subseteq\Psi_2(\mathbf V)$.
\end{proposition}

\begin{proof}
We consider only the case when $({_{1\mathbf u}x})<({_{1\mathbf u}y})$ since that case when $({_{1\mathbf u}y})<({_{1\mathbf u}x})$ is similar.
Then $\mathbf u = \mathbf w_1\stackrel{(1)}{x}\mathbf w_2\stackrel{(1)}{y}\mathbf w_3\stackrel{(2)}{x}\stackrel{(2)}{y}\mathbf w_4$ for some words $\mathbf w_1,\mathbf w_2,\mathbf w_3,\mathbf w_4$.
The inclusion $\mathbf{LRB}\subset\mathbf V$ and Lemma~\ref{L: word problem LRB} imply that $({_{1\mathbf v}x})<({_{1\mathbf v}y})$.
Let $\mathbf w$ be the word obtained from $\mathbf u$ by replacing ${_{2\mathbf u} x}\,{_{2\mathbf u} y}$ with ${_{2\mathbf u} y}\,{_{1\mathbf u} x}$.
Lemma~\ref{L: 2-limited word} allows us to assume that the identity $\mathbf u \approx \mathbf v$ is 2-limited.

In view of Remark~\ref{R: max decomposition}, there is a number $d$ such that the $d$-decomposition of $\mathbf u$ is maximal.
Let $\mathbf a$ denote the $d$-block of $\mathbf u$, which contains the critical pair $\{{_2 x}, {_2 y}\}$.
Then $\mathbf a = \mathbf a'\,{_{2\mathbf u} x}\,{_{2\mathbf u} y}\,\mathbf a''$ and $\mathbf u = \mathbf u'\mathbf a\mathbf u''$ for some $\mathbf a',\mathbf a',\mathbf u',\mathbf u'' \in \mathfrak A^\ast$.
Clearly, $\con(\mathbf a)\subseteq\mul(\mathbf u'\mathbf a)$.

\smallskip

\textit{Case }1: $D(\mathbf u,y)=\infty$.
Then the first occurrence of $y$ in $\mathbf u$ lies in the $d$-block $\mathbf a$.
This means that $\mathbf a' = \mathbf a_1\,y\,\mathbf w_3$ for some $\mathbf a_1\in \mathfrak A^\ast$ such that $y \notin\con(\mathbf a_1)$.

Suppose that $\mathbf w_2$ contains some simple letter $t$ of $\mathbf u$.
In view of the inclusion $\mathbf F_1\subset\mathbf V$ and Corollary~\ref{C: (k-1)-well-balanced}, the identity $\mathbf u \approx \mathbf v$ is $0$-well-balanced.
Hence $({_{1\mathbf v}x})<({_{1\mathbf v}t})<({_{1\mathbf v}y})$.
Then the $\mathbf V$ satisfies~\eqref{xtyxy=xtyyx} because it coincides with $\mathbf u(x,y,t) \approx \mathbf v(x,y,t)$.
Hence $\mathbf V$ is a subvariety of $\mathbf O_1$.
Now Proposition~\ref{P: 2nd occurrences O_1} applies and we conclude that $\mathbf u \approx \mathbf w$ follows from $\Psi_1(\mathbf V)$.
In view of Lemma~\ref{L: zeta_k,lambda_k^m,eta_k,mu_k^m}(i), $\mathbf P\{\nu_1\}\subseteq\mathbf P\{\nu_r\}$ for any $r\in\mathbb N$.
Then, since $\nu_1$ holds in $\mathbf V$, the identities in $\Psi_1(\mathbf V)$ follow from $\Psi_2(\mathbf V)$.
Therefore, $\Psi_2(\mathbf V)$ implies $\mathbf u \approx \mathbf w$.

So, it remains to consider the case when $\mathbf w_2$ does not contain simple letters of $\mathbf u$.
Let $\phi$ be the substitution $(x_0,x_1,z_\infty,z)\mapsto (\mathbf u'',\mathbf w_2,\mathbf a'',\mathbf w_3)$.
Then the identity $\phi(\nu_1)$ coincides with
\begin{equation}
\label{to 2nd identity of O_1}
x\mathbf w_2y\mathbf w_3xy(\mathbf a'')^2\mathbf w_3\mathbf u''\mathbf w_2 \approx x\mathbf w_2y\mathbf w_3yx(\mathbf a'')^2\mathbf w_3\mathbf u''\mathbf w_2.
\end{equation}
Then the identity $\mathbf u \approx \mathbf w$ follows from $\{\eqref{xsxt=xsxtx},\,\nu_1\}\subseteq \Psi_2(\mathbf V)$ because
$$
\mathbf u\stackrel{\eqref{xsxt=xsxtx}}\approx \mathbf w_1x\mathbf w_2y\mathbf w_3\,xy\,(\mathbf a'')^2\mathbf w_3\mathbf u''\mathbf w_2\stackrel{\eqref{to 2nd identity of O_1}}\approx \mathbf w_1x\mathbf w_2y\mathbf w_3\,yx\,(\mathbf a'')^2\mathbf w_3\mathbf u''\mathbf w_2\stackrel{\eqref{xsxt=xsxtx}}\approx \mathbf w,
$$
and we are done.

\smallskip

\textit{Case }2: $D(\mathbf u,y)<\infty$, say $D(\mathbf u,y)=s+1$ for some $s \in \mathbb N\cup\{0\}$.
This means that the depth of the subword $\mathbf w_3$ in $\mathbf u$ is equal to $s$.
If $\con(\mathbf w_2)\subseteq\mul(\mathbf u'\mathbf a)$, then the identity $\mathbf u \approx \mathbf w$ follows from $\{\eqref{xsxt=xsxtx},\,\eqref{xzyt xy z_infty z_infty z = xzyt yx z_infty z_infty z}\}\subseteq \Psi_2(\mathbf V)$ because
$$
\mathbf u \stackrel{\eqref{xsxt=xsxtx}}\approx \mathbf w_1x\mathbf w_2y\mathbf w_3xy(\mathbf a'')^2\mathbf w_2\mathbf u'' \stackrel{\eqref{xzyt xy z_infty z_infty z = xzyt yx z_infty z_infty z}}\approx \mathbf w_1x\mathbf w_2y\mathbf w_3yx(\mathbf a'')^2\mathbf w_2\mathbf u'' \stackrel{\eqref{xsxt=xsxtx}}\approx \mathbf w.
$$
So, we may assume that $\con(\mathbf w_2)\nsubseteq\mul(\mathbf u'\mathbf a)$.
Then there exist a letter $x_p$ and a number $p\in \mathbb N\cup\{0\}$ such that $D(\mathbf u,x_p)=p$; the letters ${_{1\mathbf u}x_p}$ and ${_{2\mathbf u}x_p}$ occur in the subwords $\mathbf w_2$ and $\mathbf u''$ of $\mathbf u$, respectively; and the depth of $\mathbf w_2$ in $\mathbf u$ is equal to $p$.
The choice of $x_p$ implies that  $p \le s+1$.
Further, by similar arguments as in the last two paragraphs of the proof of Proposition~\ref{P: 2nd occurrences O_1}, one can show that $\mathbf u \approx \mathbf v$ follows from $\Phi_2(\mathbf V)$.
\end{proof}

\begin{corollary}
\label{C: HFB O_i}
For any $i=1,2$, each variety from the interval $[\mathbf{LRB} \vee \mathbf F_1,\mathbf O_i]$ can be defined within $\mathbf O_i$ by identities from $\Phi\cup\Psi_i$.
Consequently, the variety $\mathbf O_i$ is \HFB.
\end{corollary}

\begin{proof}
Let $\mathbf V$ be a variety from $[\mathbf{LRB} \vee \mathbf F_1,\mathbf O_i]$.
The inclusion $\mathbf O_i\subseteq\mathbf P$ and Lemmas~\ref{L: word problem F_k} and~\ref{L: 2-limited word} imply that $\mathbf V$ can be defined within $\mathbf P$ by a set of 2-limited balanced identities.
In view of Lemma~\ref{L: word problem LRB}, $\ini(\mathbf u)=\ini(\mathbf v)$ for any identity $\mathbf u \approx \mathbf v$ of $\mathbf V$.
Then Lemma~\ref{L: fblemma} and Propositions~\ref{P: 1st and 2nd occurrences},~\ref{P: 2nd occurrences O_1} and~\ref{P: 2nd occurrences O_2} imply that the identity system $(\Phi\cup\Psi_i)(\mathbf V)$ forms an identity basis for $\mathbf V$.
It follows from Lemmas~\ref{L: gamma_k, delta_k^k,epsilon_k} and~\ref{L: zeta_k,lambda_k^m,eta_k,mu_k^m} that every subset of $\Phi\cup\Psi_i$ defines a {\FB} subvariety of $\mathbf P$.
Therefore, $\mathbf V$ is \FB.

So, it remains to show that if $\mathbf U$ is a subvariety of $\mathbf O_i$ that does not contain $\mathbf{LRB} \vee \mathbf F_1$, then $\mathbf U$ is \FB.
Clearly, either $\mathbf{LRB}\nsubseteq \mathbf U$ or $\mathbf F_1\nsubseteq \mathbf U$.
Then by Lemma~\ref{L: does not contain LRB,F_1}, either $\mathbf U \subseteq\mathbf F$ or $\mathbf U \subseteq\mathbf{LRB}\vee \mathbf C$.
Since $\mathbf F$ is {\HFB} by Lemma~\ref{L: L(F)} and $\mathbf{LRB}\vee \mathbf C$ is also {\HFB} \cite[Proposition~4.1]{Lee-12b}, $\mathbf U$ is {\HFB} in any case.
\end{proof}

\section{Limit subvarieties of $\mathbf P$}
\label{sec: limit varieties}

\subsection{The limit variety $\mathbf J_1$}

\begin{lemma}
\label{L: P{kappa_1} violates}
The variety $\mathbf P\{\kappa_1\}$ violates the identity~\eqref{xytxy=xytyx}.
\end{lemma}

\begin{proof}
Let $\mathbf u$ be a word such that $\mathbf P\{\kappa_1\}$ satisfies $xytxy \approx \mathbf u$ and $\ini_2(\mathbf u) = xytxy$.
In view of Proposition~\ref{P: deduction}, it suffices to show that if the identity $\mathbf u\approx \mathbf v$ is directly deducible from some identity $\mathbf s\approx\mathbf t\in\{\eqref{xsxt=xsxtx},\,\kappa_1\}$, that is, $\{\mathbf u,\mathbf v\}=\{\mathbf a\xi(\mathbf s)\mathbf b,\mathbf a\xi(\mathbf t)\mathbf b\}$ for some words $\mathbf a,\mathbf b\in\mathfrak A^\ast$ and some endomorphism $\xi$ of $\mathfrak A^\ast$, then $\ini_2(\mathbf v) = xytxy$.
Arguing by contradiction suppose that $\ini_2(\mathbf v) \ne xytxy$.
According to Lemmas~\ref{L: word problem LRB} and~\ref{L: word problem F_k} and the inclusion $\mathbf{LRB}\vee\mathbf F_1\subset\mathbf P\{\kappa_1\}$, $\ini_2(\mathbf v) = xytyx$.
In particular, $({_{2\mathbf v}y})<({_{2\mathbf v}x})$.

Clearly, $\mathbf s\approx\mathbf t$ does not coincide with~\eqref{xsxt=xsxtx} because it does not change the first and the second occurrences of letters in a word.
Suppose that $\mathbf s\approx\mathbf t$ coincides with $\kappa_1$.
We consider only the case when $(\mathbf s,\mathbf t)=(xx_1xx_0x_1,x^2x_1x_0x_1)$ because the other case is similar.
Since the left side of $\kappa_1$ differs from the right side of $\kappa_1$ only in swapping of the first occurrence of $x_1$ and the second occurrence of $x$ and $({_{2\mathbf v}y})<({_{2\mathbf v}x})$, there are three possibilities:
\begin{itemize}
\item $\xi({_{1\mathbf s}x_1})$ contains ${_{2\mathbf u}x}$ and $\xi({_{2\mathbf s}x})$ contains ${_{2\mathbf u}y}$;
\item ${_{2\mathbf u}x}{_{2\mathbf u}y}$ is a subword of $\xi({_{1\mathbf s}x_1})$ and $\xi({_{2\mathbf s}x})$ contains some non-first and non-second occurrence of $y$;
\item $\xi({_{1\mathbf s}x_1})$ contains ${_{1\mathbf u}x}$ and ${_{2\mathbf u}x}{_{2\mathbf u}y}$ is a subword of $\xi({_{2\mathbf s}x})$.
\end{itemize}
In any case, $\xi({_{1\mathbf s}x})$ contains ${_{1\mathbf u}y}$, while $\xi({_{2\mathbf s}x})$ contains some non-first occurrence of $y$.
Then $t \in \con(\xi(x_1x))$ because $({_{1\mathbf u}y})<({_{1\mathbf u}t})<({_{2\mathbf u}y})$.
But this is impossible because $t \in \simple(\mathbf u)$, while $x,x_1 \in \mul(\mathbf s)$.
Therefore, $\mathbf s\approx\mathbf t$ cannot coincide with $\kappa_1$ as well.
\end{proof}

For arbitrary $n \in \mathbb N$ and $\pi,\tau\in S_{2n}$, we put
$$
\begin{aligned}
\mathbf w_n[\pi,\tau]&=\biggl(\prod_{i=1}^n x_it_i\biggr) \,x\, \biggl(\prod_{i=1}^{2n} z_i\biggr) \,y\, \biggl(\prod_{i=n+1}^{2n} t_ix_i\biggr)\,txy\, \biggl(\prod_{i=1}^{2n} x_{i\pi}z_{i\tau}\biggr),\\
\mathbf w_n'[\pi,\tau]&=\biggl(\prod_{i=1}^n x_it_i\biggr) \,x\, \biggl(\prod_{i=1}^{2n} z_i\biggr) \,y\, \biggl(\prod_{i=n+1}^{2n} t_ix_i\biggr)\,tyx\, \biggl(\prod_{i=1}^{2n} x_{i\pi}z_{i\tau}\biggr).
\end{aligned}
$$
Let $\mathbf J_1 =\var \Omega_1$, where
$$
\Omega_1 = \{\eqref{xsxt=xsxtx},\,\kappa_1,\, \mathbf w_n[\pi,\tau] \approx \mathbf w_n'[\pi,\tau] \mid n \in \mathbb N, \ \pi,\tau\in S_{2n}\}.
$$

\begin{lemma}
\label{L: J_1 violates}
The variety $\mathbf J_1$ violates the identity
\begin{equation}
\label{xsytxy=xsytyx}
xsytxy \approx xsytyx.
\end{equation}
\end{lemma}

\begin{proof}
Let $\mathbf u$ be a word such that $\mathbf J_1$ satisfies $xsytxy \approx \mathbf u$ and $\ini_2(\mathbf u) = xsytxy$.
In view of Proposition~\ref{P: deduction}, to verify that $\mathbf J_1$ violates~\eqref{xsytxy=xsytyx}, it suffices to show that if the identity $\mathbf u\approx \mathbf v$ is directly deducible from some identity $\mathbf s\approx\mathbf t\in\Omega_1$, that is, $\{\mathbf u,\mathbf v\}=\{\mathbf a\xi(\mathbf s)\mathbf b,\mathbf a\xi(\mathbf t)\mathbf b\}$ for some words $\mathbf a,\mathbf b\in\mathfrak A^\ast$ and some endomorphism $\xi$ of $\mathfrak A^\ast$, then $\ini_2(\mathbf v) = xsytxy$.
Arguing by contradiction suppose that $\ini_2(\mathbf v) \ne xsytxy$.
According to Lemma~\ref{L: word problem F_k} and the inclusion $\mathbf F_1\subset\mathbf J_1$, $\ini_2(\mathbf v) = xsytyx$.
In particular, $({_{2\mathbf v}y})<({_{2\mathbf v}x})$.

In view of Lemma~\ref{L: P{kappa_1} violates}, $\mathbf s\approx\mathbf t$ cannot coincide with~\eqref{xsxt=xsxtx} or $\kappa_1$.
Therefore, $\mathbf s\approx\mathbf t$ coincides with $\mathbf w_k[\pi,\tau] \approx \mathbf w_k'[\pi,\tau]$ for some $k \in \mathbb N$ and $\pi,\tau \in S_{2k}$.
We consider only the case when $(\mathbf s,\mathbf t)=(\mathbf w_k[\pi,\tau],\mathbf w_k'[\pi,\tau])$ because the other case is similar.
Since $\mathbf w_k[\pi,\tau]$ differs from $\mathbf w_k'[\pi,\tau]$ only in swapping of the second occurrences of letters $x$ and $y$ and $({_{2\mathbf v}y})<({_{2\mathbf v}x})$, there are three possibilities:
\begin{itemize}
\item[(a)] $\xi({_{2\mathbf s}x})$ contains ${_{2\mathbf u}x}$ and $\xi({_{2\mathbf s}y})$ contains ${_{2\mathbf u}y}$;
\item[(b)] ${_{2\mathbf u}x}{_{2\mathbf u}y}$ is a subword of $\xi({_{2\mathbf s}x})$ and $\xi({_{2\mathbf s}y})$ contains some occurrence of $y$;
\item[(c)] $\xi({_{2\mathbf s}x})$ contains ${_{1\mathbf u}x}$ and ${_{2\mathbf u}x}{_{2\mathbf u}y}$ is a subword of $\xi({_{2\mathbf s}y})$.
\end{itemize}
Suppose that~(a) holds.
Then $\xi({_{1\mathbf s}x})$ contains ${_{1\mathbf u}x}$ and $\xi({_{1\mathbf s}y})$ contains ${_{1\mathbf u}y}$.
Therefore $s\in\con(\xi(xz_1z_2\cdots z_{2k}y))$, contradicting $s \in \simple(\mathbf u)$ and $\{x,y,z_1,z_2,\dots, z_{2k}\} \subset \mul(\mathbf s)$.
Suppose that~(b) holds.
Then $\xi({_{1\mathbf s}x})$ contains both ${_{1\mathbf u}x}$ and ${_{1\mathbf u}y}$.
Hence $s\in\con(\xi(x))$, which contradicts $s \in \simple(\mathbf u)$ and $x\in\mul(\mathbf s)$.
Finally,~(c) is impossible because $\xi({_{2\mathbf s}x})$ cannot contain any first occurrence of a letter in $\mathbf u$.
We see that  $\mathbf s\approx\mathbf t$ cannot coincide with $\mathbf w_k[\pi,\tau] \approx \mathbf w_k'[\pi,\tau]$ as well.
Therefore, $\mathbf J_1$ violates~\eqref{xsytxy=xsytyx}.
\end{proof}

\begin{lemma}
\label{L: V does not contain J_1}
Let $\mathbf V$ be a subvariety of $\mathbf P$.
If $\mathbf V$ does not contain $\mathbf J_1$, then $\mathbf V$ satisfies the identity~\eqref{xzyt xy z_infty z_infty z = xzyt yx z_infty z_infty z}.
\end{lemma}

\begin{proof}
If $\mathbf V$ does not contain  $\mathbf{LRB}$ or $\mathbf F_1$, then $\mathbf V$ is contained in either $\mathbf F$ or $\mathbf{LRB}\vee\mathbf C$  by Lemma~\ref{L: does not contain LRB,F_1}.
Evidently, both $\mathbf F$ and $\mathbf{LRB}\vee\mathbf C$ satisfy~\eqref{xzyt xy z_infty z_infty z = xzyt yx z_infty z_infty z}.
So, we may assume that $\mathbf V$ belongs to the interval $[\mathbf{LRB}\vee \mathbf F_1, \mathbf P]$.

In view of Lemma~\ref{L: word problem F_k} and the inclusion $\mathbf F_1 \subset \mathbf V$, we have $\simple(\mathbf a)=\simple(\mathbf b)$ and $\mul(\mathbf a)=\mul(\mathbf b)$ for any identity $\mathbf a \approx \mathbf b$ of $\mathbf V$.
This fact and Lemma~\ref{L: 2-limited word} imply that any identity of $\mathbf V$ is equivalent to a 2-limited balanced identity.
Then Lemma~\ref{L: fblemma} implies that there exists a 2-limited balanced identity $\mathbf u \approx \mathbf v$ of $\mathbf V$ with a critical pair $\{{_i x}, {_j y}\}$ that is not $\Omega_1$-removable.
It follows from the inclusion $\mathbf{LRB} \subset \mathbf V$ and Lemma~\ref{L: word problem LRB} that $(i,j) \ne (1,1)$.
Then $(i,j) \in \{(1,2),(2,1),(2,2)\}$ because the identity $\mathbf u \approx \mathbf v$ is 2-limited.

Suppose that $\{i,j\}=\{1,2\}$.
According to Proposition~\ref{P: 1st and 2nd occurrences}, the identity $\mathbf u \approx \mathbf v$ may be chosen from the set
$$
\Phi = \{\eqref{xsxt=xsxtx},\,\eqref{xyxy=xxyy},\,\gamma_k,\,\delta_k^m,\,\varepsilon_{k-1} \mid k \in \mathbb N, \ 1 \le m \le k \}.
$$
Evidently, $\mathbf u \approx \mathbf v$ does not coincide with~\eqref{xsxt=xsxtx} because it holds in $\mathbf J_1$.
Since $\mathbf J_1$ satisfies $\kappa_1$, it also satisfies $\delta_k^m$ and $\varepsilon_k$ for any $k \in \mathbb N$ and $1 \le m \le k$ by Lemmas~\ref{L: kappa_k and delta_k^k} and~\ref{L: gamma_k, delta_k^k,epsilon_k}(ii),(iii).
Besides that,~\eqref{xyxy=xxyy} holds in $\mathbf J_1$ because it is a consequence of $\kappa_1$.
It follows that $\mathbf u \approx \mathbf v$ coincides with either $\varepsilon_0$ or $\gamma_k$ for some $k \in \mathbb N$.
It is easy to see that~\eqref{xzyt xy z_infty z_infty z = xzyt yx z_infty z_infty z} is a consequence of $\varepsilon_0$ and so $\gamma_k$ for any $k \in \mathbb N$.

So, it remains to consider the case when $(i,j) = (2,2)$.
By symmetry, we may assume that $({_{1\mathbf u}x})<({_{1\mathbf u}y})$.
Then $({_{1\mathbf v}x})<({_{1\mathbf v}y})$ by Lemma~\ref{L: word problem LRB} and  the inclusion $\mathbf{LRB} \subset \mathbf V$.
If there are no simple letters between ${_{1\mathbf u}y}$ and ${_{2\mathbf u}y}$, then $D(\mathbf u,y)>1$ and so the critical pair $\{{_2 x}, {_2 y}\}$ is $\{\eqref{xsxt=xsxtx},\,\kappa_1\}$-removable in $\mathbf u \approx \mathbf v$ by Corollary~\ref{C: u'abu''=u'bau'' kappa_k}.
So, we may assume that there is a simple letter $t$ between ${_{1\mathbf u}y}$ and ${_{2\mathbf u}y}$.

If, for any letter $a$ with
$$
({_{1\mathbf u}x})<({_{1\mathbf u}a})<({_{1\mathbf u}y})<({_{2\mathbf u}y})<({_{2\mathbf u}a}),
$$
there are no first occurrences of letters between ${_{2\mathbf u}y}$ and ${_{2\mathbf u}a}$, then it is easy to see that one can choose $n\in \mathbb N$ and $\pi,\tau\in S_{2n}$ so that the critical pair $\{{_2 x}, {_2 y}\}$ is $\{\eqref{xsxt=xsxtx},\,\mathbf w_n[\pi,\tau] \approx \mathbf w_n'[\pi,\tau]\}$-removable in $\mathbf u \approx \mathbf v$.
So, we may assume that there are letters $z$ and $z_\infty$ such that
$$
({_{1\mathbf u}x})<({_{1\mathbf u}z})<({_{1\mathbf u}y})<({_{2\mathbf u}y})<({_{1\mathbf u}z_\infty})<({_{2\mathbf u}z}).
$$
Now we substitute~$yxz_\infty^2$ for $z_\infty$ in the identity $\mathbf u(x,y,t,z,z_\infty) \approx \mathbf v(x,y,t,z,z_\infty)$.
Since $({_{2\mathbf v}y})<({_{2\mathbf v}x})$, we obtain an identity that is equivalent modulo~\eqref{xsxt=xsxtx} to
$$
xzyt xy z_\infty^2 z \approx xzyt \mathbf w,
$$
where $\mathbf w\in\{zyxz_\infty^2,yzxz_\infty^2,yxzz_\infty^2,yxz_\infty^2z\}$.
If $\mathbf w=yxz_\infty^2z$, then $\mathbf V$ satisfies~\eqref{xzyt xy z_infty z_infty z = xzyt yx z_infty z_infty z}.
So, it remains to consider the case when
$$
\mathbf w\in\{zyxz_\infty^2,yzxz_\infty^2,yxzz_\infty^2\}.
$$
In this case, the variety $\mathbf V$ satisfies the identity
$$
xzyt \mathbf w \approx xzyt yx z z_\infty^2
$$
because this identity is a consequence of the identity~\eqref{xytxy=xytyx}, which is evidently satisfied by $\mathbf V$.
Now Proposition~\ref{P: 1st and 2nd occurrences} applies and we conclude that the identity $xzyt yx z z_\infty^2 \approx xzyt yx z_\infty z z_\infty$ follows from $\Phi(\mathbf V)$.
Hence $\mathbf V$ satisfies
\begin{equation}
\label{xzyt xy z_infty z_infty z = xzyt yx z_infty z z_infty}
xzyt xy z_\infty^2 z \approx xzyt yx z_\infty z z_\infty.
\end{equation}
Finally, we substitute~$z_\infty^2$ for $z_\infty$ in the identity~\eqref{xzyt xy z_infty z_infty z = xzyt yx z_infty z z_infty}.
We obtain an identity that is equivalent modulo~\eqref{xsxt=xsxtx} to~\eqref{xzyt xy z_infty z_infty z = xzyt yx z_infty z_infty z}.
We see that~\eqref{xzyt xy z_infty z_infty z = xzyt yx z_infty z_infty z} is satisfied by $\mathbf V$ in either case.
\end{proof}

\begin{proposition}
\label{P: J_1 is limit}
The variety $\mathbf J_1$ is a limit variety.
\end{proposition}

\begin{proof}
According to Lemmas~\ref{L: kappa_k and delta_k^k},~\ref{L: gamma_k, delta_k^k,epsilon_k}(ii) and~\ref{L: zeta_k,lambda_k^m,eta_k,mu_k^m}(i),(ii), $\nu_1$ is satisfied by $\mathbf J_1$.
Then Lemma~\ref{L: V does not contain J_1} implies that any proper subvariety of $\mathbf J_1$ is contained in $\mathbf O_2$ and so is {\FB} by Corollary~\ref{C: HFB O_i}.

So, it remains to establish that $\mathbf J_1$ is \NFB.
It suffices to verify that, for any $n \in \mathbb N$, the set of identities
$$
\Delta_n = \{\eqref{xsxt=xsxtx},\,\kappa_1,\, \mathbf w_k[\pi,\tau] \approx \mathbf w_k'[\pi,\tau] \mid k < n, \ \pi,\tau\in S_{2k}\}
$$
does not imply the identity $\mathbf w_n[\varepsilon,\varepsilon] \approx \mathbf w_n'[\varepsilon,\varepsilon]$, where $\varepsilon$ is the trivial permutation on $\{1,2,\dots,2n\}$.
Let $\mathbf w$ be a word such that $\mathbf J_1$ satisfies $\mathbf w_n[\varepsilon,\varepsilon] \approx \mathbf w$.
According to Lemmas~\ref{L: word problem LRB} and~\ref{L: word problem F_k} and the inclusion $\mathbf{LRB}\vee\mathbf F_1\subseteq\mathbf J_1$,
$$
\mathbf w= \biggl(\prod_{i=1}^n x_it_i\biggr) \,x\, \biggl(\prod_{i=1}^{2n} z_i\biggr) \,y\, \biggl(\prod_{i=n+1}^{2n} t_ix_i\biggr)\,t\mathbf w'
$$
for some word $\mathbf w'$ with $\con(\mathbf w')=\{x,y,x_1,z_1,x_2,z_2,\dots,x_{2n},z_{2n}\}$.
If ${_{1\mathbf w'}x}$ follows ${_{1\mathbf w'}x_1}$ in $\mathbf w'$, then the identity $\mathbf w_n[\varepsilon,\varepsilon](x,t,x_1,t_1) \approx \mathbf w(x,t,x_1,t_1)$ is equivalent modulo~\eqref{xsxt=xsxtx} to~\eqref{xsytxy=xsytyx}.
But this is impossible by Lemma~\ref{L: J_1 violates}.
Therefore, ${_{1\mathbf w'}x}$ precedes ${_{1\mathbf w'}x_1}$ in $\mathbf w'$.
By a similar argument we can show that $({_{1\mathbf w'}y})<({_{1\mathbf w'}x_1})$ and
$$
({_{1\mathbf w'}x_1})<({_{1\mathbf w'}z_1})<({_{1\mathbf w'}x_2})<({_{1\mathbf w'}z_2})<\cdots<({_{1\mathbf w'}x_{2n}})<({_{1\mathbf w'}z_{2n}}).
$$
This means $\ini(\mathbf w') = \mathbf ax_1z_1x_2z_2\cdots x_{2n}z_{2n}$, where $\mathbf a \in\{xy,yx\}$.
So, we have proved that $\ini_2(\mathbf w)\in \{\mathbf w_n[\varepsilon,\varepsilon],\mathbf w_n'[\varepsilon,\varepsilon]\}$.

Let $\mathbf u$ be a word such that $\mathbf J_1$ satisfies $\mathbf w_n[\varepsilon,\varepsilon] \approx \mathbf u$ and $\ini_2(\mathbf u) = \mathbf w_n[\varepsilon,\varepsilon]$.
According to Proposition~\ref{P: deduction}, to verify that $\mathbf w_n[\varepsilon,\varepsilon] \approx \mathbf w_n'[\varepsilon,\varepsilon]$ does not follow from $\Delta_n$, it suffices to establish that if the identity $\mathbf u\approx \mathbf v$ is directly deducible from some identity $\mathbf s\approx\mathbf t$ in $\Delta_n$, that is, $\{\mathbf u,\mathbf v\}=\{\mathbf a\xi(\mathbf s)\mathbf b,\mathbf a\xi(\mathbf t)\mathbf b\}$ for some words $\mathbf a,\mathbf b\in\mathfrak A^\ast$ and some endomorphism $\xi$ of $\mathfrak A^\ast$, then $\ini_2(\mathbf v) = \mathbf w_n[\varepsilon,\varepsilon]$.
Arguing by contradiction suppose that $\ini_2(\mathbf v) \ne \mathbf w_n[\varepsilon,\varepsilon]$.
In view of the previous paragraph, $\ini_2(\mathbf v) = \mathbf w_n'[\varepsilon,\varepsilon]$.
In particular, $({_{2\mathbf v}y})<({_{2\mathbf v}x})$.

Since $\ini_2(\mathbf u(x,y,t))=xytxy$, Lemma~\ref{L: P{kappa_1} violates} implies that $\mathbf s\approx\mathbf t$ cannot coincide with~\eqref{xsxt=xsxtx} and $\kappa_1$.
Suppose that $\mathbf s\approx\mathbf t$ coincides with $\mathbf w_k[\pi,\tau] \approx \mathbf w_k'[\pi,\tau]$ for some $k < n$ and $\pi,\tau \in S_{2k}$.
We consider only the case when $(\mathbf s,\mathbf t)=(\mathbf w_k[\pi,\tau],\mathbf w_k'[\pi,\tau])$ because the other case is similar.
Since $\mathbf w_k[\pi,\tau]$ differs from $\mathbf w_k'[\pi,\tau]$ only in swapping of the second occurrences of letters $x$ and $y$ and $({_{2\mathbf v}y})<({_{2\mathbf v}x})$, there are three possibilities:
\begin{itemize}
\item[(a)] $\xi({_{2\mathbf s}x})$ contains ${_{2\mathbf u}x}$ and $\xi({_{2\mathbf s}y})$ contains ${_{2\mathbf u}y}$;
\item[(b)] ${_{2\mathbf u}x}{_{2\mathbf u}y}$ is a subword of $\xi({_{2\mathbf s}x})$ and $\xi({_{2\mathbf s}y})$ contains some occurrence of $y$;
\item[(c)] $\xi({_{2\mathbf s}x})$ contains ${_{1\mathbf u}x}$ and ${_{2\mathbf u}x}{_{2\mathbf u}y}$ is a subword of $\xi({_{2\mathbf s}y})$.
\end{itemize}
Suppose that~(a) holds.
Then $\xi({_{1\mathbf s}x})$ contains ${_{1\mathbf u}x}$ and $\xi({_{1\mathbf s}y})$ contains ${_{1\mathbf u}y}$.
Hence $t_n,z_1\notin\con(\xi(x))$ and $z_{2n},t_{n+1}\notin\con(\xi(y))$ because $\ini_2(\mathbf u) \ne \mathbf w_n[\varepsilon,\varepsilon]$ otherwise.
Therefore, $\xi(x)=x$ and $\xi(y)=y$.
It follows that $\xi(\prod_{i=1}^{2k}z_i)=\prod_{i=1}^{2n}z_i$.
Then $\xi({_{1\mathbf s}z_p})$ contains ${_{1\mathbf u}z_q}\,{_{1\mathbf u}z_{q+1}}$ for some $p \le 2k$ and $q < 2n$.
Since $\ini_2(\mathbf u) = \mathbf w_n[\varepsilon,\varepsilon]$, the word $\xi({_{2\mathbf s}z_p})$ does not contain ${_{2\mathbf u}z_q}\,{_{2\mathbf u}z_{q+1}}$.
Hence ${_{2\mathbf u}z_q}$ is contained in the subword $\xi\bigl(\prod_{i=1}^{r-1}(x_{i\pi}z_{i\tau})x_{r\pi}\bigr)$ of $\mathbf u$, where $r\tau=p$.
Then ${_{1\mathbf u}z_q}$ must be contained in $\xi({_{1\mathbf s}a})$ for some $a \in \{z_{1\tau},x_{1\pi},z_{2\tau},x_{2\pi},\dots, z_{(r-1)\tau},x_{(r-1)\pi},x_{r\pi}\}$.
This contradicts the fact that ${_{1\mathbf u}z_q}$ is contained in $\xi({_{1\mathbf s}z_p})$.
Therefore,~(a) is impossible.
Suppose that~(b) holds.
Then $\xi({_{1\mathbf s}x})$ contains both ${_{1\mathbf u}x}$ and ${_{1\mathbf u}y}$.
This is only possible when $xz_1z_2\cdots z_{2n}y$ is a subword of $\xi(x)$.
But this contradicts the fact that $\ini_2(\mathbf u) = \mathbf w_n[\varepsilon,\varepsilon]$ and, therefore,~(b) is impossible as well.
Finally,~(c) is impossible because $\xi({_{2\mathbf s}x})$ cannot contain any first occurrence of a letter in $\mathbf u$.
We see that  $\mathbf s\approx\mathbf t$ cannot coincide with $\mathbf w_k[\pi,\tau] \approx \mathbf w_k'[\pi,\tau]$.
So, we have proved that $\Delta_n$ does not imply $\mathbf w_n[\varepsilon,\varepsilon] \approx \mathbf w_n'[\varepsilon,\varepsilon]$ and so $\mathbf J_1$ is \NFB.
\end{proof}

\subsection{The limit variety $\mathbf J_2$}

For arbitrary $n \in \mathbb N$ and $\pi,\tau\in S_{2n}$, define
$$
\begin{aligned}
\mathbf z_n[\pi,\tau]&=\biggl(\prod_{i=1}^{2n} x_it_i\biggr) \,x\, \biggl(\prod_{i=1}^n z_is_i\biggr) \,y\, \biggl(\prod_{i=1}^n z_{n+i}\biggr)\,xy\, \biggl(\prod_{i=1}^{2n} x_{i\pi}z_{i\tau}\biggr)\,t\,\biggl(\prod_{i=1}^ns_i\biggr),\\
\mathbf z_n'[\pi,\tau]&=\biggl(\prod_{i=1}^{2n} x_it_i\biggr) \,x\, \biggl(\prod_{i=1}^n z_is_i\biggr) \,y\, \biggl(\prod_{i=1}^n z_{n+i}\biggr)\,yx\, \biggl(\prod_{i=1}^{2n} x_{i\pi}z_{i\tau}\biggr)\,t\,\biggl(\prod_{i=1}^ns_i\biggr).
\end{aligned}
$$
Let $\mathbf J_2 =\var \Omega_2$, where
$$
\Omega_2 = \{\eqref{xsxt=xsxtx},\,\eta_1,\, \mathbf z_n[\pi,\tau] \approx \mathbf z_n'[\pi,\tau] \mid n \in \mathbb N, \ \pi,\tau\in S_{2n}\}.
$$

\begin{lemma}
\label{L: ({_{1u}x})<({_{2u}y}) in J_2}
Let $\mathbf u \approx \mathbf v$ be an identity of $\mathbf J_2$.
Then $({_{1\mathbf u}x})<({_{2\mathbf u}y})$ if and only if $({_{1\mathbf v}x})<({_{2\mathbf v}y})$ for any letters $x$ and $y$.
\end{lemma}

\begin{proof}
In view of Proposition~\ref{P: deduction}, it suffices to show that the required conclusion is true for an identity $\mathbf u\approx \mathbf v$ of $\mathbf J_2$ that is directly deducible from some identity $\mathbf s\approx\mathbf t\in\Omega_2$, that is, $\{\mathbf u,\mathbf v\}=\{\mathbf a\xi(\mathbf s)\mathbf b,\mathbf a\xi(\mathbf t)\mathbf b\}$ for some words $\mathbf a,\mathbf b\in\mathfrak A^\ast$ and some endomorphism $\xi$ of $\mathfrak A^\ast$.

If $\mathbf s\approx\mathbf t$ coincides with~\eqref{xsxt=xsxtx}, then the required conclusion is true because~\eqref{xsxt=xsxtx} does not change the first and the second occurrences of letters in a word.
Suppose that $\mathbf s\approx\mathbf t$ coincides with some identity from $\Omega_2\setminus\{\eqref{xsxt=xsxtx}\}$.
Then $\mathbf s$ differs from $\mathbf t$ only in swapping of the second occurrences of $x$ and $y$.
Clearly, both $\xi({_{2\mathbf s}x})$ and $\xi({_{2\mathbf s}y})$ cannot contain first occurrences of letters.
It follows that if $({_{1\mathbf u}x})<({_{2\mathbf u}y})$ for some letters $x$ and $y$, then $({_{1\mathbf v}x})<({_{2\mathbf v}y})$.
\end{proof}

\begin{lemma}
\label{L: P{eta_1} violates}
The variety $\mathbf P\{\eta_1\}$ violates the identity
\begin{equation}
\label{xx_1yxyx_0x_1=xx_1y^2xx_0x_1}
xx_1yxyx_0x_1 \approx xx_1y^2xx_0x_1.
\end{equation}
\end{lemma}

\begin{proof}
Let $\mathbf u$ be a word such that $\mathbf P\{\eta_1\}$ satisfies $xx_1yxyx_0x_1 \approx \mathbf u$ and $\ini_2(\mathbf u) = xx_1yxyx_0x_1$.
In view of Proposition~\ref{P: deduction}, it suffices to establish that if the identity $\mathbf u\approx \mathbf v$ is directly deducible from some identity $\mathbf s\approx\mathbf t\in\{\eqref{xsxt=xsxtx},\,\eta_1\}$, that is $\{\mathbf u,\mathbf v\}=\{\mathbf a\xi(\mathbf s)\mathbf b,\mathbf a\xi(\mathbf t)\mathbf b\}$ for some words $\mathbf a,\mathbf b\in\mathfrak A^\ast$ and some endomorphism $\xi$ of $\mathfrak A^\ast$, then $\ini_2(\mathbf v) = xx_1yxyx_0x_1$.
Arguing by contradiction suppose that $\ini_2(\mathbf v) \ne xx_1yxyx_0x_1$.
According to Lemmas~\ref{L: word problem LRB},~\ref{L: word problem F_k} and~\ref{L: ({_{1u}x})<({_{2u}y}) in J_2} and the inclusions $\mathbf{LRB}\vee\mathbf F_1\subseteq\mathbf J_2\subseteq\mathbf P\{\eta_1\}$, we have $\ini_2(\mathbf v) = xx_1y^2xx_0x_1$.
In particular, $({_{2\mathbf v}y})<({_{2\mathbf v}x})$.

Clearly, $\mathbf s\approx\mathbf t$ does not coincide with~\eqref{xsxt=xsxtx} because it does not change the first and the second occurrences of letters in a word.
Suppose that $\mathbf s\approx\mathbf t$ coincides with $\eta_1$.
We consider only the case when
$$
(\mathbf s,\mathbf t)=(xy_1yy_0xyx_1y_1x_0x_1,xy_1yy_0yxx_1y_1x_0x_1)
$$
because the other case is similar.
Since the left side of $\eta_1$ differs from the right side of $\eta_1$ only in swapping of the second occurrences of $x$ and $y$ and $({_{2\mathbf v}y})<({_{2\mathbf v}x})$, there are three possibilities:
\begin{itemize}
\item[(a)] $\xi({_{2\mathbf s}x})$ contains ${_{2\mathbf u}x}$ and $\xi({_{2\mathbf s}y})$ contains ${_{2\mathbf u}y}$;
\item[(b)] ${_{2\mathbf u}x}{_{2\mathbf u}y}$ is a subword of $\xi({_{2\mathbf s}x})$ and $\xi({_{2\mathbf s}y})$ contains some non-first and non-second occurrence of $y$;
\item[(c)] $\xi({_{2\mathbf s}x})$ contains ${_{1\mathbf u}x}$ and ${_{2\mathbf u}x}{_{2\mathbf u}y}$ is a subword of $\xi({_{2\mathbf s}y})$.
\end{itemize}
Suppose that~(a) holds.
Then $\xi({_{1\mathbf s}x})$ contains ${_{1\mathbf u}x}$ and $\xi({_{1\mathbf s}y})$ contains ${_{1\mathbf u}y}$.
It follows that $x_1\in\con(\xi(xy_1y))$.
Clearly, $x_1\notin\con(xy)$ because otherwise, some occurrence of $x_1$ lies between ${_{2\mathbf u}x}$ and ${_{2\mathbf u}y}$, contradicting $\ini_2(\mathbf u) = xx_1yxyx_0x_1$.
Therefore, ${_{1\mathbf u}x_1}$ coincides with $\xi({_{1\mathbf s}y_1})$ and so $\xi(x)=x$, $\xi(y)=y$ and $\xi(y_1)=x_1$.
It follows that $x_0 \in \con(\xi(x_1))$, contradicting $x_0 \in \simple(\mathbf u)$ and $x_1\in \mul(\mathbf s)$.
Suppose now that~(b) holds, then $\xi({_{1\mathbf s}x})$ contains both ${_{1\mathbf u}x}$ and ${_{1\mathbf u}y}$.
It follows that $x_1\in\con(\xi(x))$.
But this is impossible because there are no occurrences of $x_1$ between ${_{2\mathbf u}x}$ and ${_{2\mathbf u}y}$.
Finally,~(c) is impossible because $\xi({_{2\mathbf s}x})$ cannot contain any first occurrence of a letter in $\mathbf u$.
We see that $\mathbf s\approx\mathbf t$ cannot coincide with $\eta_1$ as well.
Therefore, $\mathbf P\{\eta_1\}$ violates~\eqref{xx_1yxyx_0x_1=xx_1y^2xx_0x_1}.
\end{proof}

\begin{lemma}
\label{L: J_2 violates}
The variety $\mathbf J_2$ violates the identity~\eqref{xtyxy=xtyyx}.
\end{lemma}

\begin{proof}
Let $\mathbf u$ be a word such that $\mathbf J_2$ satisfies $xtyxy \approx \mathbf u$ and $\ini_2(\mathbf u) = xtyxy$.
In view of Proposition~\ref{P: deduction}, to verify that $\mathbf J_2$ violates~\eqref{xtyxy=xtyyx}, it suffices to show that if the identity $\mathbf u\approx \mathbf v$ is directly deducible from some identity $\mathbf s\approx\mathbf t\in\Omega_2$, that is $\{\mathbf u,\mathbf v\}=\{\mathbf a\xi(\mathbf s)\mathbf b,\mathbf a\xi(\mathbf t)\mathbf b\}$ for some words $\mathbf a,\mathbf b\in\mathfrak A^\ast$ and some endomorphism $\xi$ of $\mathfrak A^\ast$, then $\ini_2(\mathbf v) = xtyxy$.
Arguing by contradiction suppose that $\ini_2(\mathbf v) \ne xtyxy$.
According to Lemmas~\ref{L: word problem LRB},~\ref{L: word problem F_k} and~\ref{L: ({_{1u}x})<({_{2u}y}) in J_2} and the inclusion $\mathbf{LRB}\vee\mathbf F_1\subset\mathbf J_2$, $\ini_2(\mathbf v) = xty^2x$.
In particular, $({_{2\mathbf v}y})<({_{2\mathbf v}x})$.

Since $\ini_2(\mathbf u)\approx \ini_2(\mathbf v)$ coincides with~\eqref{xtyxy=xtyyx}, in view of Lemma~\ref{L: P{eta_1} violates} and the fact that~\eqref{xtyxy=xtyyx} implies~\eqref{xx_1yxyx_0x_1=xx_1y^2xx_0x_1}, the identity $\mathbf s\approx\mathbf t$ cannot coincide with~\eqref{xsxt=xsxtx} or $\eta_1$.
Therefore, $\mathbf s\approx\mathbf t$ coincides with $\mathbf z_k[\pi,\tau] \approx \mathbf z_k'[\pi,\tau]$ for some $k \in \mathbb N$ and $\pi,\tau \in S_{2k}$.
We consider only the case when $(\mathbf s,\mathbf t)=(\mathbf z_k[\pi,\tau],\mathbf z_k'[\pi,\tau])$ because the other case is similar.
Since $\mathbf z_k[\pi,\tau]$ differs from $\mathbf z_k'[\pi,\tau]$ only in swapping of the second occurrences of letters $x$ and $y$ and $({_{2\mathbf v}y})<({_{2\mathbf v}x})$, there are three possibilities:
\begin{itemize}
\item[(a)] $\xi({_{2\mathbf s}x})$ contains ${_{2\mathbf u}x}$ and $\xi({_{2\mathbf s}y})$ contains ${_{2\mathbf u}y}$;
\item[(b)] ${_{2\mathbf u}x}{_{2\mathbf u}y}$ is a subword of $\xi({_{2\mathbf s}x})$ and $\xi({_{2\mathbf s}y})$ contains some occurrence of $y$;
\item[(c)] $\xi({_{2\mathbf s}x})$ contains ${_{1\mathbf u}x}$ and ${_{2\mathbf u}x}{_{2\mathbf u}y}$ is a subword of $\xi({_{2\mathbf s}y})$.
\end{itemize}
Suppose that~(a) holds.
Then $\xi({_{1\mathbf s}x})$ contains ${_{1\mathbf u}x}$ and $\xi({_{1\mathbf s}y})$ contains ${_{1\mathbf u}y}$, so that $t\in\con(\xi(xz_1s_1z_2s_2\cdots z_ks_ky))$, contradicting $\{x,y,z_1,s_1,\dots, z_k,s_k\} \subset \mul(\mathbf s)$ and $t \in \simple(\mathbf u)$.
Suppose that~(b) holds.
Then $\xi({_{1\mathbf s}x})$ contains both ${_{1\mathbf u}x}$ and ${_{1\mathbf u}y}$, so that $t\in\con(\xi(x))$, contradicting $t \in \simple(\mathbf u)$ and $x\in\mul(\mathbf s)$.
Finally,~(c) is impossible because $\xi({_{2\mathbf s}x})$ cannot contain any first occurrence of a letter in $\mathbf u$.
We see that $\mathbf s\approx\mathbf t$ cannot coincide with $\mathbf z_k[\pi,\tau] \approx \mathbf z_k'[\pi,\tau]$ as well.
Therefore, $\mathbf J_2$ violates~\eqref{xtyxy=xtyyx}.
\end{proof}

\begin{lemma}
\label{L: V does not contain J_2}
Let $\mathbf V$ be a subvariety of $\mathbf P$.
If $\mathbf V$ does not contain $\mathbf J_2$, then $\mathbf V$ satisfies one of the identities~\eqref{xtyxy=xtyyx} or $\nu_1$.
\end{lemma}

\begin{proof}
If $\mathbf V$ does not contain $\mathbf{LRB}$ or $\mathbf F$, then $\mathbf V$ is contained in either $\mathbf F$ or $\mathbf P\{\kappa_r\}$ for some $r\in\mathbb N$ by Lemmas~\ref{L: does not contain LRB,F_1}(i),~\ref{L: L(F)} and~\ref{L: V does not contain F_k}.
Evidently, $\mathbf F$ satisfies the identity~\eqref{xtyxy=xtyyx}, while $\mathbf P\{\kappa_r\}$ satisfies this identity by Corollary~\ref{C: u'abu''=u'bau'' kappa_k}.
So, we may assume that $\mathbf V$ belongs to the interval $[\mathbf{LRB}\vee \mathbf F, \mathbf P]$.
In particular, Corollary~\ref{C: (k-1)-well-balanced} implies that every identity of $\mathbf V$ is $r$-well-balanced for any $r \ge 0$.
In view of this fact and Lemma~\ref{L: 2-limited word}, any identity of $\mathbf V$ is equivalent to a 2-limited balanced identity.
Then Lemma~\ref{L: fblemma} implies that there exist a 2-limited balanced identity $\mathbf u \approx \mathbf v$ of $\mathbf V$ and a critical pair $\{{_i x}, {_j y}\}$ in $\mathbf u \approx \mathbf v$ that is not $\Omega_2$-removable in $\mathbf u \approx \mathbf v$.
It follows from the inclusion $\mathbf{LRB} \subset \mathbf V$ and Lemma~\ref{L: word problem LRB} that $(i,j) \ne (1,1)$.
Then $(i,j) \in \{(1,2),(2,1),(2,2)\}$ because the identity $\mathbf u \approx \mathbf v$ is 2-limited.

Suppose that $\{i,j\}=\{1,2\}$.
According to Proposition~\ref{P: 1st and 2nd occurrences}, the identity $\mathbf u \approx \mathbf v$ may be chosen from the set
$$
\Phi = \{\eqref{xsxt=xsxtx},\,\eqref{xyxy=xxyy},\,\gamma_k,\,\delta_k^m,\,\varepsilon_{k-1} \mid k \in \mathbb N, \ 1 \le m \le k \}.
$$
Evidently, $\mathbf u \approx \mathbf v$ does not coincide with~\eqref{xsxt=xsxtx} because it holds in $\mathbf J_2$.
It follows from Lemma~\ref{L: gamma_k, delta_k^k,epsilon_k} that every identity from $\Phi\setminus \{\eqref{xsxt=xsxtx}\}$ together with~\eqref{xsxt=xsxtx} imply the identity~\eqref{xyxy=xxyy}.
It is routine to verify that~\eqref{xtyxy=xtyyx} follows from $\{\eqref{xsxt=xsxtx},\,\eqref{xyxy=xxyy}\}$.
Therefore,~\eqref{xtyxy=xtyyx} is satisfied by $\mathbf V$, and we are done.

So, it remains to consider the case when $(i,j) = (2,2)$.
By symmetry, we may assume that $({_{1\mathbf u}x})<({_{1\mathbf u}y})$.
Then $({_{1\mathbf v}x})<({_{1\mathbf v}y})$ by Lemma~\ref{L: word problem LRB} and  the inclusion $\mathbf{LRB} \subset \mathbf V$.
If there is a simple letter $t$ between ${_{1\mathbf u}x}$ and ${_{1\mathbf u}y}$, then the identity $\mathbf u(x,y,t) \approx \mathbf v(x,y,t)$ coincides with~\eqref{xtyxy=xtyyx} because $\mathbf u \approx \mathbf v$ is $0$-well-balanced.
So, we may assume that there are no simple letters between ${_{1\mathbf u}x}$ and ${_{1\mathbf u}y}$.
If, for any letter $a$ with
$$
({_{1\mathbf u}x})<({_{1\mathbf u}a})<({_{1\mathbf u}y})<({_{2\mathbf u}y})<({_{2\mathbf u}a}),
$$
there are no simple letters between ${_{2\mathbf u}y}$ and ${_{2\mathbf u}a}$, then it is routine to show that the critical pair $\{{_2 x}, {_2 y}\}$ is $\{\eqref{xsxt=xsxtx},\,\eta_1\}$-removable.
So, we may assume that there are letters $x_0\in\simple(\mathbf u)$ and $x_1\in\mul(\mathbf u)$ such that
$$
({_{1\mathbf u}x})<({_{1\mathbf u}x_1})<({_{1\mathbf u}y})<({_{2\mathbf u}y})<({_{1\mathbf u}x_0})<({_{2\mathbf u}x_1}).
$$
If there is a simple letter $t$ between ${_{1\mathbf u}y}$ and ${_{2\mathbf u}x}$, then
the identity $\mathbf u(x_0,x_1,x,y,t) \approx \mathbf v(x_0,x_1,x,y,t)$ coincides with the identity $\eta_0$ (up to renaming of letters) because $\mathbf u \approx \mathbf v$ is $1$-well-balanced.
It is easy to see that $\eta_0$ implies $\nu_1$, and we are done.
So, it remains to consider the case when there are no simple letters between ${_{1\mathbf u}y}$ and ${_{2\mathbf u}x}$.
If, for any letter $a$ with
$$
({_{1\mathbf u}y})<({_{1\mathbf u}a})<({_{2\mathbf u}x})<({_{2\mathbf u}y})<({_{2\mathbf u}a}),
$$
there are no first occurrences of letters between ${_{2\mathbf u}y}$ and ${_{2\mathbf u}a}$, then it is easy to see that one can choose $n\in \mathbb N$ and $\pi,\tau\in S_{2n}$ so that the critical pair $\{{_2 x}, {_2 y}\}$ is $\{\eqref{xsxt=xsxtx},\,\mathbf z_n[\pi,\tau] \approx \mathbf z_n'[\pi,\tau]\}$-removable in $\mathbf u \approx \mathbf v$.
So, we may assume that there are letters $z$ and $z_\infty$ such that
$$
({_{1\mathbf u}y})<({_{1\mathbf u}z})<({_{2\mathbf u}x})<({_{2\mathbf u}y})<({_{1\mathbf u}z_\infty})<({_{2\mathbf u}z}).
$$
We may assume without loss of generality that $({_{1\mathbf u}z_\infty})<({_{2\mathbf u}z})<({_{1\mathbf u}x_0})$ because we can substitute $z_\infty z x_0$ for $x_0$.
Then we substitute $yxz_\infty^2$ for $z_\infty$ in the identity
$$
\mathbf u(x_0,x_1,x,y,z_\infty,z) \approx \mathbf v(x_0,x_1,x,y,z_\infty,z).
$$
Since the identity $\mathbf u \approx \mathbf v$ is $1$-well-balanced and $({_{2\mathbf v}y})<({_{2\mathbf v}x})$, we obtain an identity, which is equivalent modulo~\eqref{xsxt=xsxtx} to
$$
xx_1yzxyz_\infty^2zx_0x_1 \approx xx_1\mathbf c x_0x_1
$$
such that \eqref{conditions for almost nu_p} is true.
Now Lemma~\ref{L: V satisfies nu_p} applies and we conclude that $\nu_1$ is satisfied by $\mathbf V$.
We see that $\mathbf V$ satisfies either $\nu_1$ or~\eqref{xtyxy=xtyyx} in either case.
\end{proof}

\begin{proposition}
\label{P: J_2 is limit}
The variety $\mathbf J_2$ is a limit variety.
\end{proposition}

\begin{proof}
It is routine to verify that $\mathbf J_2$ satisfies~\eqref{xzyt xy z_infty z_infty z = xzyt yx z_infty z_infty z}.
Then Lemma~\ref{L: V does not contain J_2} implies that any proper subvariety of $\mathbf J_2$ is contained in ether $\mathbf O_1$ or $\mathbf O_2$ and so is {\FB} by Corollary~\ref{C: HFB O_i}.

So, it remains to establish that $\mathbf J_2$ is \NFB.
It suffices to verify that, for any $n \in \mathbb N$, the set of identities
$$
\Delta_n = \{\eqref{xsxt=xsxtx},\,\eta_1,\, \mathbf z_k[\pi,\tau] \approx \mathbf z_k'[\pi,\tau] \mid k < n, \ \pi,\tau\in S_{2k}\}
$$
does not imply the identity $\mathbf z_n[\varepsilon,\varepsilon] \approx \mathbf z_n'[\varepsilon,\varepsilon]$, where $\varepsilon$ is the trivial permutation on $\{1,2,\dots,2n\}$.
Let $\mathbf w$ be a word such that $\mathbf J_2$ satisfies $\mathbf z_n[\varepsilon,\varepsilon] \approx \mathbf w$.
According to the inclusion $\mathbf{LRB}\subset\mathbf J_2$ and Lemmas~\ref{L: word problem LRB} and~\ref{L: ({_{1u}x})<({_{2u}y}) in J_2},
$$
\mathbf w= \biggl(\prod_{i=1}^{2n} x_it_i\biggr) \,x\, \biggl(\prod_{i=1}^n z_is_i\biggr) \,y\, \biggl(\prod_{i=1}^n z_{n+i}\biggr)\,\mathbf w'\,t\,\mathbf w''
$$
for some words $\mathbf w'$ and $\mathbf w''$ such that $\con(\mathbf w')=\{x,y,x_1,z_1,x_2,z_2,\dots,x_{2n},z_{2n}\}$ and $\con(\mathbf w'')=\{s_1,s_2,\dots,s_n\}$.
If ${_{1\mathbf w'}x}$ follows ${_{1\mathbf w'}x_1}$ in $\mathbf w'$, then the identity $\mathbf w_n[\varepsilon,\varepsilon](x,x_1,t_1) \approx \mathbf w(x,x_1,t_1)$ is equivalent modulo~\eqref{xsxt=xsxtx} to~\eqref{xtyxy=xtyyx}.
But this is impossible by Lemma~\ref{L: J_2 violates}.
Therefore, ${_{1\mathbf w'}x}$ precedes ${_{1\mathbf w'}x_1}$ in $\mathbf w'$.
By a similar argument we can show that $({_{1\mathbf w'}y})<({_{1\mathbf w'}x_1})$ and
$$
({_{1\mathbf w'}x_1})<({_{1\mathbf w'}z_1})<({_{1\mathbf w'}x_2})<({_{1\mathbf w'}z_2})<\cdots<({_{1\mathbf w'}x_{2n}})<({_{1\mathbf w'}z_{2n}}).
$$
This means that $\ini(\mathbf w') = \mathbf ax_1z_1x_2z_2\cdots x_{2n}z_{2n}$, where $\mathbf a \in\{xy,yx\}$.
So, we have proved that $\ini_2(\mathbf w)=\mathbf z_n[\theta]$ for some $\theta \in S_n$, where
$$
\mathbf z_n[\theta]=\biggl(\prod_{i=1}^{2n} x_it_i\biggr) \,x\, \biggl(\prod_{i=1}^n z_is_i\biggr) \,y\, \biggl(\prod_{i=1}^n z_{n+i}\biggr)\,\mathbf a\, \biggl(\prod_{i=1}^{2n} x_iz_i\biggr)\,t\,\biggl(\prod_{i=1}^ns_{i\theta}\biggr).
$$

Let $\mathbf u$ be a word such that $\mathbf J_2$ satisfies $\mathbf z_n[\varepsilon,\varepsilon] \approx \mathbf u$ and $({_{2\mathbf u}x})<({_{2\mathbf u}y})$.
In view of the previous paragraph, $\ini_2(\mathbf u)=\mathbf z_n[\theta_1]$ for some $\theta_1 \in S_n$.
According to Proposition~\ref{P: deduction}, to verify that $\mathbf z_n[\varepsilon,\varepsilon] \approx \mathbf z_n'[\varepsilon,\varepsilon]$ does not follow from $\Delta_n$, it suffices to show that if an identity $\mathbf u\approx \mathbf v$ is directly deducible from some identity $\mathbf s\approx\mathbf t$ in $\Delta_n$, that is, $\{\mathbf u,\mathbf v\}=\{\mathbf a\xi(\mathbf s)\mathbf b,\mathbf a\xi(\mathbf t)\mathbf b\}$ for some words $\mathbf a,\mathbf b\in\mathfrak A^\ast$ and some endomorphism $\xi$ of $\mathfrak A^\ast$, then $({_{2\mathbf v}x})<({_{2\mathbf v}y})$.
Arguing by contradiction suppose that $({_{2\mathbf v}y})<({_{2\mathbf v}x})$.
Then, in view of the observation in the previous paragraph, $\ini_2(\mathbf v)=\mathbf z_n[\theta_2]$ for some $\theta_2 \in S_n$.

Further, since $\ini_2(\mathbf u(x,y,s_1,t))=xs_1yxyts_1$ and $\ini_2(\mathbf v(x,y,s_1,t))=xs_1y^2xts_1$, Lemma~\ref{L: P{eta_1} violates} implies that $\mathbf s\approx\mathbf t$ cannot coincide with~\eqref{xsxt=xsxtx} and $\eta_1$.
Suppose that $\mathbf s\approx\mathbf t$ coincides with $\mathbf z_k[\pi,\tau] \approx \mathbf z_k'[\pi,\tau]$ for some $k < n$ and $\pi,\tau \in S_{2k}$.
We consider only the case when $(\mathbf s,\mathbf t)=(\mathbf z_k[\pi,\tau],\mathbf z_k'[\pi,\tau])$ because the other case is similar.
Since $\mathbf z_k[\pi,\tau]$ differs from $\mathbf z_k'[\pi,\tau]$ only in swapping of the second occurrences of letters $x$ and $y$, there are three possibilities:
\begin{itemize}
\item[(a)] $\xi({_{2\mathbf s}x})$ contains ${_{2\mathbf u}x}$ and $\xi({_{2\mathbf s}y})$ contains ${_{2\mathbf u}y}$;
\item[(b)] ${_{2\mathbf u}x}{_{2\mathbf u}y}$ is a subword of $\xi({_{2\mathbf s}x})$ and $\xi({_{2\mathbf s}y})$ contains some occurrence of $y$;
\item[(c)] $\xi({_{2\mathbf s}x})$ contains ${_{1\mathbf u}x}$ and ${_{2\mathbf u}x}{_{2\mathbf u}y}$ is a subword of $\xi({_{2\mathbf s}y})$.
\end{itemize}
Suppose that~(a) holds.
Then $\xi({_{1\mathbf s}x})$ contains ${_{1\mathbf u}x}$ and $\xi({_{1\mathbf s}y})$ contains ${_{1\mathbf u}y}$.
Hence $t_{2n},z_1\notin\con(\xi(x))$ and $s_n,z_{n+1}\notin\con(\xi(y))$ because the equality $\ini_2(\mathbf u)=\mathbf z_n[\theta_1]$ is false otherwise.
Therefore, $\xi(x)=x$ and $\xi(y)=y$, whence $\xi(\prod_{i=1}^kz_{k+i})=\prod_{i=1}^nz_{n+i}$.
Then $\xi({_{1\mathbf s}z_p})$ contains ${_{1\mathbf u}z_q}\,{_{1\mathbf u}z_{q+1}}$ for some $k+1\le p \le 2k$ and $n+1\le q < 2n$.
Since $\ini_2(\mathbf u)=\mathbf z_n[\theta_1]$, the word $\xi({_{2\mathbf s}z_p})$ cannot contain ${_{2\mathbf u}z_q}$.
Then ${_{2\mathbf u}z_q}$ is contained in the subword $\xi\bigl(\prod_{i=1}^{r-1}(x_{i\pi}z_{i\tau})x_{r\pi}\bigr)$ of $\mathbf u$, where $r\tau=p$.
Hence ${_{1\mathbf u}z_q}$ must be contained in $\xi({_{1\mathbf s}a})$ for some
$$
a \in \{z_{1\tau},x_{1\pi},z_{2\tau},x_{2\pi},\dots, z_{(r-1)\tau},x_{(r-1)\pi},x_{r\pi}\}.
$$
This contradicts the fact that ${_{1\mathbf u}z_q}$ is contained in $\xi({_{1\mathbf s}z_p})$.
Therefore,~(a) is impossible.
Suppose that~(b) holds.
Then $\xi({_{1\mathbf s}x})$ contains both ${_{1\mathbf u}x}$ and ${_{1\mathbf u}y}$.
This is only possible when $xz_1s_1z_2s_2\cdots z_y$ is a subword of $\xi(x)$.
But this contradicts the fact that $\ini_2(\mathbf u)=\mathbf z_n[\theta_1]$ and, therefore,~(b) is impossible as well.
Finally,~(c) is impossible because $\xi({_{2\mathbf s}x})$ cannot contain any first occurrence of a letter in $\mathbf u$.
We see that  $\mathbf s\approx\mathbf t$ cannot coincide with $\mathbf z_k[\pi,\tau] \approx \mathbf z_k'[\pi,\tau]$.
So, we have proved that $\Delta_n$ does not imply $\mathbf z_n[\varepsilon,\varepsilon] \approx \mathbf z_n'[\varepsilon,\varepsilon]$ and so $\mathbf J_2$ is \NFB.
\end{proof}

The following fact provides the first example of a limit variety of monoids with infinitely many subvarieties.

\begin{corollary}
\label{P: J_2 has infinitely many subvarieties}
The limit variety $\mathbf J_2$ have countably infinitely many subvarieties.
\end{corollary}

\begin{proof}
Since $\mathbf J_2$ is a limit variety by Proposition~\ref{P: J_2 is limit}, it contains at most  countably infinitely many subvarieties.
It remains to note that $\mathbf J_2$ has infinitely many subvarieties because
$\mathbf F_1\subset\mathbf F_2\subset\cdots\subset\mathbf F_k\subset\cdots\subset \mathbf J_2$ by Lemma~\ref{L: word problem F_k}.
\end{proof}

\subsection{Classification of limit subvarieties of $\mathbf P$}

The following is the main result of the paper.

\begin{theorem}
\label{T: main result}
The following statements on any subvariety $\mathbf V$ of $\mathbf P$ are equivalent:
\begin{itemize}
\item[\textup{(i)}] $\mathbf V$ is \HFB;
\item[\textup{(ii)}] $\mathbf V \subseteq \mathbf O_1$ or $\mathbf V \subseteq \mathbf O_2$;
\item[\textup{(iii)}] $\mathbf J_1, \mathbf J_2 \nsubseteq \mathbf V$.
\end{itemize}
Consequently, $\mathbf J_1$ and $\mathbf J_1$ are the only limit subvarieties of $\mathbf P$.
\end{theorem}

\begin{proof}
The implication (ii)\,$\longrightarrow$\,(i) holds by Corollary~\ref{C: HFB O_i}, while the implication (i)\,$\longrightarrow$\,(iii) holds because the varieties $\mathbf J_1$ and $\mathbf J_2$ are {\NFB} by Propositions~\ref{P: J_1 is limit} and~\ref{P: J_2 is limit}, respectively.
By Lemmas~\ref{L: V does not contain J_1} and~\ref{L: V does not contain J_2}, the statement $\mathbf J_1, \mathbf J_2 \nsubseteq \mathbf V$ implies that $\mathbf V$ satisfies either $\{ \eqref{xzyt xy z_infty z_infty z = xzyt yx z_infty z_infty z}, \eqref{xtyxy=xtyyx} \}$ or $\{ \eqref{xzyt xy z_infty z_infty z = xzyt yx z_infty z_infty z}, \nu_1 \}$, whence either $\mathbf V \subseteq \mathbf O_1$ or $\mathbf V \subseteq \mathbf O_2$.
Therefore the implication (iii)\,$\longrightarrow$\,(ii) holds.
\end{proof}

\section{The monoid $P_2^1$}
\label{sec: the monoid P_2^1}

\subsection{The variety $\mathbf P_2^1$ is \HFB}
\label{subsec: P_2^1 is HFB}

\begin{proposition}
\label{P: P_2^1}
The monoid variety $\mathbf P_2^1$ is \HFB.
\end{proposition}

\begin{proof}
Since the identities~\eqref{xsxt=xsxtx} and~\eqref{xsytxy=xsytyx} constitute an identity basis for the variety $\mathbf P_2^1$ \cite[Corollary~6.6]{Lee-Li-11}, the inclusion $\mathbf P_2^1\subseteq\mathbf O_1$ is easily deduced.
Then $\mathbf P_2^1$ is {\HFB} because $\mathbf O_1$ is {\HFB} by Corollary~\ref{C: HFB O_i}.
\end{proof}

\begin{lemma}[Lee and Zhang~\cite{Lee-Zhang-14}]
\label{L: except P_2^1}
Up to isomorphism and anti-isomorphism, every monoid of order five or less, with the possible exception of $P_2^1$, generates a small {\HFB} variety.
\end{lemma}

\begin{corollary}
\label{C: five or less}
The variety generated by any monoid of order five is {\HFB} and so contains at most countably many subvarieties.\qed
\end{corollary}

In contrast, as we have mentioned in the introduction, there exist monoids of order six that generate varieties with continuum many subvarieties \cite{Jackson-Lee-18,Jackson-Zhang-21}.

\begin{remark}
Up to isomorphism and anti-isomorphism, every semigroup of order five or less that is different from $P_2^1$ generates a {\HFB} variety of semigroups~\cite{Lee-13}.
As observed in Remark~\ref{R: P2}, whether or not $P_2^1$ generates a {\HFB} variety of semigroups is presently an open question.
\end{remark}

\subsection{Subvarieties of $\mathbf P_2^1$}
\label{subsec: subvar P_2^1}

In this subsection, we describe the subvariety lattice of $\mathbf P_2^1$.
Recall from Subsection~\ref{subsec: critical pairs (1x,2y)} that
$$
\Phi = \{\eqref{xsxt=xsxtx},\,\eqref{xyxy=xxyy},\,\gamma_k,\,\delta_k^m,\,\varepsilon_{k-1} \mid k \in \mathbb N, \ 1 \le m \le k \}.
$$

\begin{proposition}
\label{P: subvar of P_2^1}
Each variety $\mathbf V$ from the interval $[\mathbf{LRB} \vee \mathbf F_1,\mathbf P_2^1]$ can be defined within $\mathbf P_2^1$ by identities from $\Phi$.
\end{proposition}

\begin{proof}
The inclusion $\mathbf F_1\subset \mathbf V$ and Lemma~\ref{L: word problem F_k} imply that any 2-limited identity of $\mathbf V$ is balanced.
In view of this fact and Lemma~\ref{L: 2-limited word}, $\mathbf V$ can be defined within $\mathbf P_2^1$ by a set of 2-limited balanced identities.

Let $\mathbf u \approx \mathbf v$ be a non-trivial 2-limited balanced identity of $\mathbf V$.
Then there exists a critical pair $\{{_i x}, {_j y}\}$ in $\mathbf u \approx \mathbf v$.
Let $\mathbf w$ denote the word obtained from $\mathbf u$ by replacing ${_{i\mathbf u} x}{_{j\mathbf u} y}$ with ${_{j\mathbf u} y}{_{i\mathbf u} x}$.
It follows from the inclusion $\mathbf{LRB} \subset \mathbf V$ and Lemma~\ref{L: word problem LRB} that $(i,j) \ne (1,1)$.
Then $(i,j) \in \{(1,2),(2,1),(2,2)\}$ because the identity $\mathbf u \approx \mathbf v$ is 2-limited.
If $(i,j)=(2,2)$, then the critical pair $\{{_i x}, {_j y}\}$ is $\{\eqref{xsytxy=xsytyx}\}$-removable.
Now let $\{i,j\}=\{1,2\}$.
Then the critical pair $\{{_i x}, {_j y}\}$ is $\Phi(\mathbf V)$-removable by Proposition~\ref{P: 1st and 2nd occurrences}.
We see that $\Phi(\mathbf V)\cup\{\eqref{xsytxy=xsytyx}\}$ implies $\mathbf u \approx \mathbf w$ in either case.
According to Lemma~\ref{L: fblemma}, $\mathbf u \approx \mathbf v$ is a consequence of $\Phi(\mathbf V)\cup\{\eqref{xsytxy=xsytyx}\}$ and, therefore, $\mathbf V$ is defined within $\mathbf P_2^1$ by $\Phi(\mathbf V)$.
\end{proof}

\begin{lemma}
\label{L: within P{gamma_{k+1}}}
For any $1 \le m\le k$, we have $\mathbf P\{\delta_k^m,\,\varepsilon_0\}\subseteq \mathbf P\{\gamma_{k+1}\}$.
\end{lemma}

\begin{proof}
By Lemmas~\ref{L: kappa_k and delta_k^k} and~\ref{L: gamma_k, delta_k^k,epsilon_k}(ii), the identity $\kappa_k$ is satisfied by $\mathbf P\{\delta_k^m,\,\varepsilon_0\}$.
Then the identity $\gamma_{k+1}$ holds in $\mathbf P\{\delta_k^m,\,\varepsilon_0\}$ because
$$
\begin{aligned}
&y_1y_0y_1x_{k+1}x_kx_{k+1}\mathbf b_k\stackrel{\kappa_k}\approx y_1y_0y_1x_{k+1}^2x_k\mathbf b_k \stackrel{\varepsilon_0}\approx
y_1y_0x_{k+1}y_1x_{k+1}x_k\mathbf b_k\\
\stackrel{\eqref{xsxt=xsxtx}}\approx
&y_1y_0x_{k+1}y_1x_{k+1}x_k\mathbf b_{k,m}y_1\mathbf b_{m-1}
\stackrel{\delta_k^m}\approx
y_1y_0x_{k+1}y_1x_kx_{k+1}\mathbf b_{k,m}y_1\mathbf b_{m-1}\\\stackrel{(1.1)}\approx
&y_1y_0x_{k+1}y_1x_kx_{k+1}\mathbf b_k.
\end{aligned}
$$
The lemma is thus proved.
\end{proof}

\begin{lemma}
\label{L: within P{delta_{k+1}^m}}
For any $1\le m \le \ell \le k$, we have $\mathbf P\{\delta_k^\ell,\,\varepsilon_m\}\subseteq \mathbf P\{\delta_{k+1}^m\}$.
\end{lemma}

\begin{proof}
For convenience, let $\mathbf b_{p-1,p}=\lambda$ for any $p\in\mathbb N$.
We notice that the identity $\kappa_k$ is satisfied by $\mathbf P\{\delta_k^k,\,\varepsilon_m\}$ by Lemmas~\ref{L: kappa_k and delta_k^k} and~\ref{L: gamma_k, delta_k^k,epsilon_k}(ii).
Then the identity $\delta_{k+1}^m$ holds in $\mathbf P\{\delta_k^k,\,\varepsilon_m\}$ because
$$
\begin{array}{lcl}
&&y_{m+1}y_my_{m+1}x_{k+1}x_kx_{k+1}\mathbf b_{k,m}y_m\mathbf b_{m-1}\\
&\stackrel{\eqref{xsxt=xsxtx}}\approx&
y_{m+1}y_my_{m+1}x_{k+1}x_kx_{k+1}\mathbf b_{k,m} \mathbf d_1\\
&\stackrel{\kappa_k}\approx&
y_{m+1}y_my_{m+1}x_{k+1}^2x_k\mathbf b_{k,m} \mathbf d_1\\
&\stackrel{\eqref{xsxt=xsxtx}}\approx&
y_{m+1}y_my_{m+1}x_{k+1}^2(x_k\mathbf b_{k,m}) \mathbf d_2\\
&\stackrel{\varepsilon_m}\approx&
y_{m+1}y_mx_{k+1}y_{m+1}x_{k+1}(x_k\mathbf b_{k,m}) \mathbf d_2\\
&\stackrel{\eqref{xsxt=xsxtx}}\approx&
y_{m+1}y_mx_{k+1}y_{m+1}x_{k+1}x_k\mathbf b_{k,k}y_{m+1}\mathbf b_{k-1,m} \mathbf d_1\\
&\stackrel{\delta_k^k}\approx&
y_{m+1}y_mx_{k+1}y_{m+1}x_kx_{k+1}\mathbf b_{k,k}y_{m+1}\mathbf b_{k-1,m} \mathbf d_1\\
&\stackrel{\eqref{xsxt=xsxtx}}\approx&
y_{m+1}y_mx_{k+1}y_{m+1}x_kx_{k+1}\mathbf b_{k,m}y_m\mathbf b_{m-1},
\end{array}
$$
where
$$
\begin{aligned}
&\mathbf d_1 =
\begin{cases}
y_m&\text{if }m=1,\\
(y_mx_{m-2})x_{m-1}&\text{if }m=2,\\
(y_mx_{m-2})x_{m-1}x_{m-3}(y_mx_{m-2})\mathbf b_{m-3}&\text{if }m\ge3,
\end{cases}
\\
&\mathbf d_2 =
\begin{cases}
y_m&\text{if }m=1,\\
x_{m-2}(x_k\mathbf b_{k,m})\mathbf b_{m-2}&\text{if }m\ge2.
\end{cases}
\end{aligned}
$$
To complete the proof, it remains to notice that $\mathbf P\{\delta_k^\ell,\,\varepsilon_m\}\subseteq\mathbf P\{\delta_k^k,\,\varepsilon_m\}$ by Lemma~\ref{L: gamma_k, delta_k^k,epsilon_k}(ii).
\end{proof}

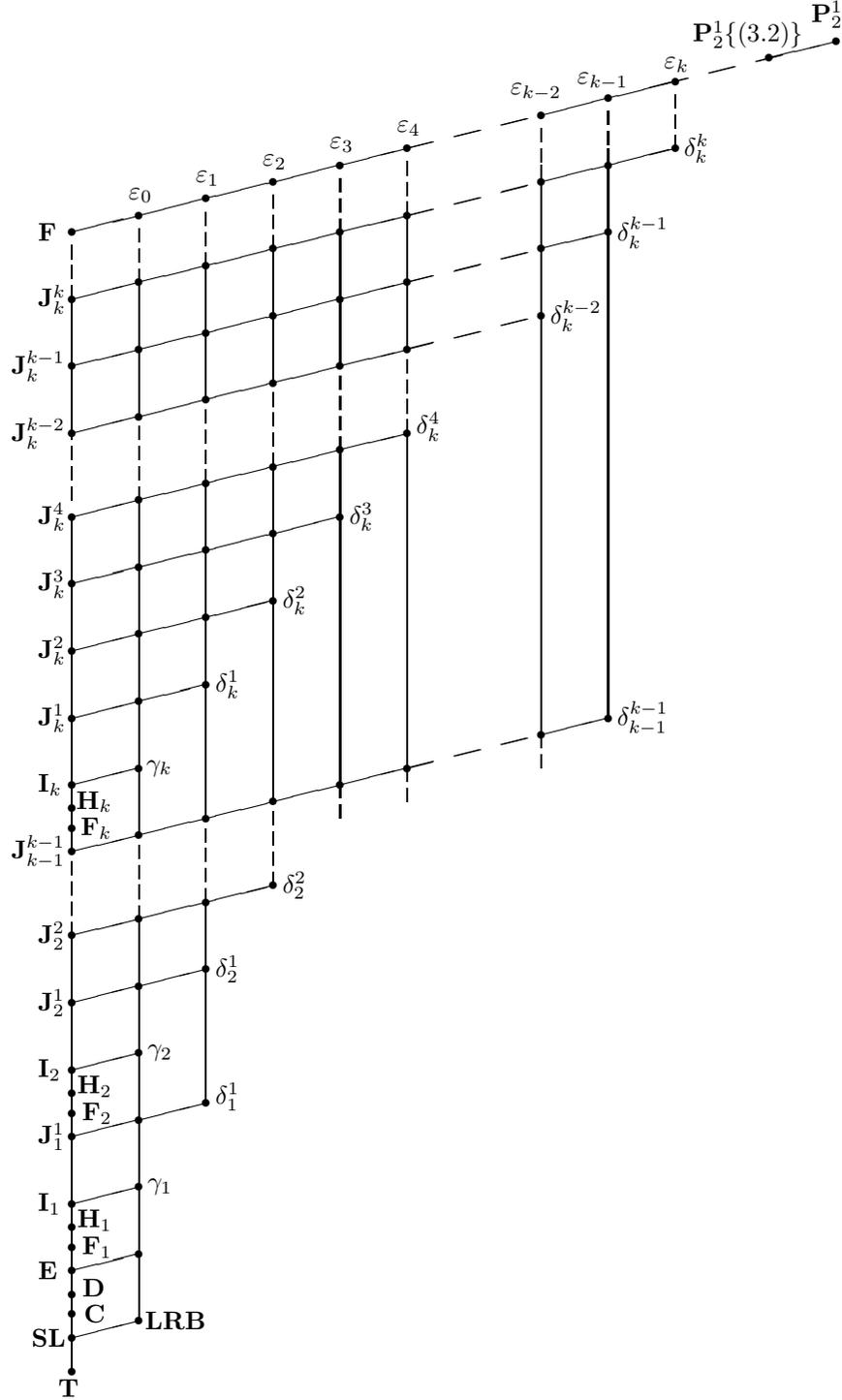
\begin{figure}
	\begin{center}
		\begin{picture}(125,200)
		\setlength{\unitlength}{0.92mm}
		\put(10,7.5){\circle*{1}}\put(10,2.5){\circle*{1}}%
        \put(20,10){\circle*{1}}\put(20,20){\circle*{1}}
        \put(20,30){\circle*{1}}\put(20,40){\circle*{1}}
        \put(30,42.5){\circle*{1}}%
        \put(20,50){\circle*{1}}\put(20,60){\circle*{1}}
        \put(30,62.5){\circle*{1}}%
        \put(20,70){\circle*{1}}\put(30,72.5){\circle*{1}}
        \put(40,75){\circle*{1}}
        \put(20,82.5){\circle*{1}}\put(30,85){\circle*{1}}
        \put(40,87.5){\circle*{1}}\put(50,90){\circle*{1}}
        \put(60,92.5){\circle*{1}}\put(80,97.5){\circle*{1}}
        \put(90,100){\circle*{1}}%
        \put(20,92.5){\circle*{1}}\put(20,102.5){\circle*{1}}
        \put(30,105){\circle*{1}}%
        \put(20,112.5){\circle*{1}}\put(30,115){\circle*{1}}
        \put(40,117.5){\circle*{1}}%
        \put(20,122.5){\circle*{1}}\put(30,125){\circle*{1}}
        \put(40,127.5){\circle*{1}}\put(50,130){\circle*{1}}%
        \put(20,132.5){\circle*{1}}\put(30,135){\circle*{1}}
        \put(40,137.5){\circle*{1}}\put(50,140){\circle*{1}}
        \put(60,142.5){\circle*{1}}%
        \put(20,145){\circle*{1}}\put(30,147.5){\circle*{1}}
        \put(40,150){\circle*{1}}\put(50,152.5){\circle*{1}}
        \put(60,155){\circle*{1}}\put(80,160){\circle*{1}}%
        \put(20,155){\circle*{1}}\put(30,157.5){\circle*{1}}
        \put(40,160){\circle*{1}}\put(50,162.5){\circle*{1}}
        \put(60,165){\circle*{1}}\put(80,170){\circle*{1}}
        \put(90,172.5){\circle*{1}}%
        \put(20,165){\circle*{1}}\put(30,167.5){\circle*{1}}
        \put(40,170){\circle*{1}}\put(50,172.5){\circle*{1}}
        \put(60,175){\circle*{1}}\put(80,180){\circle*{1}}
        \put(90,182.5){\circle*{1}}\put(100,185){\circle*{1}}%
        \put(20,175){\circle*{1}}\put(30,177.5){\circle*{1}}
        \put(40,180){\circle*{1}}\put(50,182.5){\circle*{1}}
        \put(60,185){\circle*{1}}\put(80,190){\circle*{1}}
        \put(90,192.5){\circle*{1}}\put(100,195){\circle*{1}}
        \put(114,198.5){\circle*{1}}\put(124,201){\circle*{1}}%
        \put(10,17.5){\circle*{1}}\put(10,27.5){\circle*{1}}
        \put(10,37.5){\circle*{1}}\put(10,47.5){\circle*{1}}
        \put(10,57.5){\circle*{1}}\put(10,67.5){\circle*{1}}
        \put(10,80){\circle*{1}}\put(10,90){\circle*{1}}
        \put(10,100){\circle*{1}}\put(10,110){\circle*{1}}
        \put(10,120){\circle*{1}}\put(10,130){\circle*{1}}
        \put(10,142.5){\circle*{1}}\put(10,152.5){\circle*{1}}
        \put(10,162.5){\circle*{1}}\put(10,172.5){\circle*{1}}%
        \put(10,11){\circle*{1}}\put(10,14){\circle*{1}}
        \put(10,21){\circle*{1}}\put(10,24){\circle*{1}}
        \put(10,41){\circle*{1}}\put(10,44){\circle*{1}}
        \put(10,83.5){\circle*{1}}\put(10,86.5){\circle*{1}}

        \put(20,10){\line(0,1){60}}\put(20,40){\line(4,1){10}}
        \put(20,60){\line(4,1){10}}\put(20,70){\line(4,1){20}}
        \put(30,42.5){\line(0,1){30}}%
        \put(20,82.5){\line(4,1){40}}\put(20,102.5){\line(4,1){10}}
        \put(20,112.5){\line(4,1){20}}\put(20,122.5){\line(4,1){30}}
        \put(20,132.5){\line(4,1){40}}\put(80,97.5){\line(4,1){10}}
        \put(20,82.5){\line(0,1){50}}\put(30,85){\line(0,1){50}}
        \put(40,87.5){\line(0,1){50}}\put(50,90){\line(0,1){50}}
        \put(60,92.5){\line(0,1){50}}\put(80,97.5){\line(0,1){85}}
        \put(90,100){\line(0,1){85}}%

        \put(20,145){\line(4,1){40}}\put(20,155){\line(4,1){40}}
        \put(20,165){\line(4,1){40}}
        \put(20,175){\line(4,1){40}}\put(80,170){\line(4,1){10}}
        \put(80,180){\line(4,1){20}}\put(80,190){\line(4,1){20}}
        \put(114,198.5){\line(4,1){10}}\put(20,145){\line(0,1){20}}
        \put(30,147.5){\line(0,1){20}}
        \put(40,150){\line(0,1){20}}
        \put(50,152.5){\line(0,1){20}}\put(60,155){\line(0,1){20}}%

        \put(10,2.5){\line(0,1){65}}\put(10,80){\line(0,1){50}}
        \put(10,142.5){\line(0,1){20}}
        \put(10,7.5){\line(4,1){10}}
        \put(10,17.5){\line(4,1){10}}\put(10,27.5){\line(4,1){10}}
        \put(10,37.5){\line(4,1){10}}\put(10,47.5){\line(4,1){10}}
        \put(10,57.5){\line(4,1){10}}\put(10,67.5){\line(4,1){10}}
        \put(10,80){\line(4,1){10}}\put(10,90){\line(4,1){10}}
        \put(10,100){\line(4,1){10}}
        \put(10,110){\line(4,1){10}}
        \put(10,120){\line(4,1){10}}\put(10,130){\line(4,1){10}}
        \put(10,142.5){\line(4,1){10}}\put(10,152.5){\line(4,1){10}}
        \put(10,162.5){\line(4,1){10}}\put(10,172.5){\line(4,1){10}}%
        \put(60,92.5){\line(4,1){4}}\put(67,94.3){\line(4,1){4}}
        \put(74,96){\line(4,1){4}}%

        \put(60,155){\line(4,1){4}}\put(67,156.7){\line(4,1){4}}
        \put(74,158.5){\line(4,1){4}}%
        \put(60,165){\line(4,1){4}}\put(67,166.7){\line(4,1){4}}
        \put(74,168.5){\line(4,1){4}}%
        \put(60,175){\line(4,1){4}}\put(67,176.7){\line(4,1){4}}
        \put(74,178.5){\line(4,1){4}}%
        \put(60,185){\line(4,1){4}}\put(67,186.7){\line(4,1){4}}
        \put(74,188.5){\line(4,1){4}}%
        \put(98,194.5){\line(4,1){4}}
        \put(105,196.2){\line(4,1){4}}
        \put(112,198){\line(4,1){4}}%

        \multiput(20,70)(0,3){4}{\line(0,1){2}}
        \multiput(30,72.5)(0,3){4}{\line(0,1){2}}
        \multiput(40,75)(0,3){4}{\line(0,1){2}}
        \multiput(50,85)(0,3){2}{\line(0,1){2}}
        \multiput(60,87.5)(0,3){2}{\line(0,1){2}}
        \multiput(80,92.5)(0,3){2}{\line(0,1){2}}%
        \multiput(20,132.5)(0,3){4}{\line(0,1){2}}
        \multiput(30,135)(0,3){4}{\line(0,1){2}}
        \multiput(40,137.5)(0,3){4}{\line(0,1){2}}
        \multiput(50,140)(0,3){4}{\line(0,1){2}}
        \multiput(60,142.5)(0,3){4}{\line(0,1){2}}
        \multiput(20,165)(0,3){3}{\line(0,1){2}}
        \multiput(30,167.5)(0,3){3}{\line(0,1){2}}
        \multiput(40,170)(0,3){3}{\line(0,1){2}}
        \multiput(50,172.5)(0,3){3}{\line(0,1){2}}
        \multiput(60,175)(0,3){3}{\line(0,1){2}}
        \multiput(80,180)(0,3){3}{\line(0,1){2}}
        \multiput(90,182.5)(0,3){3}{\line(0,1){2}}
        \multiput(100,185)(0,3){3}{\line(0,1){2}}%
        \multiput(10,162.5)(0,3){3}{\line(0,1){2}}
        \multiput(10,132.5)(0,3){4}{\line(0,1){2}}
        \multiput(10,67.5)(0,3){4}{\line(0,1){2}}

        \put(30,10){\makebox(0,0)[r]{$\mathbf{LRB}$}}
		\put(25,30){\makebox(0,0)[r]{$\gamma_1$}}
        \put(25,50){\makebox(0,0)[r]{$\gamma_2$}}
        \put(25,92.5){\makebox(0,0)[r]{$\gamma_k$}}%
        \put(22,178){\makebox(0,0)[r]{$\varepsilon_0$}}
        \put(32,180.5){\makebox(0,0)[r]{$\varepsilon_1$}}
        \put(42,183){\makebox(0,0)[r]{$\varepsilon_2$}}
        \put(52,185.5){\makebox(0,0)[r]{$\varepsilon_3$}}
        \put(62,188){\makebox(0,0)[r]{$\varepsilon_4$}}
        \put(83,193.5){\makebox(0,0)[r]{$\varepsilon_{k-2}$}}
        \put(93,195){\makebox(0,0)[r]{$\varepsilon_{k-1}$}}
        \put(102,197.5){\makebox(0,0)[r]{$\varepsilon_k$}}
        \put(119,202){\makebox(0,0)[r]{$\mathbf P_2^1\{\eqref{xyxy=xxyy}\}$}}
        \put(125,205){\makebox(0,0)[r]{$\mathbf P_2^1$}}%
        \put(35,43.5){\makebox(0,0)[r]{$\delta_1^1$}}
        \put(35,62.5){\makebox(0,0)[r]{$\delta_2^1$}}
        \put(35,105){\makebox(0,0)[r]{$\delta_k^1$}}
        \put(45,75){\makebox(0,0)[r]{$\delta_2^2$}}
        \put(45,117.5){\makebox(0,0)[r]{$\delta_k^2$}}
        \put(55,130){\makebox(0,0)[r]{$\delta_k^3$}}
        \put(65,143.5){\makebox(0,0)[r]{$\delta_k^4$}}
        \put(99,100){\makebox(0,0)[r]{$\delta_{k-1}^{k-1}$}}
        \put(89,160){\makebox(0,0)[r]{$\delta_{k}^{k-2}$}}
        \put(99,172.5){\makebox(0,0)[r]{$\delta_{k}^{k-1}$}}
        \put(105,185){\makebox(0,0)[r]{$\delta_{k}^{k}$}}%
        \put(8,0){\makebox(0,0)[l]{$\mathbf T$}}
        \put(4,7.5){\makebox(0,0)[l]{$\mathbf{SL}$}}
        \put(5,17.5){\makebox(0,0)[l]{$\mathbf E$}}
        \put(5,27.5){\makebox(0,0)[l]{$\mathbf I_1$}}
        \put(5,37.5){\makebox(0,0)[l]{$\mathbf J_1^1$}}
        \put(5,47.5){\makebox(0,0)[l]{$\mathbf I_2$}}
        \put(5,57.5){\makebox(0,0)[l]{$\mathbf J_2^1$}}
        \put(5,67.5){\makebox(0,0)[l]{$\mathbf J_2^2$}}
        \put(1,80){\makebox(0,0)[l]{$\mathbf J_{k-1}^{k-1}$}}
        \put(5,90){\makebox(0,0)[l]{$\mathbf I_k$}}
        \put(5,100){\makebox(0,0)[l]{$\mathbf J_k^1$}}
        \put(5,110){\makebox(0,0)[l]{$\mathbf J_k^2$}}
        \put(5,120){\makebox(0,0)[l]{$\mathbf J_k^3$}}
        \put(5,130){\makebox(0,0)[l]{$\mathbf J_k^4$}}
        \put(1,142.5){\makebox(0,0)[l]{$\mathbf J_{k}^{k-2}$}}
        \put(1,152.5){\makebox(0,0)[l]{$\mathbf J_{k}^{k-1}$}}
        \put(5,162.5){\makebox(0,0)[l]{$\mathbf J_{k}^{k}$}}
        \put(5,172.5){\makebox(0,0)[l]{$\mathbf F$}}%
        \put(15,11){\makebox(0,0)[r]{$\mathbf C$}}
        \put(15,15){\makebox(0,0)[r]{$\mathbf D$}}
        \put(16,21){\makebox(0,0)[r]{$\mathbf F_1$}}
        \put(16,25){\makebox(0,0)[r]{$\mathbf H_1$}}
        \put(16,41){\makebox(0,0)[r]{$\mathbf F_2$}}
        \put(16,45){\makebox(0,0)[r]{$\mathbf H_2$}}
        \put(16,83.5){\makebox(0,0)[r]{$\mathbf F_k$}}
        \put(16,87.5){\makebox(0,0)[r]{$\mathbf H_k$}}
		\end{picture}
	\end{center}
\caption{The lattice $\mathfrak L(\mathbf P_2^1)$}
\label{L(P_2^1)}
\end{figure}

\begin{proposition}
\label{P: L(P_2^1)}
The lattice $\mathfrak L(\mathbf P_2^1)$ is as shown in Fig.~\ref{L(P_2^1)}.\footnote{For convenience, some varieties in Fig.~\ref{L(P_2^1)} are marked by the identities, which define them within the variety $\mathbf P_2^1$.
For instance, the identity $\gamma_1$ marks the variety $\mathbf P_2^1\{\gamma_1\}$.}
\end{proposition}

\begin{proof}
Lemma~\ref{L: does not contain LRB,F_1}(ii) implies that the lattice $\mathfrak L(\mathbf P_2^1)$ is the disjoint union of the lattice $\mathfrak L(\mathbf{LRB}\vee \mathbf C)$ and the interval $[\mathbf F_1, \mathbf P_2^1]$.
As verified in Lee \cite[Fig.~2]{Lee-12}, the lattice $\mathfrak L(\mathbf{LRB}\vee \mathbf C)$ is as shown in Fig.~\ref{L(P_2^1)}.
In view of Lemma~\ref{L: does not contain LRB,F_1}(i), the interval $[\mathbf F_1, \mathbf P_2^1]$ is a disjoint union of the intervals $[\mathbf F_1, \mathbf F]$ and $[\mathbf{LRB}\vee\mathbf F_1, \mathbf P_2^1]$.
According to Lemma~\ref{L: L(F)}, the interval $[\mathbf F_1, \mathbf F]$ has the form shown in Fig.~\ref{L(P_2^1)}.
Since any identity from $\Phi\setminus\{\eqref{xsxt=xsxtx}\}$ together with~\eqref{xsxt=xsxtx} imply \eqref{xyxy=xxyy} by Lemma~\ref{L: gamma_k, delta_k^k,epsilon_k}, the variety $\mathbf P_2^1\{\eqref{xyxy=xxyy}\}$ is the greatest proper subvariety of $\mathbf P_2^1$.
(This result is also deducible from Lee and Li \cite[Lemma 6.7]{Lee-Li-11}.)
Let $\mathbf V$ be a variety from $[\mathbf{LRB}\vee\mathbf F_1,\mathbf P_2^1\{\eqref{xyxy=xxyy}\}]$.

Suppose that $\mathbf P_2^1\{\delta_1^1\}\nsubseteq\mathbf V$.
Since $\mathbf P_2^1\{\delta_1^1\}$ satisfies the identities \eqref{xyxy=xxyy}, $\delta_k^m$ and $\varepsilon_k$ for any $k\in\mathbb N$ and $1\le m\le k$ by Parts~(ii) and~(iii) of Lemma~\ref{L: gamma_k, delta_k^k,epsilon_k}, Proposition~\ref{P: subvar of P_2^1} implies that $\mathbf V$ satisfies $\varepsilon_0$ or $\gamma_\ell$ for some $\ell\in\mathbb N$.
Now we apply Lemma~\ref{L: gamma_k, delta_k^k,epsilon_k}(i) and obtain that $\mathbf P_2^1\{\gamma_\ell\}\subseteq\mathbf P_2^1\{\varepsilon_0\}$ for any $\ell\in\mathbb N$.
Hence $\mathbf V$ satisfies $\varepsilon_0$.
Thus, $\mathbf V$ belongs to the interval $[\mathbf{LRB}\vee\mathbf F_1,\mathbf P_2^1\{\varepsilon_0\}]$.

Suppose now that $\mathbf P_2^1\{\delta_k^k\}\subseteq\mathbf V$ for all $k \in \mathbb N$.
Then, since $\mathbf P_2^1\{\delta_k^k\}$ violates $\varepsilon_{k-1}$, Lemma~\ref{L: gamma_k, delta_k^k,epsilon_k} and Proposition~\ref{P: subvar of P_2^1} imply that $\mathbf V= \mathbf P_2^1\{\eqref{xyxy=xxyy}\}$.

Finally, suppose that $\mathbf P_2^1\{\delta_k^k\}\subseteq\mathbf V$ but $\mathbf P_2^1\{\delta_{k+1}^{k+1}\}\nsubseteq\mathbf V$ for some $k \in \mathbb N$.
Notice that
\begin{itemize}
\item[\textup(a)] $\mathbf J_s^1\nsubseteq\mathbf P_2^1\{\gamma_s\}$ but $\mathbf J_s^1\subseteq\mathbf P_2^1\{\delta_s^1\}$ for any $s\in\mathbb N$;
\item[\textup(b)] $\mathbf I_{s+1}\nsubseteq\mathbf P_2^1\{\delta_s^s\}$ but $\mathbf I_{s+1}\subset\mathbf P_2^1\{\gamma_{s+1}\}$ for any $s\in\mathbb N$;
\item[\textup(c)] $\mathbf J_s^{t+1}\nsubseteq\mathbf P_2^1\{\delta_s^t\}$ but $\mathbf J_s^{t+1}\subseteq\mathbf P_2^1\{\delta_s^{t+1}\}$ for any $1\le t<s$;
\item[\textup(d)] $\mathbf J_{s+1}^t\nsubseteq\mathbf P_2^1\{\delta_s^t\}$ but $\mathbf J_{s+1}^t\subseteq\mathbf P_2^1\{\delta_{s+1}^t\}$ for any $1\le t\le s$.
\end{itemize}
These facts together with Lemma~\ref{L: gamma_k, delta_k^k,epsilon_k}(ii),(iii) imply that $\mathbf P_2^1\{\delta_k^k\}$ violates the identities $\gamma_p$ and $\delta_p^q$, where $p \in \mathbb N$ and $1\le q<k$.
Now we apply Lemma~\ref{L: gamma_k, delta_k^k,epsilon_k}(ii),(iii) again and conclude that $\mathbf P_2^1\{\delta_{k+1}^{k+1}\}$ satisfies $\delta_r^m$ and $\varepsilon_r$ for all $r$ and $m$ such that $k+1\le m\le r$.
Then, since $\mathbf P_2^1\{\delta_k^k\}\subseteq\mathbf V$ and $\mathbf P_2^1\{\delta_{k+1}^{k+1}\}\nsubseteq\mathbf V$, Proposition~\ref{P: subvar of P_2^1} implies that $\mathbf V$ satisfies one of the identities $\varepsilon_k$ or $\delta_p^k$ for some $p\ge k$.
Now we apply Lemma~\ref{L: gamma_k, delta_k^k,epsilon_k}(ii) again and obtain that $\mathbf P_2^1\{\delta_p^k\}\subseteq\mathbf P_2^1\{\varepsilon_k\}$.
Hence $\mathbf V$ satisfies $\varepsilon_k$.
Thus, $\mathbf V$ belongs to the interval $[\mathbf P_2^1\{\delta_k^k\},\mathbf P_2^1\{\varepsilon_k\}]$.

We see that the interval $[\mathbf{LRB}\vee\mathbf F_1,\mathbf P_2^1\{\eqref{xyxy=xxyy}\}]$ is a disjoint union of the intervals $[\mathbf{LRB}\vee\mathbf F_1, \mathbf P_2^1\{\varepsilon_0\}]$, $\{\mathbf P_2^1\{\eqref{xyxy=xxyy}\}\}$ and $[\mathbf P_2^1\{\delta_k^k\}, \mathbf P_2^1\{\varepsilon_k\}]$ for all $k\in\mathbb N$.
So, it remains to verify that these intervals are as shown in Fig.~\ref{L(P_2^1)}.

First, we will show that the interval $[\mathbf{LRB}\vee\mathbf F_1, \mathbf P_2^1\{\varepsilon_0\}]$ has the form shown in Fig.~\ref{L(P_2^1)}.
In view of Lemma~\ref{L: gamma_k, delta_k^k,epsilon_k}, $\mathbf P_2^1\{\varepsilon_0\}$ satisfies $\{\eqref{xyxy=xxyy},\,\varepsilon_r\mid r\in\mathbb N\}$.
Then Proposition~\ref{P: subvar of P_2^1} implies that every variety from the interval $[\mathbf{LRB}\vee\mathbf F_1, \mathbf P_2^1\{\varepsilon_0\}]$ can be defined within $\mathbf P_2^1\{\varepsilon_0\}$ by some possibly empty subset of $\{\gamma_p,\delta_p^q\mid1\le q\le p\}$; in particular $\mathbf{LRB}\vee\mathbf F_1=\mathbf P_2^1\{\gamma_1\}$.
It follows from Lemma~\ref{L: within P{gamma_{k+1}}} that $\mathbf P_2^1\{\delta_p^q,\varepsilon_0\}\subseteq\mathbf P_2^1\{\gamma_{p+1}\}$ for any $1\le q\le p$.
Besides that, Lemma~\ref{L: gamma_k, delta_k^k,epsilon_k}(ii) implies that $\mathbf P_2^1\{\gamma_p\}\subseteq\mathbf P_2^1\{\delta_p^q,\,\varepsilon_0\}$.
Therefore, the interval $[\mathbf{LRB}\vee\mathbf F_1, \mathbf P_2^1\{\varepsilon_0\}]$ is a union of the singleton interval $\{\mathbf P_2^1\{\varepsilon_0\}\}$ and the intervals of the form $[\mathbf P_2^1\{\gamma_p\},\mathbf P_2^1\{\gamma_{p+1}\}]$, where $p \in\mathbb N$.
In view of Lemma~\ref{L: gamma_k, delta_k^k,epsilon_k}, $\mathbf P_2^1\{\gamma_{p+1}\}$ satisfies $\gamma_s$, $\delta_s^t$ and $\varepsilon_r$ for all $s$, $t$ and $r$ such that $p+1\le s$, $1\le t \le s$ and $r\ge0$.
It follows from (a)--(d) that $\mathbf P_2^1\{\gamma_p\}$ violates $\gamma_s$ and $\delta_s^t$ for all $s$ and $t$ such that $1\le t \le s < p$.
In view of these facts and Proposition~\ref{P: subvar of P_2^1}, every variety from the interval $[\mathbf P_2^1\{\gamma_p\},\mathbf P_2^1\{\gamma_{p+1}\}]$ can be defined within $\mathbf P_2^1\{\gamma_{p+1}\}$ by some possibly empty subset of $\{\gamma_p,\delta_p^1,\delta_p^2,\dots,\delta_p^p\}$.
Now we apply Lemma~\ref{L: gamma_k, delta_k^k,epsilon_k} again and obtain that the interval $[\mathbf P_2^1\{\gamma_p\},\mathbf P_2^1\{\gamma_{p+1}\}]$ forms the chain
$$
\mathbf P_2^1\{\gamma_p\}\subseteq \mathbf P_2^1\{\delta_p^1,\gamma_{p+1}\} \subseteq \mathbf P_2^1\{\delta_p^2,\gamma_{p+1}\}\subseteq \cdots \subseteq\mathbf P_2^1\{\delta_p^p,\gamma_{p+1}\}\subseteq \mathbf P_2^1\{\gamma_{p+1}\}.
$$
It remains to notice that all these inclusions are strict by (a)--(d).
We see that the interval $[\mathbf P_2^1\{\gamma_p\},\mathbf P_2^1\{\gamma_{p+1}\}]$ and, therefore, the interval $[\mathbf{LRB}\vee\mathbf F_1, \mathbf P_2^1\{\varepsilon_0\}]$ are as shown in Fig.~\ref{L(P_2^1)}.

It remains to show that, for any $k\in\mathbb N$, the interval $[\mathbf P_2^1\{\delta_k^k\}, \mathbf P_2^1\{\varepsilon_k\}]$ has the form shown in Fig.~\ref{L(P_2^1)}.
The proof of this fact is very similar to the arguments from the previous paragraph but we provide it for the sake of completeness.
It follows from (a)--(d) that $\mathbf P_2^1\{\delta_k^k\}$ violates $\gamma_r$, $\delta_s^t$ and $\varepsilon_t$ for all $r$, $s$ and $t$ such that $r\ge 0$, $s\ge1$, $t\le s$ and $1\le t<k$.
In view of Lemma~\ref{L: gamma_k, delta_k^k,epsilon_k}, $\mathbf P_2^1\{\varepsilon_k\}$ satisfies $\{\eqref{xyxy=xxyy},\,\varepsilon_r\mid r\ge k\}$.
Then Proposition~\ref{P: subvar of P_2^1} implies that every variety from the interval $[\mathbf P_2^1\{\delta_k^k\}, \mathbf P_2^1\{\varepsilon_k\}]$ can be defined within $\mathbf P_2^1\{\varepsilon_k\}$ by some possibly empty subset of $\{\delta_p^q\mid k\le q\le p\}$.
It follows from Lemma~\ref{L: within P{delta_{k+1}^m}} that $\mathbf P_2^1\{\delta_p^q,\varepsilon_k\}\subseteq\mathbf P_2^1\{\delta_{p+1}^k\}$ for any $k\le q\le p$.
Besides that, Lemma~\ref{L: gamma_k, delta_k^k,epsilon_k} implies that $\mathbf P_2^1\{\delta_p^k\}\subseteq\mathbf P_2^1\{\delta_p^q,\varepsilon_k\}$.
Therefore, the interval $[\mathbf P_2^1\{\delta_k^k\}, \mathbf P_2^1\{\varepsilon_k\}]$ is the union of the singleton interval $\{\mathbf P_2^1\{\varepsilon_k\}\}$ and the intervals of the form $[\mathbf P_2^1\{\delta_p^k\},\mathbf P_2^1\{\delta_{p+1}^k\}]$, where $p \ge k$.
Finally, Lemma~\ref{L: gamma_k, delta_k^k,epsilon_k}, Proposition~\ref{P: subvar of P_2^1} and (a)--(d) imply that every variety from the interval $[\mathbf P_2^1\{\delta_p^k\},\mathbf P_2^1\{\delta_{p+1}^k\}]$ can be defined within $\mathbf P_2^1\{\delta_{p+1}^k\}$ by some possibly empty subset of $\{\delta_p^k,\delta_p^{k+1},\dots,\delta_p^p\}$.
Now we apply Lemma~\ref{L: gamma_k, delta_k^k,epsilon_k} again and obtain that the interval $[\mathbf P_2^1\{\delta_p^k\},\,\mathbf P_2^1\{\delta_{p+1}^k\}]$ forms the chain
$$
\mathbf P_2^1\{\delta_p^k\}\subseteq \mathbf P_2^1\{\delta_p^{k+1},\,\delta_{p+1}^k\} \subseteq \mathbf P_2^1\{\delta_p^{k+2},\,\delta_{p+1}^k\}\subseteq \cdots \subseteq \mathbf P_2^1\{\delta_p^p,\,\delta_{p+1}^k\}\subseteq \mathbf P_2^1\{\delta_{p+1}^k\}.
$$
It remains to notice that all these inclusions are strict by (a)--(d).
We see that the interval $[\mathbf P_2^1\{\delta_p^k\},\mathbf P_2^1\{\delta_{p+1}^k\}]$ and so the interval $[\mathbf P_2^1\{\delta_k^k\}, \mathbf P_2^1\{\varepsilon_k\}]$ are as shown in Fig.~\ref{L(P_2^1)}.
\end{proof}

\subsection{Finitely universal varieties}
\label{subsec: fin uni}

Following Shevrin et al.~\cite{Shevrin-Vernikov-Volkov-09}, a variety of algebras is \textit{finitely universal} if every finite lattice is order-embeddable into its subvariety lattice.
The first example of a finitely universal variety of semigroups was presented by Burris and Nelson~\cite{Burris-Nelson-71} half a century ago.
In contrast, the first example of a finitely universal variety of monoids was
just recently found, in particular, there exist finitely universal varieties of monoids that are finitely generated \cite[Section 5]{Gusev-Lee-20}.
In view of this result, it is natural to find the least order of a monoid that generates a finitely universal variety.
The following statement shows that this order is more than five.

\begin{proposition}
\label{P: fin uni of order five}
Every monoid of order five or less generates a variety that is not finitely universal.
\end{proposition}

\begin{proof}
In view of Lemma~\ref{L: except P_2^1}, it suffices to show that $\mathbf P_2^1$ is not finitely universal.
It is clear from Fig.~\ref{L(P_2^1)} that every proper subvariety of $\mathbf P_2^1$ has at most two coverings, so that $\mathbf P_2^1$ is indeed not finitely universal.
For instance, the modular non-distributive lattice in Fig.~\ref{5-element lattice} is not embeddable in the lattice $\mathfrak L(\mathbf P_2^1)$.
\end{proof}

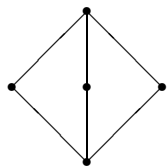
\begin{figure}[htb]
\unitlength=1mm
\linethickness{0.4pt}
\begin{center}
\begin{picture}(20,20)
\put(0,10){\circle*{1}}
\put(10,0){\circle*{1}}
\put(10,10){\circle*{1}}
\put(10,20){\circle*{1}}
\put(20,10){\circle*{1}}
\put(0,10){\line(1,1){10}}
\put(10,0){\line(1,1){10}}
\put(10,0){\line(0,1){20}}
\put(10,0){\line(-1,1){10}}
\put(20,10){\line(-1,1){10}}
\end{picture}
\end{center}
\caption{A modular non-distributive lattice}
\label{5-element lattice}
\end{figure}

\begin{remark}
In contrast, up to isomorphism and anti-isomorphism, there exist precisely four distinct semigroups of order four that generate a finitely universal variety, while the variety generated by any other semigroup of order four or less is not finitely universal~\cite{Lee-07}.
\end{remark}

\section*{APPENDIX. Tables of identities}

\begin{table}[tbh]
\caption{Identities labeled by Greek letters}
\begin{center}
\begin{tabular}{|c|c|}
\hline
Identity & Page number \\
\hline
$\alpha_k$ & \pageref{alpha_k} \\
$\beta_k$ & \pageref{beta_k} \\
$\gamma_k$ & \pageref{gamma_k} \\
$\delta_k^m$ & \pageref{delta_k^m} \\
$\varepsilon_k$ & \pageref{varepsilon_k} \\
$\zeta_k$ & \pageref{zeta_k}\\
$\eta_k$ & \pageref{eta_k}\\
$\kappa_k$ & \pageref{kappa_k} \\
$\lambda_k^m$ & \pageref{lambda_k^m}\\
$\mu_k^m$ & \pageref{mu_k^m}\\
$\nu_k$ & \pageref{nu_k}\\
\hline
\end{tabular}
\end{center}
\label{Greek letters}
\end{table}

\begin{table}[tbh]
\caption{``Non-local'' identities labeled by numbers}
\begin{center}
\begin{tabular}{|c|c|}
\hline
Identity & Page number \\
\hline
\eqref{xsxt=xsxtx} &  \pageref{xsxt=xsxtx} \\
\eqref{xyxy=xxyy} & \pageref{xyxy=xxyy} \\
\eqref{xtyxy=xtyyx} & \pageref{xtyxy=xtyyx} \\
\eqref{xytxy=xytyx} & \pageref{xytxy=xytyx} \\
\eqref{xzyt xy z_infty z_infty z = xzyt yx z_infty z_infty z} & \pageref{xzyt xy z_infty z_infty z = xzyt yx z_infty z_infty z} \\
\eqref{xsytxy=xsytyx} & \pageref{xsytxy=xsytyx}\\
\hline
\end{tabular}
\end{center}
\label{numbers}
\end{table}

\section*{Acknowledgements}
The authors thank Edmond W. H. Lee for careful reading the previous version and his valuable comments and suggestions that contributed to a significant improvement of the manuscript.
The authors are also grateful the anonymous referee for the suggestion to add tables of identities.

\end{document}